\documentclass[a4paper,reqno]{amsart}
\usepackage{indentfirst}
\usepackage{amsfonts}
\usepackage{amscd}
\usepackage{amsthm}
\usepackage{amssymb}
\usepackage{amsmath}
\usepackage{thmtools}
\usepackage{epsfig}
\usepackage{graphicx}
\usepackage{slashed}
\usepackage[all,cmtip]{xy}
\usepackage{tikz}
\usepackage{tikz-cd}
\usepackage[disable]{todonotes}
\usetikzlibrary{calc,arrows}
\usetikzlibrary{decorations.markings,shapes.geometric}
\usepackage{enumitem}
\usepackage{setspace}
\usepackage{hyperref}
\usetikzlibrary{matrix,arrows,positioning}

\def\itemNum$#1${\item $\displaystyle#1$
   \hfill\refstepcounter{equation}(\theequation)}

\newtheorem{Lem}{Lemma}[section]
\newtheorem{Prop}[Lem]{Proposition}

\newtheorem*{Def}{Definition}

\theoremstyle{plain}
\newtheorem{Thm}[Lem]{Theorem}

\newtheorem*{Assu}{Reduction Assumption}

\theoremstyle{definition}
\declaretheorem[numbered=no,name=Example,qed={\lower-0.3ex\hbox{$\triangleleft$}}]{Ex}
\newtheorem*{Rem}{Remark}
\newtheorem*{Rems}{Remarks}

\newcommand{\Hom}{\text{\textnormal{Hom}}}
\newcommand{\Aut}{\text{\textnormal{Aut}}}
\newcommand{\Ext}{\text{\textnormal{Ext}}}

\newcommand{\rk}{\text{\textnormal{rk}}}

\newcommand{\map}{\text{\textnormal{map}}}
\newcommand{\pt}{\text{\textnormal{pt}}}

\mathchardef\mhyphen="2D
\newcommand{\git}{/\!\!/}
\newcommand{\revgit}{\backslash\!\!\backslash}
\newcommand{\id}{\text{\textnormal{id}}}
\def\ho{\operatorname{ho}\!}
\def\hopro{\ho \underleftarrow{\lim}{}}
\newcommand{\Seg}{\mathcal{S}\mathsf{eg}}

\newcommand{\hgline}[2]{
\pgfmathsetmacro{\thetaone}{#1}
\pgfmathsetmacro{\thetatwo}{#2}
\pgfmathsetmacro{\theta}{(\thetaone+\thetatwo)/2}
\pgfmathsetmacro{\phi}{abs(\thetaone-\thetatwo)/2}
\pgfmathsetmacro{\close}{less(abs(\phi-90),0.0001)}
\ifdim \close pt = 1pt
    \draw[red] (\theta+180:1) -- (\theta:1);
\else
    \pgfmathsetmacro{\R}{tan(\phi)}
    \pgfmathsetmacro{\distance}{sqrt(1+\R^2)}
    \draw[red] (\theta:\distance) circle (\R);
\fi
}

\begin{document}

\title{Relative $2$-Segal spaces}

\author[M.\,B. Young]{Matthew B. Young}
\address{The Institute of Mathematical Sciences and Department of Mathematics\\
The Chinese University of Hong Kong\\
Shatin, Hong Kong}
\email{myoung@ims.cuhk.hk.edu}

\date{\today}


\begin{abstract}
We introduce a relative version of the $2$-Segal simplicial spaces defined by Dyckerhoff and Kapranov and G\'{a}lvez-Carrillo, Kock and Tonks. Examples of relative $2$-Segal spaces include the categorified unoriented cyclic nerve, real pseudoholomorphic polygons in almost complex manifolds and the $\mathcal{R}_{\bullet}$-construction from Grothendieck-Witt theory. We show that a relative $2$-Segal space defines a categorical representation of the Hall algebra associated to the base $2$-Segal space. In this way, after decategorification we recover a number of known constructions of Hall algebra representations. We also describe some higher categorical interpretations of relative $2$-Segal spaces.
\end{abstract}

\maketitle

\tableofcontents

\setcounter{footnote}{0}

\section*{Introduction}
Motivated by Segal's notion of a $\Gamma$-space \cite{segal1974}, Rezk introduced Segal spaces in his study of the homotopy theory of $(\infty,1)$-categories \cite{rezk2001}. Generalizing these ideas, for each integer $k \geq 1$, Dyckerhoff and Kapranov introduced $k$-Segal spaces \cite{dyckerhoff2012b}. Very roughly, a simplicial topological space $X_{\bullet}$ is called $k$-Segal if it satisfies a collection of locality conditions governed by polyhedral subdivisions of $k$-dimensional cyclic polytopes. When $k=1$, so that the locality conditions are governed by subdivisions of the interval, the $1$-Segal conditions state that, for each $n \geq 2$, the canonical map to the homotopy fibre product
\[
X_n \rightarrow \overbrace{X_1 \times^R_{X_0} \cdots \times^R_{X_0} X_1}^{n \; \mbox{\scriptsize factors}}
\]
is a weak homotopy equivalence. Hence  $1$-Segal spaces reduce to Rezk's Segal spaces. The $2$-Segal spaces, which were introduced independently by G\'{a}lvez-Carrillo, Kock and Tonks \cite{galvez2015} under the name decomposition spaces, obey locality conditions governed by subdivisions of convex plane polygons. The first non-trivial conditions derive from the two triangulations of the square,
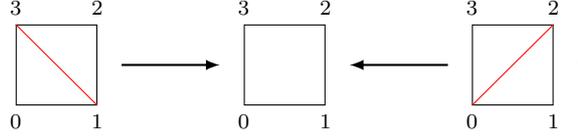
\[
\begin{tikzpicture}[baseline= (a).base]
\node (origin) at (0,0) (origin) {$\begin{tikzpicture}
\node[draw,minimum size=1.5cm,regular polygon,regular polygon sides=4] (a) {};
\draw[] (a.corner 1) node[above]{\footnotesize $2$} ;
\draw[] (a.corner 2) node[above]{\footnotesize $3$} ;
\draw[] (a.corner 3) node[below]{\footnotesize $0$} ;
\draw[] (a.corner 4) node[below]{\footnotesize $1$} ; 
\end{tikzpicture}$}; 

\node (P0) at (0:3cm) { $
\begin{tikzpicture}
\node[draw,minimum size=1.5cm,regular polygon,regular polygon sides=4] (a) {};
\draw[] (a.corner 1) node[above]{\footnotesize $2$} ;
\draw[] (a.corner 2) node[above]{\footnotesize $3$} ;
\draw[] (a.corner 3) node[below]{\footnotesize $0$} ;
\draw[] (a.corner 4) node[below]{\footnotesize $1$} ; 
\path[-,font=\scriptsize,color=red]
    (a.corner 1) edge[-]  (a.corner 3); 
\end{tikzpicture}
$}; 

\node (P1) at (180:3cm) {$
\begin{tikzpicture}
\node[draw,minimum size=1.5cm,regular polygon,regular polygon sides=4] (a) {};
\draw[] (a.corner 1) node[above]{\footnotesize $2$} ;
\draw[] (a.corner 2) node[above]{\footnotesize $3$} ;
\draw[] (a.corner 3) node[below]{\footnotesize $0$} ;
\draw[] (a.corner 4) node[below]{\footnotesize $1$} ; 
\path[-,font=\scriptsize,color=red]
    (a.corner 2) edge[-]  (a.corner 4);  
\end{tikzpicture}
$};
\draw [-latex, thick] (P0) -- (origin);
\draw [-latex, thick] (P1) -- (origin); 
\end{tikzpicture},
\]
and state that the induced morphisms
\begin{equation}
\label{eq:assCond}
X_{\{0,1,3\}} \times^R_{X_{\{1,3\}}} X_{\{1,2,3\}} \leftarrow X_{\{0,1,2,3\}} \rightarrow X_{\{0,1,2\}} \times^R_{X_{\{0,2\}}} X_{\{0,2,3\}}
\end{equation}
are weak homotopy equivalences. A large number of examples of $2$-Segal spaces from a diverse range of subjects are described in \cite{dyckerhoff2012b} and \cite{galvez2015}.

One motivation to study $2$-Segal spaces is the theory of Hall algebras. Indeed, as exploited by both Dyckerhoff and Kapranov \cite{dyckerhoff2012b} and G\'{a}lvez-Carrillo, Kock and Tonks \cite{galvez2015}, the $2$-Segal conditions admit a natural interpretation as higher coherence conditions on a multiplication defined on the $1$-simplices $X_1$, the weak equivalences \eqref{eq:assCond} imposing weak associativity. From the Hall algebra point of view, the most important example of a $2$-Segal space is the Waldhausen $\mathcal{S}_{\bullet}$-construction applied to an exact category \cite{waldhausen1985}, or more generally an exact $\infty$-category \cite{dyckerhoff2012b}. Applying suitable realization functors to this $2$-Segal space recovers various familiar incarnations of the Hall algebra \cite{ringel1990}, \cite{lusztig1991}, \cite{joyce2007}, \cite{kontsevich2011}. However, these realizations use only the lowest $2$-Segal conditions, namely \eqref{eq:assCond}. Taking into account the remaining conditions leads to higher categorical structures and thus to categorical Hall algebras. Applications of $2$-Segal spaces to other areas, including combinatorics, topological field theories and Fukaya categories, can be found in \cite{dyckerhoff2013b}, \cite{dyckerhoff2015}, \cite{galvez2015}.

Partially motivated by the representation theory of Hall algebras, in this paper we introduce relative higher Segal spaces. For an integer $k \geq 1$, a relative $k$-Segal space over a $k$-Segal space $X_{\bullet}$ is a $(k-1)$-Segal space $Y_{\bullet}$ together with a morphism $Y_{\bullet} \rightarrow X_{\bullet}$ which satisfies $k$-dimensional locality conditions involving both $X_{\bullet}$ and $Y_{\bullet}$; by convention the $0$-Segal conditions are vacuous. The simplest case is that of right relative $1$-Segal spaces, the locality conditions reducing to the condition that the map $Y_{\bullet} \rightarrow X_{\bullet}$ be a right fibration of Segal spaces in the sense of \cite{kazhdan2014}, \cite{brito2016}. More interesting is the case of relative $2$-Segal spaces. The relative $2$-Segal conditions on a morphism $Y_{\bullet} \rightarrow X_{\bullet}$ are governed by reflection symmetric polyhedral subdivisions of symmetric convex plane polygons, the most basic of which are the two subdivisions of the plane hexagon,
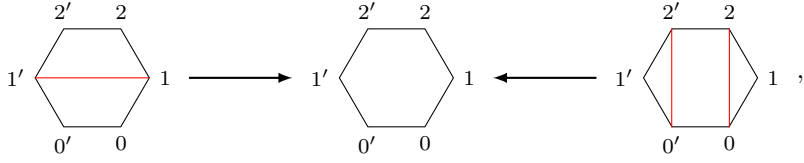
\[
\begin{tikzpicture}[baseline= (a).base]
\node (origin) at (0,0) (origin) {$\begin{tikzpicture}
\node[draw,minimum size=1.5cm,regular polygon,regular polygon sides=6] (a) {};
\draw[] (a.corner 1) node[above]{\footnotesize $2$} ;
\draw[] (a.corner 2) node[above]{\footnotesize $2^{\prime}$} ;
\draw[] (a.corner 3) node[left]{\footnotesize $1^{\prime}$} ;
\draw[] (a.corner 4) node[below]{\footnotesize $0^{\prime}$} ;
\draw[] (a.corner 5) node[below]{\footnotesize $0$} ;
\draw[] (a.corner 6) node[right]{\footnotesize $1$} ;
\end{tikzpicture}
$}; 

\node (P0) at (0:4cm) { $
\begin{tikzpicture}
\node[draw,minimum size=1.5cm,regular polygon,regular polygon sides=6] (a) {};
\draw[] (a.corner 1) node[above]{\footnotesize $2$} ;
\draw[] (a.corner 2) node[above]{\footnotesize $2^{\prime}$} ;
\draw[] (a.corner 3) node[left]{\footnotesize $1^{\prime}$} ;
\draw[] (a.corner 4) node[below]{\footnotesize $0^{\prime}$} ;
\draw[] (a.corner 5) node[below]{\footnotesize $0$} ;
\draw[] (a.corner 6) node[right]{\footnotesize $1$} ;

\path[-,font=\scriptsize,color=red]
    (a.corner 1) edge[-]  (a.corner 5);  
\path[-,font=\scriptsize,color=red]
    (a.corner 2) edge[-]  (a.corner 4);  
\end{tikzpicture}
$}; 

\node (P1) at (180:4cm) {$
\begin{tikzpicture}
\node[draw,minimum size=1.5cm,regular polygon,regular polygon sides=6] (a) {};
\draw[] (a.corner 1) node[above]{\footnotesize $2$} ;
\draw[] (a.corner 2) node[above]{\footnotesize $2^{\prime}$} ;
\draw[] (a.corner 3) node[left]{\footnotesize $1^{\prime}$} ;
\draw[] (a.corner 4) node[below]{\footnotesize $0^{\prime}$} ;
\draw[] (a.corner 5) node[below]{\footnotesize $0$} ;
\draw[] (a.corner 6) node[right]{\footnotesize $1$} ;
\path[-,font=\scriptsize,color=red]
    (a.corner 3) edge[-]  (a.corner 6);  
\end{tikzpicture}
$};
\draw [-latex, thick] (P0) -- (origin);
\draw [-latex, thick] (P1) -- (origin); 
\end{tikzpicture},
\]
and translate into the requirement that the induced morphisms
\begin{equation}
\label{eq:modAssCond}
Y_{\{0,1\}} \times^R_{Y_{\{1\}}} Y_{\{1,2\}} \leftarrow Y_{\{0,1,2\}} \rightarrow X_{\{0,1,2\}} \times^R_{X_{\{0,2\}}} Y_{\{0,2\}}
\end{equation}
be weak equivalences. Similar to the case of $2$-Segal spaces, the relative $2$-Segal conditions give higher coherence conditions for appropriately defined left and right actions of the algebra object $X_1$ on the $0$-simplices $Y_0$. From this point of view, the weak equivalences \eqref{eq:modAssCond} are the weak module associativity constraints. Relative $2$-Segal spaces therefore lead naturally to categorical representations of the Hall algebra of $X_{\bullet}$; see Theorems \ref{thm:univHallMod} and \ref{thm:hallMonModCat} for particular instances of this construction. In this way we obtain natural categorifications of many of the Hall algebra representations which have appeared in the literature. For example, we prove that a stable framed variant of the Waldhausen $\mathcal{S}_{\bullet}$-construction is relative $2$-Segal  over the ordinary $\mathcal{S}_{\bullet}$-construction, thus categorifying the Hall algebra representations studied in \cite{soibelman2016}, \cite{franzen2016}; see Theorem \ref{thm:stabFrameSegal}. We also prove that the $\mathcal{R}_{\bullet}$-construction from Grothendieck-Witt theory (i.e. higher algebraic $KR$-theory) \cite{shapiro1996}, \cite{hornbostel2004} is relative $2$-Segal over the $\mathcal{S}_{\bullet}$-construction; see Theorem \ref{thm:sd2Segal}. The input for the $\mathcal{R}_{\bullet}$-construction is a proto-exact category with duality which satisfies a reduction assumption. In the case of exact categories the $\mathcal{R}_{\bullet}$-construction categorifies the Hall algebra representations of \cite{leeuwen1991}, \cite{enomoto2009}, \cite{mbyoung2016}, \cite{mbyoung2016b} while for the proto-exact category $\mathsf{Rep}_{\mathbb{F}_1}(Q)$ of representations of a quiver over $\mathbb{F}_1$ we obtain new modules over Szczesny's combinatorial Hall algebras \cite{szczesny2012}. The latter modules will be the subject of future work.

We also give examples of relative $2$-Segal spaces which do not come from previously known Hall algebra representations. Starting from an almost complex manifold $M$ with a real structure, we construct in Theorem \ref{thm:realPseudoPoly} a relative $2$-Segal semi-simplicial set consisting of real pseudoholomorphic polygons in $M$, the base $2$-Segal set being the pseudoholomorphic polygon space of Dyckerhoff and Kapranov \cite{dyckerhoff2012b}. In Theorem \ref{thm:twNerveRel2Segal} we prove that the categorified unoriented twisted cyclic nerve of a category with endomorphism and compatible duality structure is relative $2$-Segal over the categorified twisted cyclic nerve. This example can be viewed as a homotopical incarnation of the unoriented loop space of an orbifold.

A common theme of many of the relative $2$-Segal spaces constructed in this paper is that they are in a sense unoriented. It is tempting to view these examples in the context of orientifold string theory. (The stable framed $\mathcal{S}_{\bullet}$-construction is different, being related to string theory with defects.) A general feature of the orientifold construction is that it imposes $\mathbb{Z}_2$-equivariance conditions on objects in the parent string theory, such as reduction of structure groups of vector bundles from general linear to orthogonal or symplectic groups, as in the $\mathcal{R}_{\bullet}$-construction, or replacing oriented string worldsheets with unoriented worldsheets, similar to the real pseudoholomorphic polygon and unoriented nerve constructions. In \cite[Remark 3.7.8]{dyckerhoff2012b} it is speculated that there exists a sort of mirror symmetry relating the $2$-Segal spaces arising from the $\mathcal{S}_{\bullet}$-construction with those arising from the pseudoholomorphic polygon construction. It is natural to speculate that such a mirror symmetry admits an orientifold enhancement, relating the relative $2$-Segal spaces arising from the $\mathcal{R}_{\bullet}$- and real pseudoholomorphic polygon constructions.

Finally, we describe some applications of relative higher Segal spaces to higher category theory, thus lifting some of the results of \cite{rezk2001}, \cite{joyal2007}, \cite{dyckerhoff2012b}. It is a classical fact that $1$-Segal simplicial sets can be characterized as the essential image of the fully faithful nerve functor $N_{\bullet}: \mathsf{Cat} \rightarrow \mathsf{Set}_{\Delta}$. In a similar vein, right relative $1$-Segal simplicial sets are the essential image of the relative nerve construction applied to the category of discrete right fibrations which, via the Grothendieck construction, can be interpreted as presheaves on small categories; see Proposition \ref{prop:relNerve}. Using the work of several authors \cite{joyalPrep}, \cite{lurie2009b}, \cite{brito2016}, we explain a quasicategorical generalization of these statements by considering instead right relative $1$-Segal combinatorial simplicial spaces and $(\infty,1)$-presheaves. Secondly, in \cite{dyckerhoff2012b} it is proved that the category of $2$-Segal simplicial sets is equivalent to both the category of multivalued categories and to the category of $\sqcup$-semisimple semibicategories. Pursuing an interpretation in terms of actions of categories as in the relative $1$-Segal case, we lift of these statements to the relative setting, establishing equivalences of the category of relative $2$-Segal simplicial sets with both the category of modules over multivalued categories and with the category of $\mathsf{Cat}^{\sqcup}$-valued presheaves on $\sqcup$-semisimple semibicategories; see Theorems \ref{thm:multiCatModule} and \ref{thm:psh2Cat}, respectively.

\begin{Rem}
After the first version of this paper was completed, a preprint by Tashi Walde \cite{walde2016} was posted to the arXiv which also aims at developing a theory of modules over higher Segal spaces. We comment where appropriate on the overlap.
\end{Rem}

\subsubsection*{Acknowledgements}
The author would like to thank Ajneet Dhillon, Tobias Dyckerhoff, Joachim Kock, Karol Szumi\l{}o and Pal Zsamboki for interesting discussions regarding the subject of this paper. The author was supported by the Direct Grants and Research Fellowship Scheme from the Chinese University of Hong Kong.

\section{Higher Segal spaces}

In this section we recall, following closely \cite{dyckerhoff2012b}, some required background material from the theory of  higher Segal spaces.

\subsection{Simplicial objects}

Let $\Delta$ be the category whose objects are the non-empty finite ordinals $[n] = \{0 < \cdots < n\}$, $n \geq 0$, and whose morphisms are weakly monotone set maps. Let also $\Delta_{\mathsf{aug}}$ be the category of all finite non-empty ordinals, of which $\Delta$ is a skeleton. Denote by $\Delta_{\mathsf{inj}} \subset \Delta$ the subcategory of injective morphisms. We sometimes consider the object $[n] \in \Delta$ as a category itself. Explicitly, the objects of $[n]$ are labelled by integers $0 \leq i \leq n$ and the morphism set $\Hom_{[n]}(i,j)$ is empty if $i > j$ and consists of a single element if $i \leq j$.

A simplicial object of a category $\mathcal{C}$ is a functor $X_{\bullet}: \Delta^{\mathsf{op}} \rightarrow \mathcal{C}$. We write $X_n$ for $X_{[n]} \in \mathcal{C}$ if it will not lead to confusion. The face and degeneracy maps of $X_{\bullet}$ are denoted by
\[
\partial_i : X_n \rightarrow X_{n-1}, \qquad s_i : X_n \rightarrow X_{n+1}, \qquad 0 \leq i \leq n.
\]
More generally, a functor $X_{\bullet}: \Delta^{\mathsf{op}}_{\mathsf{inj}} \rightarrow \mathcal{C}$ is a semi-simplicial object of $\mathcal{C}$. A simplicial object $X_{\bullet}$ admits a canonical extension to a functor $\Delta_{\mathsf{aug}}^{\mathsf{op}} \rightarrow \mathcal{C}$, which we continue to denote by $X_{\bullet}$. For $I \in \Delta_{\mathsf{aug}}$, we write $\Delta^I$ for the simplicial set $\Hom_{\Delta}(-, I)$. In particular, $\Delta^n = \Delta^{[n]}$.

Given categories $\mathcal{C}$ and $\mathcal{D}$, with $\mathcal{C}$ small, denote by $[\mathcal{C}, \mathcal{D}]$ or $\mathcal{D}^{\mathcal{C}}$ the category of functors $\mathcal{C} \rightarrow \mathcal{D}$. Let $\mathsf{Set}$, $\mathsf{Grpd}$ and $\mathsf{Top}$ be the categories of sets, small groupoids and compactly generated topological spaces, respectively. Objects of the categories $[\Delta^{\mathsf{op}}, \mathsf{Top}]$ and $[\Delta^{\mathsf{op}}, \mathsf{Grpd}]$ are called simplicial spaces and groupoids while objects of $\mathbb{S} = [\Delta^{\mathsf{op}}, \mathsf{Set}]$ and $\mathbb{S}_{\Delta}=[\Delta^{\mathsf{op}}, \mathbb{S}]$ are called simplicial sets and combinatorial simplicial spaces, respectively.

\subsection{\texorpdfstring{$1$}{}-Segal spaces}
\label{sec:1SegSpace}

Segal spaces (called $1$-Segal spaces below) were introduced by Rezk \cite[\S 4]{rezk2001}; see also \cite[\S 1]{segal1974}. The definition below is slightly different, omitting a fibrancy condition. Write $- \times_U^R-$ for the homotopy fibre product over a topological space $U$.

\begin{Def}
A semi-simplicial space $X_{\bullet}: \Delta_{\mathsf{inj}}^{\mathsf{op}} \rightarrow \mathsf{Top}$ is called $1$-Segal if for every $n \geq 2$ the map
\[
X_n \rightarrow X_1 \times_{X_0}^R \cdots  \times_{X_0}^R X_1
\]
induced by the inclusions $\{ i, i+1\} \hookrightarrow [n]$, $i=0, \dots, n-1$, is a weak homotopy equivalence.
\end{Def}

It is straightforward to verify that a semi-simplicial space $X_{\bullet}$ is $1$-Segal if and only if one of the following two conditions hold:
\begin{enumerate}[label=(\roman*)]
\item For every $n \geq 2$ and every $0 \leq i_1 < \cdots < i_l \leq n$ the map
\[
X_n \rightarrow X_{i_1} \times_{X_{\{i_1\}}}^R \cdots  \times_{X_{\{i_l\}}}^R X_{n- i_l}
\]
induced by the inclusions $\{ 0,\dots, i_1\}, \cdots, \{i_l, \dots, n\} \hookrightarrow [n]$ is a weak equivalence.

\item For every $n \geq 2$ and every $0 \leq i \leq n$ the map
\[
X_n \rightarrow X_{\{0, \dots, i\}} \times_{X_{\{i\}}}^R X_{\{i, \dots, n\}}
\]
induced by the inclusions $\{0, \dots, i\}, \{i, \dots, n\} \hookrightarrow [n]$ is a weak equivalence.
\end{enumerate}

With only minor changes one can formulate the theory of $1$-Segal objects (along with their higher and relative variants defined below) of a combinatorial model category, in which case $-\times^R-$ becomes a homotopy limit; see \cite[\S 5]{dyckerhoff2012b}. In particular, we can speak of $1$-Segal simplicial sets, groupoids or combinatorial simplicial spaces. For simplicity we will state results in terms of simplicial spaces.

\begin{Ex}
The nerve $N_{\bullet} (\mathcal{C})$ of a small category $\mathcal{C}$ is the simplicial set which assigns to $[n] \in \Delta$ the set underlying the category $\mathcal{C}^{[n]}$. It is well-known that $N_{\bullet} (\mathcal{C})$ is a $1$-Segal simplicial set. In fact, the nerve functor $N_{\bullet}: \mathsf{Cat} \rightarrow \mathbb{S}$ is fully faithful with essential image the $1$-Segal simplicial sets. The category $\mathcal{X}$ associated to a $1$-Segal simplicial set $X_{\bullet}$ has objects $\text{Ob}(\mathcal{X})=X_0$ and morphisms
\[
\Hom_{\mathcal{X}}(x_0, x_1) = \{x_0\} \times_{X_{\{0\}}} X_{\{0,1\}} \times_{X_{\{ 1\}}} \{x_1\}.
\]
Composition of morphisms is defined using the lowest $1$-Segal conditions while associativity follows from the higher $1$-Segal conditions.
\end{Ex}

\begin{Ex}
The categorified nerve $\mathcal{N}_{\bullet}(\mathcal{C})$ of a small category  $\mathcal{C}$ is the $1$-Segal simplicial groupoid which assigns to $[n] \in \Delta$ the maximal groupoid of the category $\mathcal{C}^{[n]}$ \cite[\S 3.5]{rezk2001}. Passing to classifying spaces gives a $1$-Segal simplicial space $B\mathcal{N}_{\bullet}(\mathcal{C})$. In Rezk's framework the categorified nerve is preferred to the ordinary nerve as the former is a complete $1$-Segal space.
\end{Ex}

\subsection{\texorpdfstring{$2$}{}-Segal spaces}
\label{sec:2Segal}

The $1$-Segal spaces are the first in an infinite tower of higher Segal spaces introduced in \cite{dyckerhoff2012b}. In this section we focus on $2$-Segal spaces, the next step in this tower. See \cite{galvez2015} for a second approach to (unital) $2$-Segal spaces.

For each integer $n \geq 2$ let $P_n \subset \mathbb{R}^2$ be a convex $(n+1)$-gon with a total order on its vertices which is consistent with the counterclockwise orientation of $\mathbb{R}^2$. The total order induces a canonical identification of the set of vertices of $P_n$ with $[n]$. Let $\mathcal{P}$ be a polyhedral subdivision of $P_n$. Associating to each polygon of $\mathcal{P}$ its set of vertices defines a collection of subsets of $[n]$ and hence a simplicial subset $\Delta^{\mathcal{P}} \subset \Delta^n$. Let $X_{\bullet}$ be a semi-simplicial space. The polyhedral subdivision $\mathcal{P}$ induces a map
\[
f_{\mathcal{P}}: X_n \simeq (\Delta^n, X_{\bullet})_R \rightarrow (\Delta^{\mathcal{P}}, X_{\bullet})_R
\]
where, following \cite[\S 2.2]{dyckerhoff2012b}, for a semi-simplicial set $D$ the derived space of $D$-membranes of $X_{\bullet}$ is defined to be
\[
(D, X_{\bullet})_R = \hopro^{\mathsf{Top}}_{\{\Delta^p \hookrightarrow D \} \in \Delta_{\mathsf{inj}} \slash D} X_p.
\]

\begin{Def}
A semi-simplicial space $X_{\bullet}$ is called $2$-Segal if for every $n \geq 3$ and every triangulation $\mathcal{T}$ of $P_n$ the map $f_{\mathcal{T}} : X_n \rightarrow (\Delta^{\mathcal{T}}, X_{\bullet})_R$ is a weak equivalence.
\end{Def}

As in the case of $1$-Segal spaces, the $2$-Segal conditions can be verified using coarser subdivisions. Indeed, it is proved in \cite[Proposition 2.3.2]{dyckerhoff2012b} that a semi-simplicial space $X_{\bullet}$ is $2$-Segal if and only if one of the following conditions holds:
\begin{enumerate}[label=(\roman*)]
\item For every $n \geq 3$ and every polyhedral subdivision $\mathcal{P}$ of $P_n$ the map $f_{\mathcal{P}} : X_n \rightarrow (\Delta^{\mathcal{P}}, X_{\bullet})_R$ is a weak equivalence.

\item For every $n \geq 3$ and every $0 \leq i < j \leq n$ the map
\begin{equation}
\label{eq:2SegalDiagram}
f_{\{i,j\}}: X_n \rightarrow X_{\{i, \dots, j\}} \times_{X_{\{i,j\}}}^R X_{\{0, \dots, i, j, \dots, n\} }
\end{equation}
induced by the inclusions $\{i, \dots, j\}, \{0, \dots, i, j, \dots n\} \hookrightarrow [n]$ is a weak equivalence.

\item For every $n \geq 3$ the map \eqref{eq:2SegalDiagram} is a weak equivalence if $i=0$ or $j=n$.
\end{enumerate}

The following definition uses degeneracy maps and so can only be formulated in the simplicial setting.

\begin{Def}
A $2$-Segal simplicial space $X_{\bullet}$ is called unital $2$-Segal if for every $n \geq 2$ and every $0 \leq i \leq n-1$ the map
\[
\partial_{\{i\}} \times s_i : X_{n-1} \rightarrow X_{\{i\}} \times^R_{X_{\{i,i+1\}}} X_n
\]
is a weak equivalence.
\end{Def}

One simple construction of $2$-Segal spaces is the following.

\begin{Prop}[{\cite[Propositions 2.3.3, 2.5.3]{dyckerhoff2012b}, \cite[Proposition 3.5]{galvez2015}}]
\label{prop:2SegFrom1Seg}
Let $X_{\bullet}$ be a $1$-Segal semi-simplicial space. Then $X_{\bullet}$ is $2$-Segal. If in fact $X_{\bullet}$ is a simplicial space, then $X_{\bullet}$ is unital $2$-Segal.
\end{Prop}

\subsection{The Waldhausen \texorpdfstring{$\mathcal{S}_{\bullet}$}{}-construction}
\label{sec:waldhausen}

We recall a motivating example of a unital $2$-Segal space. We will work with proto-exact categories, a not necessarily additive generalization of exact categories in the sense of Quillen \cite{quillen1973}.

\begin{Def}[{\cite[\S 2.4]{dyckerhoff2012b}}]
A proto-exact category is a pointed category $\mathcal{C}$, with zero object $0$, together with two classes of morphisms, $\mathfrak{I}$ and $\mathfrak{D}$, called inflations and deflations and denoted by $\rightarrowtail$ and $\twoheadrightarrow$, respectively, which have the following properties:
\begin{enumerate}[label=(\roman*)]
\item Any morphism $0 \rightarrow U$ is in $\mathfrak{I}$ and any morphism $U \rightarrow 0$ is in $\mathfrak{D}$.

\item The classes $\mathfrak{I}$ and $\mathfrak{D}$ are closed under composition and contain all isomorphisms.

\item A commutative square of the form
\begin{equation}
\label{eq:biCartDiag}
\begin{tikzpicture}[baseline= (a).base]
\node[scale=1] (a) at (0,0){
\begin{tikzcd}
U \arrow[two heads]{d} \arrow[tail]{r} & V \arrow[two heads]{d}\\
W \arrow[tail]{r} & X
\end{tikzcd}
};
\end{tikzpicture}
\end{equation}
is Cartesian if and only if it is coCartesian.

\item Any diagram $W \rightarrowtail X \twoheadleftarrow V$ can be completed to a biCartesian diagram of the form \eqref{eq:biCartDiag}.

\item Any diagram $W \twoheadleftarrow U \rightarrowtail V$ can be completed to a biCartesian diagram of the form \eqref{eq:biCartDiag}.
\end{enumerate}
\end{Def}

BiCartesian squares of the form
\[
\begin{tikzpicture}[baseline= (a).base]
\node[scale=1] (a) at (0,0){
\begin{tikzcd}
U \arrow[two heads]{d} \arrow[tail]{r} & V \arrow[two heads]{d}\\
0 \arrow[tail]{r} & X
\end{tikzcd}
};
\end{tikzpicture}
\]
are called conflations and play the role of short exact sequences in $\mathcal{C}$. Familiar examples of proto-exact categories include abelian and, more generally, exact categories. A more exotic example is given by the category of representations of a quiver over $\mathbb{F}_1$, as described in \cite{szczesny2012}.

The Waldhausen $\mathcal{S}_{\bullet}$-construction associates to a proto-exact category $\mathcal{C}$ a simplicial groupoid $\mathcal{S}_{\bullet}(\mathcal{C})$ as follows \cite[\S 1.3]{waldhausen1985}, \cite[\S 2.4]{dyckerhoff2012b}. Let $\mathsf{Ar}_n =[[1],[n]]$ be the arrow category of $[n]$. The assignment $[n] \mapsto \mathsf{Ar}_n$ defines a cosimplicial category. An object $\{(i \rightarrow j) \mapsto A_{\{i,j\}}\}_{0 \leq i \leq j \leq n}$ of the functor category $[\mathsf{Ar}_n, \mathcal{C}]$ is a commutative diagram in $\mathcal{C}$ of the form
\[
\begin{tikzpicture}[baseline= (a).base]
\node[scale=1] (a) at (0,0){
\begin{tikzcd}
A_{\{0,0\}} \arrow{r} & A_{\{0,1\}} \arrow{d} \arrow{r} & \cdots \arrow{r} & A_{\{0,n-1\}} \arrow{d}  \arrow{r} & A_{\{0,n\}} \arrow{d} \\
 & A_{\{1,1\}} \arrow{r} & \cdots   \arrow{r} & A_{\{1,n-1\}} \arrow{d}  \arrow{r} & A_{\{1,n\}} \arrow{d} \\
 &  & \ddots & \vdots \arrow{d} & \vdots \arrow{d} \\
  &  & & A_{\{n-1,n-1\}}  \arrow{r} & A_{\{n-1,n\}} \arrow{d} \\
  &  &  &  & A_{\{n,n\}}.
\end{tikzcd}
};
\end{tikzpicture}
\]
Let $\mathcal{W}_n(\mathcal{C}) \subset [\mathsf{Ar}_n, \mathcal{C}]$ be the full subcategory consisting of diagrams which have the following properties:
\begin{enumerate}[label=(\roman*)]
\item For each $0 \leq i \leq n$ the object $A_{\{i,i\}}$ is isomorphic to $0 \in \mathcal{C}$.

\item All horizontal morphisms are inflations and all vertical morphisms are deflations.

\item Each square that can be formed in the diagram is biCartesian.
\end{enumerate}
Let $\mathcal{S}_n(\mathcal{C})$ be the maximal groupoid of $\mathcal{W}_n(\mathcal{C})$. Then $\mathcal{S}_{\bullet}(\mathcal{C})$ is a simplicial groupoid, the degeneracy map $s_i: \mathcal{S}_n(\mathcal{C}) \rightarrow \mathcal{S}_{n+1}(\mathcal{C})$ inserting a row/column of identity morphisms after the $i$th row/column and the face map $\partial_i: \mathcal{S}_n(\mathcal{C}) \rightarrow \mathcal{S}_{n-1}(\mathcal{C})$ deleting the $i$th row/column and composing the obvious morphisms.

\begin{Thm}[{\cite[Proposition 2.4.8]{dyckerhoff2012b}, \cite[Theorem 10.14]{galvez2015}}]
\label{thm:waldSegal}
For any proto-exact category $\mathcal{C}$, the simplicial groupoid $\mathcal{S}_{\bullet}(\mathcal{C})$ is unital $2$-Segal.
\end{Thm}

When $\mathcal{C}$ is an exact category the simplicial space $B\mathcal{S}_{\bullet}(\mathcal{C})$ plays a fundamental role in the higher algebraic $K$-theory of $\mathcal{C}$. Indeed, we have $K_i(\mathcal{C}) = \pi_i \Omega \vert B\mathcal{S}_{\bullet}(\mathcal{C}) \vert$, $i \geq 0$, where the basepoint of $\vert B\mathcal{S}_{\bullet}(\mathcal{C}) \vert$ is taken to be $0 \in \mathcal{C}$. See \cite{thomason1990}, \cite{waldhausen1985}.

\begin{Rem}
A variation of the Waldhausen $\mathcal{S}_{\bullet}$-construction was defined in \cite{bergner2016}, giving a functor from the category of augmented stable double categories to the category of simplicial sets. It was proved that this functor is fully faithful with essential image the unital $2$-Segal simplicial sets.
\end{Rem}

\section{Relative higher Segal spaces}

\subsection{Relative \texorpdfstring{$1$}{}-Segal spaces}
\label{sec:rel1SegSpaces}

Before introducing and studying relative  $2$-Segal spaces, which will be the main objects of interest in this paper, we study the more basic relative $1$-Segal spaces.

\begin{Def}
Let $X_{\bullet}$ be a $1$-Segal semi-simplicial space. A morphism $F_{\bullet}: Y_{\bullet} \rightarrow X_{\bullet}$ of semi-simplicial spaces is called right relative $1$-Segal if for every $n \geq 1$ and every $0 \leq i \leq n$ the outside square of the diagram
\begin{equation}
\label{eq:rel1SegalDiagram}
\begin{tikzpicture}[baseline= (a).base]
\node[scale=1] (a) at (0,0){
\begin{tikzcd}
Y_n \arrow{r} \arrow{d} & Y_{\{i, \dots,n\}} \arrow{d}  \\
Y_{\{0, \dots, i\}} \arrow{r} \arrow{d}[left]{F_{\{0, \dots,i\}}} & Y_{\{i\}} \arrow{d}[right]{F_{\{i\}}} \\
X_{\{0, \dots, i\}} \arrow{r} & X_{\{i\}}
\end{tikzcd}
};
\end{tikzpicture}
\end{equation}
is homotopy Cartesian.
\end{Def}

Similarly, a left relative $1$-Segal space is a morphism $F_{\bullet}: Y_{\bullet} \rightarrow X_{\bullet}$, with $X_{\bullet}$ being $1$-Segal, for which the outside square of all diagrams of the form
\[
\begin{tikzpicture}[baseline= (a).base]
\node[scale=1] (a) at (0,0){
\begin{tikzcd}
Y_n \arrow{r} \arrow{d} & Y_{\{i, \dots,n\}} \arrow{d} \arrow{r}[above]{F_{\{i, \dots, n\}}} & [1.5em] X_{\{i, \dots, n\}} \arrow{d} \\
Y_{\{0, \dots, i\}} \arrow{r} \arrow{r} & Y_{\{i\}} \arrow{r}[below]{F_{\{i\}}} & X_{\{i\}}
\end{tikzcd}
};
\end{tikzpicture}
\]
is homotopy Cartesian. All results below will be formulated for right relative $1$-Segal spaces; analogous results hold for left relative $1$-Segal spaces.

\begin{Ex}
Let $X_{\bullet}$ be a $1$-Segal semi-simplicial space. Then the identity morphism $\mathbf{1}_{X_{\bullet}}: X_{\bullet} \rightarrow X_{\bullet}$ is both left and right relative $1$-Segal.
\end{Ex}

\begin{Ex}
Let $x_*$ be an object of a small category $\mathcal{C}$ and let $\mathcal{C}_{\slash x_*}$ be the corresponding overcategory. The forgetful functor $\mathcal{C}_{\slash x_*} \rightarrow \mathcal{C}$ induces a simplicial morphism $N_{\bullet}(\mathcal{C}_{\slash x_*}) \rightarrow N_{\bullet}(\mathcal{C})$ which is right relative $1$-Segal. Using instead the undercategory $_{x_* \slash}\mathcal{C}$ we obtain a left relative $1$-Segal simplicial set $N_{\bullet}(_{x_* \slash}\mathcal{C}) \rightarrow N_{\bullet}(\mathcal{C})$.
\end{Ex}

\begin{Ex}
Suppose that a group $\mathsf{G}$ acts on a set $E$. Then $\mathsf{G}$ acts diagonally on the Cartesian product $E^n$, $n \geq 1$. The action groupoid $\mathsf{G} \revgit E^{n+1}$ is the category with objects $E^{n+1}$ and morphisms $\Hom_{\mathsf{G} \revgit E^{n+1}}(e_{\bullet}, e^{\prime}_{\bullet}) = \{g \in \mathsf{G} \mid g \cdot e_{\bullet} = e^{\prime}_{\bullet} \}$. The assignment $[n] \mapsto \mathsf{G} \revgit E^{n+1}$ defines a $1$-Segal simplicial groupoid $\mathcal{S}_{\bullet}(\mathsf{G},E)$, the face (resp. degeneracy) maps omitting (resp. repeating) the appropriate entries of $E^{\bullet+1}$. Passing to classifying spaces yields a $1$-Segal space $B\mathcal{S}_{\bullet}(\mathsf{G},E)$, called the Hecke-Waldhausen space of $(\mathsf{G},E)$ \cite[\S 2.6]{dyckerhoff2012b}.

Let $\mathsf{H} \leq \mathsf{G}$ be a subgroup. The inclusion $\mathsf{H} \hookrightarrow \mathsf{G}$ defines a simplicial morphism $\mathcal{S}_{\bullet}(\mathsf{H},E) \rightarrow \mathcal{S}_{\bullet}(\mathsf{G},E)$. At the level of classifying spaces we obtain a morphism $B\mathcal{S}_{\bullet}(\mathsf{H},E) \rightarrow B\mathcal{S}_{\bullet}(\mathsf{G},E)$ which is both left and right relative $1$-Segal; this can be verified in much the same way as the $1$-Segal property of $B\mathcal{S}_{\bullet}(\mathsf{G},E)$. At the level of geometric realizations this map is homotopy equivalent to the induced morphism $B \mathsf{H} \rightarrow B \mathsf{G}$ of classifying spaces (\textit{cf}. \cite[Proposition 2.6.7]{dyckerhoff2012b}).
\end{Ex}

We have the following alternative characterization of right relative $1$-Segal spaces.

\begin{Prop}
\label{prop:rel1Segal}
A semi-simplicial morphism $F_{\bullet}: Y_{\bullet} \rightarrow X_{\bullet}$, with $X_{\bullet}$ being $1$-Segal, is right relative $1$-Segal if and only if $Y_{\bullet}$ is $1$-Segal and the map
\begin{equation}
\label{eq:lowRel1Seg}
(F_1, \partial_0): Y_1 \rightarrow X_1 \times^R_{X_0} Y_0
\end{equation}
is a weak equivalence.
\end{Prop}

\begin{proof}
Suppose that $F_{\bullet}$ is right relative $1$-Segal. Taking $i=n=1$ in diagram \eqref{eq:rel1SegalDiagram} implies that the map \eqref{eq:lowRel1Seg} is a weak equivalence. For arbitrary $i$ and $n$, both the outside square and the bottom square (which is a degenerate version of the outside square) of diagram \eqref{eq:rel1SegalDiagram} are homotopy Cartesian. By the $2$-out-of-$3$ property of weak equivalences the top square is therefore also homotopy Cartesian. Hence $Y_{\bullet}$ is $1$-Segal.

Conversely, suppose that $Y_{\bullet}$ is $1$-Segal and that the map \eqref{eq:lowRel1Seg} is a weak equivalence. Then we have the following sequence of weak equivalences:
\begin{eqnarray*}
Y_n & \xrightarrow[]{\mbox{\tiny w.e.}} & Y_{\{0, \dots, i\}} \times^R_{Y_{\{i\}}} Y_{\{i, \dots, n\}} \\
& \xrightarrow[]{\mbox{\tiny w.e.}} & Y_{\{0, 1\}} \times^R_{Y_{\{1\}}} Y_{\{1, \dots, i\}} \times^R_{Y_{\{i\}}} Y_{\{i, \dots, n\}} \\
& \xrightarrow[]{\mbox{\tiny w.e.}} & X_{\{0, 1\}} \times^R_{X_{\{1\}}} Y_{\{1\}} \times^R_{Y_{\{1\}}} Y_{\{1, \dots, i\}} \times^R_{Y_{\{i\}}} Y_{\{i, \dots, n\}} \\
& \vdots & \\
& \xrightarrow[]{\mbox{\tiny w.e.}} & X_{\{0, 1\}} \times^R_{X_{\{1\}}} \cdots \times^R_{X_{\{i-1\}}} X_{\{i-1,i\}} \times^R_{X_{\{i\}}} Y_{\{i, \dots, n\}} \\
& \xleftarrow[]{\mbox{\tiny w.e.}} & X_{\{0, \dots, i\}} \times^R_{X_{\{i\}}} Y_{\{i, \dots, n\}}.
\end{eqnarray*}
The map $Y_n \rightarrow X_{\{0, \dots, i\}} \times^R_{X_{\{i\}}} Y_{\{i, \dots, n\}}$ makes this chain of maps commute and is thus a weak equivalence. Hence $F_{\bullet}$ is right relative $1$-Segal. 
\end{proof}

Kazhdan and Varshavsky \cite{kazhdan2014} and de Brito \cite{brito2016} define a right Segal fibration to be a morphism of $1$-Segal spaces $F_{\bullet}: Y_{\bullet} \rightarrow X_{\bullet}$ for which the map \eqref{eq:lowRel1Seg} is a weak equivalence. It follows from Proposition \ref{prop:rel1Segal} that right relative $1$-Segal spaces and right Segal fibrations are the same objects. Using \cite[Proposition 1.10]{brito2016} we obtain a model category theoretic interpretation of the relative $1$-Segal conditions. Namely, right relative $1$-Segal objects in $\mathbb{S}_{\Delta}$ are the fibrant objects of a natural left Bousfield localization of $(\Seg_1)_{\slash X_{\bullet}}$, the overcategory model structure on $(\mathbb{S}_{\Delta})_{\slash X_{\bullet}}$ induced by Rezk's $1$-Segal model structure $\Seg_1$ on $\mathbb{S}_{\Delta}$ \cite[Theorem 7.1]{rezk2001}. 

Right relative $1$-Segal simplicial sets admit a simple nerve theoretic characterization, analogous to that of $1$-Segal simplicial sets; in Section \ref{sec:multiCat} we will prove a similar result in the $2$-Segal setting. To formulate this, let $\mathsf{DRFib} \subset \mathsf{Cat}^{[1]}$ be the full subcategory of discrete right fibrations and let $1 \mhyphen \mathsf{SegRel}\mathbb{S} \subset \mathbb{S}^{[1]}$ be the full subcategory of right relative $1$-Segal simplicial sets. Let also $\mathsf{Psh}$ be the category whose objects are presheaves on small categories and whose morphisms are pairs
\begin{equation}
\label{eq:pshMorph}
(\phi, \theta): (\mathcal{F}: \mathcal{C}^{\mathsf{op}} \rightarrow \mathsf{Set}) \rightarrow (\mathcal{F}^{\prime}: \mathcal{C}^{\prime \mathsf{op}} \rightarrow \mathsf{Set})
\end{equation}
consisting of a functor $\phi: \mathcal{C} \rightarrow \mathcal{C}^{\prime}$ and a natural transformation $\theta: \mathcal{F} \Rightarrow \mathcal{F}^{\prime} \circ \phi^{\mathsf{op}}$. The Grothendieck construction defines a functor $\int: \mathsf{Psh} \rightarrow \mathsf{Cat}^{[1]}$, assigning to the presheaf $\mathcal{F}$ the functor $F: \int_{\mathcal{C}} \mathcal{F} \rightarrow \mathcal{C}$ and assigning to a morphism \eqref{eq:pshMorph} the diagram
\[
\begin{tikzcd}
\int_{\mathcal{C}} \mathcal{F} \arrow{r}{\tilde{\theta}} \arrow{d}[left]{F} & \int_{\mathcal{C}^{\prime}} \mathcal{F}^{\prime} \arrow{d}{F^{\prime}} \\ \mathcal{C} \arrow{r}[below]{\phi}& \mathcal{C}^{\prime}.
\end{tikzcd}
\]
Recall that $\int_{\mathcal{C}} \mathcal{F}$ is the category whose objects are pairs $(c, \tilde{c})$, with $c \in \mathcal{C}$ and $\tilde{c} \in \mathcal{F}(c)$, and whose morphisms $(c_1, \tilde{c}_1) \rightarrow (c_2, \tilde{c}_2)$ are morphisms $c_1 \xrightarrow[]{x} c_2$ which satisfy $\mathcal{F}(x)(\tilde{c}_2) = \tilde{c}_1$. The functor $\tilde{\theta}$ assigns to $(c, \tilde{c}) \in \int_{\mathcal{C}} \mathcal{F}$ the object $(\phi(c), \theta_c(\tilde{c}))$ and assigns to $x: (c_1, \tilde{c}_1) \rightarrow (c_2, \tilde{c}_2)$ the morphism $\phi(x)$.

\begin{Prop}
\label{prop:relNerve}
The relative nerve functor
\[
N^{[1]}_{\bullet} : \mathsf{Cat}^{[1]} \rightarrow \mathbb{S}^{[1]}, \qquad (\mathcal{Y} \xrightarrow[]{F} \mathcal{X}) \mapsto (N_{\bullet}(\mathcal{Y}) \xrightarrow[]{N_{\bullet}(F)} N_{\bullet}(\mathcal{X}))
\]
is fully faithful and fits into the commutative diagram of functors
\[
\begin{tikzcd}
\mathsf{Psh} \arrow{r}{\int} \arrow{dr}[below]{\sim} & \mathsf{Cat}^{[1]} \arrow{r}{N^{[1]}_{\bullet}} & \mathbb{S}^{[1]} \\ {} & \mathsf{DRFib} \arrow[hookrightarrow]{u} \arrow{r}[below]{\sim}& 1 \mhyphen \mathsf{SegRel}\mathbb{S} \arrow[hookrightarrow]{u}
\end{tikzcd}
\]
with indicated equivalences. In particular, there is an equivalence of categories $1 \mhyphen \mathsf{SegRel}\mathbb{S} \simeq \mathsf{Psh}$.
\end{Prop}

\begin{proof}
That $N_{\bullet}^{[1]}$ is fully faithful is well-known; see for example \cite[Proposition 2.1]{segal1968}. Commutativity of the triangle in the above diagram and the fact that $\int$ induces an equivalence $\mathsf{Psh} \simeq \mathsf{DRFib}$ are both standard. To see that $N_{\bullet}^{[1]}$ restricts as claimed, let $(F: \mathcal{Y} \rightarrow \mathcal{X} ) \in \mathsf{DRFib}$. Then $N_{\bullet}(F) : N_{\bullet}(\mathcal{Y}) \rightarrow N_{\bullet}(\mathcal{X})$ is a morphism of $1$-Segal simplicial sets. The condition that $F$ be a discrete right fibration is precisely the condition that the map $N_1 (\mathcal{Y}) \rightarrow N_1(\mathcal{X}) \times_{N_{\{1\}}(\mathcal{X})} N_{\{1\}}(\mathcal{Y})$ be a bijection. Proposition \ref{prop:rel1Segal} therefore implies that $N_{\bullet}(F)$ is right relative $1$-Segal. The construction of a quasi-inverse $1 \mhyphen \mathsf{SegRel}\mathbb{S} \rightarrow \mathsf{DRFib}$ is similar.
\end{proof}

To end this section we explain a quasicategorical generalization of Proposition \ref{prop:relNerve}. Suppose that $X_{\bullet} \in \mathbb{S}_{\Delta}$ is a complete $1$-Segal combinatorial simplicial space. The quasicategory $\mathcal{X}$ modelled by $X_{\bullet}$ (see \cite[\S 5]{rezk2001}, \cite[\S 4]{joyal2007}) has object set the $0$-simplices of $X_0$ and has mapping spaces
\[
\map_{\mathcal{X}}(x_0, x_1) = \{x_0\} \times^R_{X_{\{0\}}} X_{\{0,1\}} \times^R_{X_{\{ 1\}}} \{x_1\}.
\]
Here $\{x\}$ is regarded as the simplicial set $\Delta^0$. The lowest $1$-Segal conditions define, up to homotopy, a composition law
\[
\map_{\mathcal{X}}(x_0, x_1) \times \map_{\mathcal{X}}(x_1, x_2) \rightarrow \map_{\mathcal{X}}(x_0, x_2)
\]
which, by the remaining $1$-Segal conditions, is coherently associative. Suppose now that we are given a right relative $1$-Segal morphism $F_{\bullet} : Y_{\bullet} \rightarrow X_{\bullet}$. For each object $x \in \mathcal{X}$ define a Kan complex by $\mathcal{F}(x) = \{x\} \times^R_{X_0} Y_0 \in \mathbb{S}$. The diagram
\[
\begin{tikzpicture}[baseline= (a).base]
\node[scale=1] (a) at (0,0){
\begin{tikzcd}[column sep=tiny]
\{x_0\} \times^R_{X_{\{ 0\}}} Y_{\{0,1\}} \arrow{d}[left]{\mbox{\tiny w.e.}} \arrow{r} & \{x_0\} \times^R_{X_{\{ 0\}}} Y_{\{0\}}   \\
\{x_0\} \times^R_{X_{\{ 0\}}} X_{\{0,1\}} \times^R_{X_{\{ 1\}}} Y_{\{1\}} & \\
(\{x_0\} \times^R_{X_{\{ 0\}}} X_{\{0,1\}} \times^R_{X_{\{ 1\}}} \{x_1\} ) \times 
(\{x_1\} \times^R_{X_{\{1\}}} Y_{\{1\}}) ,\arrow{u}
\end{tikzcd}
};
\end{tikzpicture}
\]
whose indicated arrow is a weak equivalence by the lowest right relative $1$-Segal condition, defines up to homotopy an action map
\[
\map_{\mathcal{X}}(x_0,x_1) \times \mathcal{F}(x_1) \rightarrow \mathcal{F}(x_0).
\]
The remaining right relative $1$-Segal conditions ensure that this action is coherently associative. In this way we obtain an $(\infty,1)$-presheaf on $\mathcal{X}$. Conversely, by combining \cite[Proposition 5.1.1.1]{lurie2009b} (see also \cite{joyalPrep}) and \cite[Theorem 1.22]{brito2016} we see that, up to weak equivalence, any $(\infty,1)$-presheaf on $\mathcal{X}$ arises in this way.

\subsection{Relative \texorpdfstring{$2$}{}-Segal spaces}

In this section we give a direct definition of relative $2$-Segal spaces. In Section \ref{sec:symmSubdivisions} we will describe a second approach using polyhedral subdivisions, in line with the definition of $2$-Segal spaces.

\begin{Def}
Let $X_{\bullet}$ be a $2$-Segal semi-simplicial space. A morphism $F_{\bullet}: Y_{\bullet} \rightarrow X_{\bullet}$ of semi-simplicial spaces is called relative $2$-Segal if
\begin{enumerate}
\item for every $n \geq 2$ and every $0 \leq i < j \leq n$ the outside square of the diagram
\begin{equation}
\label{eq:rel2SegalDiagram}
\begin{tikzpicture}[baseline= (a).base]
\node[scale=1] (a) at (0,0){
\begin{tikzcd}
Y_n \arrow{r} \arrow{d} & Y_{\{0, \dots, i,j, \dots,n\}} \arrow{d}  \\
Y_{\{i, \dots, j\}} \arrow{r} \arrow{d}[left]{F_{\{i, \dots, j\}}} & Y_{\{i,j\}} \arrow{d}[right]{F_{\{i,j\}}} \\
X_{\{i, \dots, j\}} \arrow{r} & X_{\{i,j\}}
\end{tikzcd}
};
\end{tikzpicture}
\end{equation}
is homotopy Cartesian, and
\item the simplicial space $Y_{\bullet}$ is $1$-Segal.
\end{enumerate}
\end{Def}

\begin{Rems}
\hspace{2em}
\begin{enumerate}
\item Morphisms $F_{\bullet}: Y_{\bullet} \rightarrow X_{\bullet}$ of simplicial spaces for which the first condition above holds play an important role in \cite{galvez2015}, where they are called unique lifting of factorization (ULF) functors. If, in addition, $F_{\bullet}$ preserves units (see below), then $F_{\bullet}$ is called a conservative ULF functor.

\item The above notion of relative $2$-Segal space coincides with that of Walde; see \cite[Proposition 3.5.10]{walde2016}.

\item Rather than considering diagrams of the form \eqref{eq:rel2SegalDiagram}, one could require that the outside square of all diagrams of the form
\[
\begin{tikzpicture}[baseline= (a).base]
\node[scale=1] (a) at (0,0){
\begin{tikzcd}
Y_n \arrow{r} \arrow{d} & Y_{\{0, \dots, i,j, \dots,n\}} \arrow{d} \arrow{r}[above]{F_{\{0, \dots, i,j, \dots,n\}}} & [3em] X_{\{0, \dots, i,j, \dots,n\}} \arrow{d} \\
Y_{\{i, \dots, j\}} \arrow{r} & Y_{\{i,j\}} \arrow{r}[below]{F_{\{i,j\}}} & X_{\{i,j\}}
\end{tikzcd}
};
\end{tikzpicture}
\]
be homotopy Cartesian. However, morphisms $F_{\bullet}$ satisfying these conditions are less interesting from our perspective since, for example, they do not lead to categorified Hall algebra representations.
\end{enumerate} 
\end{Rems}

The following result will be helpful in verifying the relative $2$-Segal conditions.

\begin{Prop}
\label{prop:equivRelSeg}
Let $F_{\bullet} : Y_{\bullet} \rightarrow X_{\bullet}$ be a morphism of semi-simplicial spaces. Assume that $X_{\bullet}$ is $2$-Segal. The following statements are equivalent:
\begin{enumerate}
\item For every $n \geq 1$ and every $0 \leq i < j \leq n$ the outside square of diagram \eqref{eq:rel2SegalDiagram} is homotopy Cartesian.

\item $Y_{\bullet}$ is $2$-Segal and for every $n \geq 1$ and the outside square of diagram \eqref{eq:rel2SegalDiagram} is homotopy Cartesian if $i=0$ or $j=n$.

\item $Y_{\bullet}$ is $2$-Segal and for every $n \geq 1$ the square
\begin{equation}
\label{eq:botRel2SegalDiagram}
\begin{tikzpicture}[baseline= (a).base]
\node[scale=1] (a) at (0,0){
\begin{tikzcd}
Y_{\{0, \dots, n\}} \arrow{r} \arrow{d}[left]{F_{\{0, \dots,n\}}} & Y_{\{0,n\}} \arrow{d}[right]{F_{\{0,n\}}} \\ X_{\{0,\dots,n\}} \arrow{r} & X_{\{0,n\}}
\end{tikzcd}
};
\end{tikzpicture}
\end{equation}
is homotopy Cartesian.
\end{enumerate}
\end{Prop}

\begin{proof}
Assume that the first condition holds. Since the bottom square of \eqref{eq:rel2SegalDiagram} is a degenerate version of the outer square, the $2$-out-of-$3$ property of weak equivalences implies that $Y_{\bullet}$ is $2$-Segal. Hence the second condition holds. It is clear that the second condition implies the third. Assume that the third condition holds. Then the bottom square of diagram \eqref{eq:rel2SegalDiagram} is homotopy Cartesian by assumption while the top square is homotopy Cartesian since $Y_{\bullet}$ is $2$-Segal. It follows that the outside square is also homotopy Cartesian, showing that the first condition holds.
\end{proof}

We will often show that a morphism $F_{\bullet}: Y_{\bullet} \rightarrow X_{\bullet}$ is relative $2$-Segal by first proving that $Y_{\bullet}$ is $1$-Segal, applying Proposition \ref{prop:2SegFrom1Seg} and then verifying the third condition of Proposition \ref{prop:equivRelSeg}.

We briefly describe a model theoretic interpretation of relative $2$-Segal spaces.
Denote by $\mathcal{R}$ the Reedy model structure on $\mathbb{S}_{\Delta}$. Following \cite[\S 5.2]{dyckerhoff2012b}, let $\Seg_2$ be the left Bousfield localization of $\mathcal{R}$ along the maps
\[
\mathfrak{S}eg_2 = \{ \Delta^{\mathcal{T}} \hookrightarrow \Delta^n \mid n \geq 3, \;\; \mathcal{T} \mbox{ is a triangulation of } P_n \}.
\]
Fibrant objects of $(\mathbb{S}_{\Delta}, \Seg_2)$ are then (Reedy fibrant) $2$-Segal combinatorial simplicial spaces. For $X_{\bullet} \in \mathbb{S}_{\Delta}$ write $(\Seg_2)_{\slash X_{\bullet}}$ for the induced model structure on the overcategory $(\mathbb{S}_{\Delta})_{\slash X_{\bullet}}$ and let $(\Seg_2)^{\mathsf{fib}}_{\slash X_{\bullet}}$ be its left Bousfield localization along
\[
\mathfrak{S}eg_2^{\mathsf{fib}} = \{ \Delta^{\{0,n\}} \hookrightarrow \Delta^n \xrightarrow[]{x_n} X_{\bullet} \mid x_n \in X_n, \; n \geq 2\}.
\]
Assuming that $X_{\bullet}$ is $2$-Segal, Proposition \ref{prop:equivRelSeg} implies that the fibrant objects of $((\mathbb{S}_{\Delta})_{\slash X_{\bullet}},(\Seg_2)^{\mathsf{fib}}_{\slash X_{\bullet}})$ are the morphisms $F_{\bullet}: Y_{\bullet} \rightarrow X_{\bullet}$ for which the outside square of the diagrams \eqref{eq:rel2SegalDiagram} is homotopy Cartesian, that is, the ULF functors. From this point of view, a relative $2$-Segal combinatorial simplicial space is a fibrant object of $((\mathbb{S}_{\Delta})_{\slash X_{\bullet}},(\Seg_2)^{\mathsf{fib}}_{\slash X_{\bullet}})$ whose total space is in addition $1$-Segal. It is important to note that the main results of this paper do not hold for arbitrary fibrant objects of $((\mathbb{S}_{\Delta})_{\slash X_{\bullet}},(\Seg_2)^{\mathsf{fib}}_{\slash X_{\bullet}})$.

We now formulate the relative analogue of the unital conditions.

\begin{Def}
A relative $2$-Segal simplicial space $F_{\bullet} : Y_{\bullet} \rightarrow X_{\bullet}$ is called unital relative $2$-Segal if for every $n \geq 2$ and every $0 \leq i \leq n-1$ the map
\[
(F_{\{i\}} \circ \partial_{\{i\}}) \times s_i :  Y_{n-1} \rightarrow X_{\{i\}} \times^R_{X_{\{i,i+1\}}} Y_n
\]
is a weak equivalence.
\end{Def}

We have the following relative analogue of Proposition \ref{prop:2SegFrom1Seg}.

\begin{Prop}
Let $F_{\bullet} : Y_{\bullet} \rightarrow X_{\bullet}$ be a right relative $1$-Segal semi-simplicial space. Then $F_{\bullet}$ is relative $2$-Segal. If, moreover, $F_{\bullet}$ is a morphism of simplicial spaces, then $F_{\bullet}$ is unital relative $2$-Segal.
\end{Prop}

\begin{proof}
Proposition \ref{prop:rel1Segal} implies that $Y_{\bullet}$ is $1$-Segal. The relative $2$-Segal morphism $Y_{\{0, \dots, n\}} \rightarrow X_{\{0, \dots, n\}} \times_{X_{\{0,n\}}}^R Y_{\{0, n\}}$ factors as the composition
\begin{eqnarray*}
Y_{\{0, \dots, n\}} & \rightarrow & X_{\{0, \dots, n\}} \times_{X_{\{n\}}}^R Y_{\{n\}} \\
& \rightarrow & X_{\{0, \dots, n\}} \times_{X_{\{0,n\}}}^R X_{\{0,n\}} \times_{X_{\{n\}}}^R  Y_{\{n\}} \\
& \rightarrow & X_{\{0, \dots, n\}} \times_{X_{\{0,n\}}}^R  Y_{\{0, n\}}.
\end{eqnarray*}
The first and third morphisms are weak equivalences by the right relative $1$-Segal conditions on $F_{\bullet}$ while the second  morphism is a weak equivalence for trivial reasons. Hence the composition is a weak equivalence. The unital condition is verified in a similar way.
\end{proof}

\begin{Ex}
For any $1$-Segal semi-simplicial space $X_{\bullet}$, the identity morphism $\mathbf{1}_{X_{\bullet}}$ is relative $2$-Segal. If $X_{\bullet}$ is $2$-Segal but not $1$-Segal, then $\mathbf{1}_{X_{\bullet}}$ is not relative $2$-Segal.
\end{Ex}

We end this section with a construction of relative $2$-Segal spaces which can be seen as the $2$-Segal analogue of the fact that the morphisms $N_{\bullet}(\mathcal{C}_{\slash x_*}) \rightarrow N_{\bullet}(\mathcal{C})$ and $N_{\bullet}(_{x_*\slash}\mathcal{C}) \rightarrow N_{\bullet}(\mathcal{C})$, defined in Section \ref{sec:rel1SegSpaces}, are right and left relative $1$-Segal, respectively. Recall the left join functor
\[
l: \Delta \rightarrow \Delta, \qquad \{0 , \dots ,  n\} \mapsto \{0^{\prime} , 0 , \dots ,  n\}.
\]
The left path space of a simplicial space $X_{\bullet}$ is the composition
\[
P^{\lhd}X_{\bullet} : \Delta^{\mathsf{op}} \xrightarrow[]{l^{\mathsf{op}}} \Delta^{\mathsf{op}} \xrightarrow[]{X_{\bullet}} \mathsf{Top}. 
\]
After identifying $P^{\lhd}X_m$ with $X_{m+1}$ for each $m \geq 0$, the face map $\partial^{\lhd}_i: P^{\lhd}X_n \rightarrow P^{\lhd}X_{n-1}$ is identified with $\partial_{i+1}$. The remaining face maps $\partial_{0^{\prime}}$ assemble to a simplicial morphism $F^{\lhd}_{\bullet}: P^{\lhd}X_{\bullet} \rightarrow X_{\bullet}$. Using instead the right join functor
\[
r: \Delta \rightarrow \Delta, \qquad \{0 , \dots ,  n\} \mapsto \{0 , \dots ,  n,n^{\prime} \}
\]
we obtain the right path space $F^{\rhd}_{\bullet}: P^{\rhd}X_{\bullet} \rightarrow X_{\bullet}$.

\begin{Prop}
\label{prop:relPathSpace}
Assume that $X_{\bullet}$ is (unital) $2$-Segal. Then the left and right path spaces
\[
F^{\lhd}_{\bullet}: P^{\lhd}X_{\bullet} \rightarrow X_{\bullet}, \qquad F^{\rhd}_{\bullet}: P^{\rhd}X_{\bullet} \rightarrow X_{\bullet}
\]
are (unital) relative $2$-Segal.
\end{Prop}

\begin{proof}
Since $X_{\bullet}$ is $2$-Segal, the path space criterion \cite[Theorem 6.3.2]{dyckerhoff2012b}, \cite[Theorem 4.11]{galvez2015} implies that $P^{\lhd}X_{\bullet}$ and $P^{\rhd}X_{\bullet}$ are $1$-Segal.  That the remaining relative $2$-Segal conditions hold is also proved in \cite[Theorem 4.11]{galvez2015}. Explicitly, the relative $2$-Segal map
\[
P^{\lhd}X_n = X_{\{ 0^{\prime}, 0, \dots, n\}} \rightarrow X_n \times^R_{X_{\{0,n\}}} X_{\{0^{\prime}, 0, n\}} = X_n \times^R_{X_{\{0,n\}}} P^{\lhd}X_{\{0,n\}}
\]
and the relative unit map
\[
P^{\lhd}X_{n-1} = X_{\{ 0^{\prime}, 0, \dots, n-1\}} \rightarrow X_{\{i\}} \times^R_{X_{\{i,i+1\}}} X_{\{0^{\prime}, 0, \dots, n\}} = X_{\{i\}} \times^R_{X_{\{i,i+1\}}} P^{\lhd}X_n
\]
are $2$-Segal and unit maps for $X_{\bullet}$ and so are weak equivalences by assumption.
\end{proof}

\subsection{Symmetric polyhedral subdivisions}
\label{sec:symmSubdivisions}

In this section we use combinatorial geometry to give a uniform treatment of the two conditions which define relative $2$-Segal spaces.

For each $n \geq 0$ let $\mathbf{n}$ be the ordered set $\{ 0 < \cdots < n < n^{\prime} <  \cdots < 0^{\prime}\}$. There is a unique isomorphism $\mathbf{n} \simeq [2n+1]$ in $\Delta$. Let $P_{\mathbf{n}} \subset \mathbb{R}^2$ be a convex $(2n+2)$-gon which is symmetric with respect to reflection about the $y$-axis and which has no vertices on the $y$-axis. Identify the vertices of $P_{\mathbf{n}}$ with $\mathbf{n}$ in a way which is consistent with the counterclockwise orientation of $\mathbb{R}^2$ and so that reflection about the $y$-axis interchanges the vertices $i$ and $i^{\prime}$.

A polyhedral subdivision $\mathcal{P}$ of $P_{\mathbf{n}}$ is called symmetric if it is fixed by reflection about the $y$-axis. Explicitly, $\mathcal{P}$ is symmetric if and only if it consists of 
\begin{enumerate}[label=(\roman*)]
\item horizontal diagonals $\{i^{\prime}, i\}$, $0 < i < n$, and
\item pairs of diagonals $\{i,j\}$ and $\{i^{\prime},j^{\prime}\}$, $0 \leq i < j \leq n$,
\end{enumerate}
and these diagonals have no intersection in the interior of $P_{\mathbf{n}}$.

\begin{Def}
Let $\mathcal{P}$ be a symmetric polyhedral subdivision of $P_{\mathbf{n}}$. The category $\mathcal{S}_{\mathcal{P}}$ of symmetric simplices in $\mathcal{P}$ is defined as follows.
\begin{itemize}
\item Objects of $\mathcal{S}_{\mathcal{P}}$ are $\mathbb{Z}_2$-equivariant monomorphisms $\sigma: D \hookrightarrow \Delta^{\mathcal{P}}$ of simplicial sets of the form
\[
\Delta^k \sqcup \Delta^k \hookrightarrow \Delta^{\mathcal{P}}
\qquad \mbox{or} \qquad \Delta^{\mathbf{m}} \hookrightarrow \Delta^{\mathcal{P}},
\]
where $\mathbb{Z}_2$ acts on $\Delta^k \sqcup \Delta^k$ by interchanging the two copies of $\Delta^k$ and acts on $\Delta^{\mathbf{m}}$ by first interchanging $i$ and $i^{\prime}$ and then using the canonical isomorphism $\mathbf{m}^{\mathsf{op}} \simeq \mathbf{m}$.

\item Morphisms in $\mathcal{S}_{\mathcal{P}}$ are commutative diagrams
\[
\begin{tikzcd}
D \arrow[hookrightarrow]{rd}[below]{\sigma}  \arrow{rr}{\phi} &  & D^{\prime} \\
 &\Delta^{\mathcal{P}}. \arrow[hookleftarrow]{ru}[below]{\sigma^{\prime}} &
\end{tikzcd}
\]
\end{itemize}
\end{Def}

It follows that a morphism $\phi$ in $\mathcal{S}_{\mathcal{P}}$ is of one of the following three types:
\begin{enumerate}[label=(\roman*)]
\item $\Delta^k \sqcup \Delta^k \hookrightarrow \Delta^l \sqcup \Delta^l$ for some $k \leq l$, in which case $\phi$ is determined by its restriction $\Delta^k \hookrightarrow \Delta^l$ to a summand.

\item $\Delta^{\mathbf{m}} \hookrightarrow \Delta^{\mathbf{n}}$ for some $m \leq n$, in which case $\phi$ is determined by its restriction $\Delta^m \hookrightarrow \Delta^n$ for the canonical inclusions $[m] \hookrightarrow \mathbf{m}$ and $[n] \hookrightarrow \mathbf{n}$.

\item $\Delta^k \sqcup \Delta^k \hookrightarrow \Delta^{\mathbf{m}}$ for some $k \leq m$, in which case $\phi$ is determined by its restriction $\Delta^k \hookrightarrow \Delta^m$.
\end{enumerate}

Given a morphism of semi-simplicial spaces $F_{\bullet}: Y_{\bullet} \rightarrow X_{\bullet}$ and a symmetric polyhedral subdivision $\mathcal{P}$ of $P_{\mathbf{n}}$, define a functor $F_{\mathcal{P}}: \mathcal{S}^{\mathsf{op}}_{\mathcal{P}} \rightarrow \mathsf{Top}$ as follows. At the level of objects $F_{\mathcal{P}}$ is given by
\[
F_{\mathcal{P}} (\Delta^k \sqcup \Delta^k \hookrightarrow \Delta^{\mathcal{P}}) = X_k, \qquad F_{\mathcal{P}}(\Delta^{\mathbf{m}} \hookrightarrow \Delta^{\mathcal{P}}) = Y_m.
\]
Given a morphism $\phi$ of type (i), (ii) or (iii), the morphism $F_{\mathcal{P}}(\phi)$ is defined by the semi-simplicial structure of $X_{\bullet}$, the semi-simplicial structure of $Y_{\bullet}$ and the semi-simplicial morphism $F_{\bullet}$, respectively. The derived space of symmetric $\mathcal{P}$-membranes of $F_{\bullet}$ is then defined to be the homotopy limit
\[
(\Delta^{\mathcal{P}}, F_{\bullet})_R = \hopro^{\mathsf{Top}}_{\sigma \in \mathcal{S}_{\mathcal{P}}} F_{\mathcal{P}}(\sigma).
\]
A refinement $\mathcal{P}$ of $\mathcal{P}^{\prime}$ defines a simplicial subset $\Delta^{\mathcal{P}} \subset \Delta^{\mathcal{P}^{\prime}}$ and thus a morphism $(\Delta^{\mathcal{P}^{\prime}}, F_{\bullet})_R \rightarrow (\Delta^{\mathcal{P}}, F_{\bullet})_R$. In particular, taking $\mathcal{P}^{\prime}$ to be the trivial subdivision and noting that $(\Delta^{\mathbf{n}}, F_{\bullet})_R \simeq Y_n$, we obtain a morphism
\[
f_{\mathcal{P}}^{F_{\bullet}}: Y_n \rightarrow (\Delta^{\mathcal{P}}, F_{\bullet})_R.
\]

\begin{Ex}
Consider the symmetric polyhedral subdivisions $\mathcal{P}$ and $\mathcal{P}^{\prime}$ of $P_{\mathbf{2}}$ depicted in Figure \ref{fig:symmSubdivisions}. In these cases the associated derived spaces symmetric membranes are homotopy fibre products,
\[
(\Delta^{\mathcal{P}}, F_{\bullet})_R \simeq Y_{\{0,1\}} \times^R_{Y_{\{ 1 \}}} Y_{\{1,2\}}, \qquad (\Delta^{\mathcal{P}^{\prime}}, F_{\bullet})_R \simeq X_{\{0,1,2\}} \times^R_{X_{\{ 0,2 \}}} Y_{\{0,2\}}.
\]
For more complicated subdivisions, such as $\mathcal{P}^{\prime \prime}$ of the same figure, a limit is indeed required to define the derived space of symmetric membranes.
\end{Ex}

\begin{figure}
\begin{center}
\[
\mathcal{P}=\begin{tikzpicture}[baseline={([yshift=-.5ex]current bounding box.center)}]
\node[draw,minimum size=1.5cm,regular polygon,regular polygon sides=6] (a) {};

\draw[] (a.corner 1) node[above]{\scriptsize $2$} ; 
\draw[] (a.corner 2) node[above]{\scriptsize $2^{\prime}$} ; 
\draw[] (a.corner 3) node[left]{\scriptsize $1^{\prime}$} ; 
\draw[] (a.corner 4) node[below]{\scriptsize $0^{\prime}$} ; 
\draw[] (a.corner 5) node[below]{\scriptsize $0$} ; 
\draw[] (a.corner 6) node[right]{\scriptsize $1$} ; 

\path[-,font=\scriptsize,color=red]
    (a.corner 3) edge[-]  (a.corner 6); 
\end{tikzpicture}
\;\;\;\;\;\;\;\;
\mathcal{P}^{\prime}=\begin{tikzpicture}[baseline={([yshift=-.5ex]current bounding box.center)}]
\node[draw,minimum size=1.5cm,regular polygon,regular polygon sides=6] (a) {};

\draw[] (a.corner 1) node[above]{\scriptsize $2$} ; 
\draw[] (a.corner 2) node[above]{\scriptsize $2^{\prime}$} ; 
\draw[] (a.corner 3) node[left]{\scriptsize $1^{\prime}$} ; 
\draw[] (a.corner 4) node[below]{\scriptsize $0^{\prime}$} ; 
\draw[] (a.corner 5) node[below]{\scriptsize $0$} ; 
\draw[] (a.corner 6) node[right]{\scriptsize $1$} ; 

\path[-,font=\scriptsize,color=red]
    (a.corner 2) edge[-]  (a.corner 4);
\path[-,font=\scriptsize,color=red]
    (a.corner 1) edge[-]  (a.corner 5);     
\end{tikzpicture}
\;\;\;\;\;\;\;\;
\mathcal{P}^{\prime \prime}=
\begin{tikzpicture}[baseline={([yshift=-.5ex]current bounding box.center)}]
\node[draw,minimum size=1.5cm,regular polygon,regular polygon sides=10] (a) {};

\path[-,font=\scriptsize,color=red]
    (a.corner 2) edge[-]  (a.corner 6); \path[-,font=\scriptsize,color=red]
    (a.corner 1) edge[-]  (a.corner 7); 
\path[-,font=\scriptsize,color=red]
    (a.corner 1) edge[-]  (a.corner 9); 
\path[-,font=\scriptsize,color=red]
    (a.corner 7) edge[-]  (a.corner 9);
\path[-,font=\scriptsize,color=red]
    (a.corner 2) edge[-]  (a.corner 4); 
\path[-,font=\scriptsize,color=red]
    (a.corner 4) edge[-]  (a.corner 6);  
\draw[] (a.corner 1) node[above]{\scriptsize $4$} ;  
\draw[] (a.corner 2) node[above]{\scriptsize $4^{\prime}$} ;  
\draw[] (a.corner 3) node[above]{\scriptsize $3^{\prime}$} ;  
\draw[] (a.corner 4) node[left]{\scriptsize $2^{\prime}$} ;  
\draw[] (a.corner 5) node[below]{\scriptsize $1^{\prime}$} ;   
\draw[] (a.corner 6) node[below]{\scriptsize $0^{\prime}$} ;  
\draw[] (a.corner 7) node[below]{\scriptsize $0$} ;   
\draw[] (a.corner 8) node[below]{\scriptsize $1$} ;  
\draw[] (a.corner 9) node[right]{\scriptsize $2$} ;  
\draw[] (a.corner 10) node[above]{\scriptsize $3$} ;  
\end{tikzpicture}
\]
\caption{Examples of symmetric polyhedral subdivisions.}
\label{fig:symmSubdivisions}
\end{center}
\end{figure}
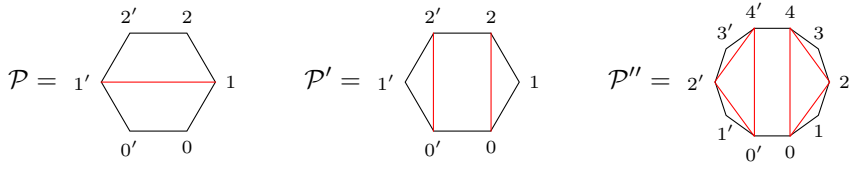

\begin{Prop}
\label{prop:symmDecomp}
Let $F_{\bullet}: Y_{\bullet} \rightarrow X_{\bullet}$ be a morphism of semi-simplicial spaces. Assume that $X_{\bullet}$ is $2$-Segal. Then the following statements are equivalent:
\begin{enumerate}
\item The morphism $F_{\bullet}$ is relative $2$-Segal.

\item For every $n \geq 2$ and every symmetric polyhedral subdivision $\mathcal{P}$ of $P_{\mathbf{n}}$ the morphism $f_{\mathcal{P}}^{F_{\bullet}}$ is a weak equivalence.

\item For every $n \geq 2$ and every maximal symmetric polyhedral subdivision $\mathcal{P}$ of $P_{\mathbf{n}}$ the morphism $f_{\mathcal{P}}^{F_{\bullet}}$ is a weak equivalence.
\end{enumerate}
\end{Prop}

\begin{proof}
Assume that the second statement holds. If $\mathcal{P}$ is the symmetric polyhedral subdivision of $P_{\mathbf{n}}$ consisting only of the diagonal $\{i^{\prime},i\}$, then
\[
(\Delta^{\mathcal{P}}, F_{\bullet})_R \simeq Y_{\{0, \dots, i\}} \times^R_{Y_{\{i\}}} Y_{\{i, \dots, n\}}
\]
while if $\mathcal{P}^{\prime}$ consists only of the diagonals $\{0,n\}$ and $\{0^{\prime},n^{\prime}\}$, then
\[
(\Delta^{\mathcal{P}^{\prime}}, F_{\bullet})_R \simeq X_{\{0, \dots, n\}} \times^R_{X_{\{0,n\}}} Y_{\{0, n\}}.
\]
Hence $f_{\mathcal{P}}^{F_{\bullet}}$ and $f_{\mathcal{P}^{\prime}}^{F_{\bullet}}$ reduce to the $1$-Segal maps for $Y_{\bullet}$ and the relative $2$-Segal maps for $F_{\bullet}$, respectively, proving that $F_{\bullet}$ is relative $2$-Segal.

Conversely, assume that $F_{\bullet}$ is relative $2$-Segal and let $\mathcal{P}$ be a symmetric polyhedral subdivision. There exists a sequence $\mathcal{P}_0, \mathcal{P}_1, \dots, \mathcal{P}_k=\mathcal{P}$ of symmetric polyhedral subdivisions such that $\mathcal{P}_0$ is the trivial subdivision and $\mathcal{P}_j$ is obtained from $\mathcal{P}_{j-1}$ by adding a single horizontal diagonal or a mirror pair of non-symmetric diagonals. Then $f^{F_{\bullet}}_{\mathcal{P}}$ factors as the composition
\[
Y_n \simeq (\Delta^{\mathcal{P}_0}, F_{\bullet})_R \rightarrow (\Delta^{\mathcal{P}_1}, F_{\bullet})_R \rightarrow \cdots \rightarrow (\Delta^{\mathcal{P}_{k-1}}, F_{\bullet})_R \rightarrow (\Delta^{\mathcal{P}}, F_{\bullet})_R.
\]
The construction of the $\mathcal{P}_j$ ensures that each morphism in this composition is a weak equivalence, being induced by a $1$-Segal map for $Y_{\bullet}$ or a relative $2$-Segal map for $F_{\bullet}$. It follows that $f^{F_{\bullet}}_{\mathcal{P}}$ is a weak equivalence and the second statement holds.

The equivalence of the second and third statements is proved similarly.
\end{proof}

\subsection{The stable framed \texorpdfstring{$\mathcal{S}_{\bullet}$}{}-construction}

Motivated by the Hall algebra representations of \cite{soibelman2016}, \cite{franzen2016}, in this section we modify the Waldhausen $\mathcal{S}_{\bullet}$-construction so as to construct relative $2$-Segal groupoids from an abelian category together with a choice of stability condition.

Fix a field $k$. Let $\mathcal{C}$ be a $k$-linear abelian category with Grothendieck group $K_0(\mathcal{C})$. A stability function on $\mathcal{C}$ is a group homomorphism $Z : K_0(\mathcal{C}) \rightarrow \mathbb{C}$ such that
\[
Z(A) \in \mathbb{H}_+= \{ m \exp(\sqrt{-1}\pi  \phi) \mid m >0, \; \phi \in (0,1] \} \subset \mathbb{C}
\]
for all non-zero objects $A \in \mathcal{C}$ \cite[\S 2]{bridgeland2007}. Let $\phi(A) \in (0,1]$ be the phase of $Z(A)$. A non-zero object $A \in \mathcal{C}$ is called $Z$-semistable if $\phi(A^{\prime}) \leq \phi(A)$ for all non-trivial subobjects $A^{\prime} \subset A$.  For each $\phi \in (0,1]$ the full subcategory $\mathcal{C}^{Z \mhyphen \mathsf{ss}}_{\phi} \subset \mathcal{C}$ consisting of the zero object together with all objects which are $Z$-semistable of phase $\phi$ is again abelian. We can therefore form the Waldhausen simplicial groupoid $\mathcal{S}_{\bullet}(\mathcal{C}^{Z \mhyphen \mathsf{ss}}_{\phi})$, which is $2$-Segal by Theorem \ref{thm:waldSegal}.

We formulate a notion of framing following \cite[\S 4]{soibelman2016}. Fix a left exact functor $\Phi: \mathcal{C} \rightarrow \mathsf{Vect}_k$ with values in the category of finite dimensional vector spaces over $k$. A framed object of $\mathcal{C}$ is then a pair $(M,s)$ consisting of an object $M \in \mathcal{C}$ and a section $s \in \Phi(M)$. A morphism of framed objects $(M,s) \rightarrow (M^{\prime}, s^{\prime})$ is a pair $(\pi, \lambda) \in \Hom_{\mathcal{C}}(M, M^{\prime}) \times k$ which satisfies $\Phi(\pi)(s) = \lambda s^{\prime}$. A framed object $(M,s)$ is called stable framed if $M$ is $Z$-semistable and $\phi(A) < \phi(M)$ for all proper subobjects $A \subset M$ for which $s \in \Phi(A) \subset \Phi(M)$.

\begin{Ex}
We recall two standard examples of framings.
\begin{enumerate}
\item Let $\mathsf{Rep}_k(Q)$ be the category of finite dimensional representations of a quiver $Q$. Sending a representation to its dimension vector defines a surjective group homomorphism $K_0(\mathsf{Rep}_k(Q)) \rightarrow \mathbb{Z}^{Q_0}$. A tuple $\zeta =(\zeta_i)_{i \in Q_0} \in \mathbb{H}_+^{Q_0}$ defines a function $\mathbb{Z}^{Q_0} \rightarrow \mathbb{C}, d \mapsto \sum_{i \in Q_0} \zeta_i d_i$ which induces a stability function $Z$ on $\mathsf{Rep}_k(Q)$. For any $f \in \mathbb{Z}_{\geq 0}^{Q_0}$ the functor $\Phi: U \mapsto \bigoplus_{i \in Q_0} \Hom_k(k^{f_i}, U_i)$ is a framing.

\item Let $\mathsf{Coh}(X)$ be the category of coherent sheaves on a smooth projective curve $X$. Then $Z= - \deg + \sqrt{-1} \cdot \rk$ defines a stability function. The global sections functor $\Phi= H^0(X,-)$ is a framing. \qedhere
\end{enumerate}
\end{Ex}

Define a stable framed modification of the $\mathcal{S}_{\bullet}$-construction as follows. For each $n \geq 0$ let $\mathcal{S}^{\mathsf{st} \mhyphen \mathsf{fr}}_n(\mathcal{C}^{Z \mhyphen \mathsf{ss}}_{\phi})$ be the maximal groupoid of the category of diagrams of the form
\[
\begin{tikzpicture}[baseline= (a).base]
\node[scale=1] (a) at (0,0){
\begin{tikzcd}
A_{\{0,0\}} \arrow{r} & A_{\{0,1\}} \arrow{d} \arrow{r} & \cdots   \arrow{r} & A_{\{0,n\}} \arrow{d}  \arrow{r} & (M_0, s_0) \arrow{d} \\
 & A_{\{1,1\}} \arrow{r} & \cdots   \arrow{r} & A_{\{1,n\}} \arrow{d}  \arrow{r} & (M_1, s_1) \arrow{d} \\
 &  & \ddots & \vdots \arrow{d} & \vdots \arrow{d} \\
  &  & & A_{\{n,n\}}  \arrow{r} & (M_n, s_n) \arrow{d} \\
  &  &  &  & (M_{n+1},s_{n+1})
\end{tikzcd}
};
\end{tikzpicture}
\]
where each $(M_i,s_i)$ is a framed object and which have the following properties:
\begin{enumerate}[label=(\roman*)]
\item Upon forgetting the framing data $s_0, \dots, s_{n+1}$ the resulting diagram is an object of $\mathcal{S}_{n+1}(\mathcal{C}^{Z \mhyphen \mathsf{ss}}_{\phi})$.
 
\item Each pair $(M_i, s_i)$, $i=0, \dots, n$, is a stable framed object.
\end{enumerate}
In the above diagram a morphism $A_{\{i,n\}} \rightarrow (M_i,s_i)$ is simply a morphism of the underlying objects of $\mathcal{C}$. The groupoids $\mathcal{S}^{\mathsf{st} \mhyphen \mathsf{fr}}_n(\mathcal{C}^{Z \mhyphen \mathsf{ss}}_{\phi})$ assemble to a simplicial groupoid. There is a canonical simplicial morphism $F_{\bullet} :\mathcal{S}^{\mathsf{st} \mhyphen \mathsf{fr}}_{\bullet}(\mathcal{C}^{Z \mhyphen \mathsf{ss}}_{\phi}) \rightarrow \mathcal{S}_{\bullet}(\mathcal{C}^{Z \mhyphen \mathsf{ss}}_{\phi})$ which forgets the rightmost column and bottom row of a diagram.

\begin{Thm}
\label{thm:stabFrameSegal}
Let $\mathcal{C}$ be a $k$-linear abelian category with stability function $Z$ and framing $\Phi$. For each $\phi \in (0,1)$ the map $F_{\bullet}: \mathcal{S}^{\mathsf{st} \mhyphen \mathsf{fr}}_{\bullet}(\mathcal{C}^{Z \mhyphen \mathsf{ss}}_{\phi}) \rightarrow \mathcal{S}_{\bullet}(\mathcal{C}^{Z \mhyphen \mathsf{ss}}_{\phi})$ is unital relative $2$-Segal.
\end{Thm}

\begin{proof}
The proof that $\mathcal{S}^{\mathsf{st} \mhyphen \mathsf{fr}}_{\bullet}(\mathcal{C}^{Z \mhyphen \mathsf{ss}}_{\phi})$ is $1$-Segal reduces to the $2$-Segal property of $\mathcal{S}_{\bullet}(\mathcal{C}^{Z \mhyphen \mathsf{ss}}_{\phi})$, so we omit it. To verify the second of the relative $2$-Segal conditions we need to show that the functor
\[
\Psi_n: \mathcal{S}^{\mathsf{st} \mhyphen \mathsf{fr}}_n (\mathcal{C}^{Z \mhyphen \mathsf{ss}}_{\phi}) \rightarrow \mathcal{S}_n(\mathcal{C}^{Z \mhyphen \mathsf{ss}}_{\phi}) \times^{(2)}_{\mathcal{S}_{\{0, n\}}(\mathcal{C}^{Z \mhyphen \mathsf{ss}}_{\phi})} \mathcal{S}^{\mathsf{st} \mhyphen \mathsf{fr}}_{\{0,n\}}(\mathcal{C}^{Z \mhyphen \mathsf{ss}}_{\phi})
\]
is an equivalence. Here $- \times^{(2)} -$ denotes the $2$-pullback of groupoids.

Let $\mathcal{F}_n(\mathcal{C}^{Z \mhyphen \mathsf{ss}}_{\phi})$ be the groupoid of $n$-flags in $\mathcal{C}^{Z \mhyphen \mathsf{ss}}_{\phi}$, that is, diagrams in $\mathcal{C}^{Z \mhyphen \mathsf{ss}}_{\phi}$ of the form
\[
0 \rightarrowtail A_1 \rightarrowtail  \cdots \rightarrowtail A_{n-1} \rightarrowtail A_n.
\]
The forgetful functor $\mu_n : \mathcal{S}_n(\mathcal{C}^{Z \mhyphen \mathsf{ss}}_{\phi}) \rightarrow \mathcal{F}_n (\mathcal{C}^{Z \mhyphen \mathsf{ss}}_{\phi})$ is an equivalence of groupoids \cite[\S 1.3]{waldhausen1985}. Similarly, let $\mathcal{F}^{\mathsf{st} \mhyphen \mathsf{fr}}_n(\mathcal{C}^{Z \mhyphen \mathsf{ss}}_{\phi})$ be the groupoid of $n$-flags in a stable framed object. Explicitly, an object of $\mathcal{F}^{\mathsf{st} \mhyphen \mathsf{fr}}_n (\mathcal{C}^{Z \mhyphen \mathsf{ss}}_{\phi})$ is a diagram of the form
\begin{equation}
\label{eq:stableFlag}
0 \rightarrowtail A_1 \rightarrowtail  \cdots \rightarrowtail A_{n-1} \rightarrowtail A_n \rightarrowtail (M,s)
\end{equation}
with $A_1, \dots, A_n, M \in \mathcal{C}_{\phi}^{Z \mhyphen \mathsf{ss}}$ and $(M,s)$ stable framed. We claim that the forgetful functor $\nu_n : \mathcal{S}^{\mathsf{st} \mhyphen \mathsf{fr}}_n (\mathcal{C}^{Z \mhyphen \mathsf{ss}}_{\phi}) \rightarrow \mathcal{F}^{\mathsf{st} \mhyphen \mathsf{fr}}_n (\mathcal{C}^{Z \mhyphen \mathsf{ss}}_{\phi})$ is an equivalence. A quasi-inverse $\eta_n$ of $\nu_n$ can be constructed as follows. Given a flag \eqref{eq:stableFlag} set $A_{\{0, k\}} = A_k$, $1 \leq k \leq n$, and $(M_0,s_0) = (M,s)$. Define $A_{\{1, k\}}$, $1 \leq k \leq n$, and $M_1$ as the pushouts
\[
\begin{tikzpicture}[baseline= (a).base]
\node[scale=1] (a) at (0,0){
\begin{tikzcd}
A_{\{0, 1\}} \arrow[two heads]{d} \arrow[tail]{r} & A_{\{0,k\}} \arrow[dashed,two heads]{d} \\
0 \arrow[dashed,tail]{r} & A_{\{1,k\}}
\end{tikzcd}
};
\end{tikzpicture}
\qquad  \qquad
\begin{tikzpicture}[baseline= (a).base]
\node[scale=1] (a) at (0,0){
\begin{tikzcd}
A_{\{0, 1\}} \arrow[two heads]{d} \arrow[tail]{r} & M_0 \arrow[dashed,two heads]{d} \\
0 \arrow[dashed,tail]{r} & M_1.
\end{tikzcd}
};
\end{tikzpicture}
\]
Since $\mathcal{C}_{\phi}^{Z \mhyphen \mathsf{ss}}$ is abelian, it is automatic that each $A_{\{1,k\}}$ and $M_1$ are $Z$-semistable of phase $\phi$. Let $s_1$ be the image of $s_0$ under the morphism $\Phi(M_0) \rightarrow \Phi(M_1)$. We need to show that $(M_1, s_1)$ is stable framed. Note that $s_1$ is non-zero. Indeed, if $s_1=0$, then $s_0 \in \Phi(A_{\{0,1\}}) \subset \Phi(M_0)$ and $\phi(A_{\{0,1\}}) = \phi(M_0)$, contradicting the assumed framed stability of $(M_0,s_0)$. Let then $B_1 \subset M_1$ be a proper subobject with $s_1 \in \Phi(B_1) \subset \Phi(M_1)$. There exists a unique object $B_0 \in \mathcal{C}$ which fits into the exact commutative diagram
\[
\begin{tikzpicture}[baseline= (a).base]
\node[scale=1] (a) at (0,0){
\begin{tikzcd}
A_{\{0, 1\}} \arrow{d}[left]{\id} \arrow[tail]{r} & B_0 \arrow[tail]{d} \arrow[two heads]{r} & B_1 \arrow[tail]{d}  \\
A_{\{0, 1\}} \arrow[tail]{r} & M_0 \arrow[two heads]{d} \arrow[two heads]{r} & M_1 \arrow[two heads]{d} \\
& M_0 \slash B_0 \arrow{r}[below]{\sim} & M_1 \slash B_1. &
\end{tikzcd}
};
\end{tikzpicture}
\]
Since $s_1 \in \Phi(B_1)$, using the left exactness of $\Phi$ we see that $s_1$ lies in the kernel of $\Phi(M_1) \rightarrow \Phi(M_1 \slash B_1)$. It follows that $s_0$ is in the kernel of $\Phi(M_0) \rightarrow \Phi(M_0 \slash B_0)$. Hence $s_0 \in \Phi(B_0) \subset \Phi(M_0)$. Since $(M_0,s_0)$ is stable framed we have $\phi(B_0) < \phi(M_0)$. The equalities $\phi(A_{\{0,1\}})=\phi(M_0) = \phi(M_1)$ combined with the standard See-Saw properties of stability functions then give $\phi(M_1) > \phi(B_1)$, as desired. This procedure defines the top two rows of $\eta_n(A_{\{0, \bullet\}} \rightarrowtail (M_0,s_0))$ and can be iterated to define the remaining $n-2$ rows. This defines the desired quasi-inverse.

We can now prove the theorem. Consider the following commutative diagram:
\[
\begin{tikzpicture}
  \matrix (m) [matrix of math nodes,row sep=3em,column sep=4em,minimum width=2em] {
     \mathcal{S}^{\mathsf{st} \mhyphen \mathsf{fr}}_n (\mathcal{C}^{Z \mhyphen \mathsf{ss}}_{\phi}) &  \mathcal{S}_n (\mathcal{C}^{Z \mhyphen \mathsf{ss}}_{\phi}) \times^{(2)}_{\mathcal{S}_{\{0,n\}}(\mathcal{C}^{Z \mhyphen \mathsf{ss}}_{\phi})} \mathcal{S}^{\mathsf{st} \mhyphen \mathsf{fr}}_{\{0,n\}} (\mathcal{C}^{Z \mhyphen \mathsf{ss}}_{\phi}) \\
    \mathcal{F}^{\mathsf{st} \mhyphen \mathsf{fr}}_n (\mathcal{C}^{Z \mhyphen \mathsf{ss}}_{\phi}) & \mathcal{F}_n (\mathcal{C}^{Z \mhyphen \mathsf{ss}}_{\phi}) \times^{(2)}_{\mathcal{F}_{\{ 0,n \}}(\mathcal{C}^{Z \mhyphen \mathsf{ss}}_{\phi})} \mathcal{F}^{\mathsf{st} \mhyphen \mathsf{fr}}_{\{0,n\}} (\mathcal{C}^{Z \mhyphen \mathsf{ss}}_{\phi}).\\};
  \path[-stealth]
    (m-1-1) edge node [left] {$\nu_n$} (m-2-1)
            edge node [above] {$\Psi_n$} (m-1-2)
    (m-2-1.east|-m-2-2) edge node [below] {$\widetilde{\Psi}_n$} (m-2-2) 
    (m-1-2) edge node[right] {$\mu_n \times \nu_{\{0,n\}}$} (m-2-2);
\end{tikzpicture}
\]
Noting that the groupoids $\mathcal{S}_1(\mathcal{C}_{\phi}^{Z \mhyphen \mathsf{ss}})$ and $\mathcal{F}_1(\mathcal{C}_{\phi}^{Z \mhyphen \mathsf{ss}})$ are canonically equivalent, the discussion above shows that the vertical functors are equivalences. That $\Psi_n$ is an equivalence therefore reduces to the statement that $\widetilde{\Psi}_n$ is an equivalence, which is obvious. We omit the verification of the relative unital condition.
\end{proof}

\begin{Rem}
The notion of quotient datum on an abelian category $\mathcal{A}$ is introduced in \cite[\S 5.5]{walde2016}, where it is shown to define a relative $2$-Segal groupoid over $\mathcal{S}_{\bullet}(\mathcal{A})$. From this point of view, Theorem \ref{thm:stabFrameSegal} can be rephrased as the statement that a stability function $Z$ and framing $\Phi$ on an abelian category $\mathcal{C}$ defines a (mild generalization of a) quotient datum on the Artin mapping cylinder $\mathcal{C}_{\Phi}$.
\end{Rem}

\subsection{Real pseudoholomorphic polygons}

\label{sec:pseudoPoly}

In this section we give a relative variant of the pseudoholomorphic polygon construction of \cite[\S 3.8]{dyckerhoff2012b}.

Recall that an almost complex structure on a smooth manifold $M$ is an endomorphism $J : TM \rightarrow TM$ of the tangent bundle which satisfies $J \circ J = - \mathbf{1}_{TM}$. A continuous map $u: (\Sigma, \mathfrak{j}) \rightarrow (M, J)$ of almost complex manifolds is called pseudoholomorphic if it is smooth and satisfies the equation
\[
du + J \circ du \circ \mathfrak{j}=0.
\]

Let $\mathbb{H}$ be the Lobachevsky plane, realized as the open unit disk in $\mathbb{C}$ with centre the origin and metric
\[
ds^2 = \frac{dz \wedge d{\overline{z}}}{(1- \vert z \vert^2)^2}.
\]
The ideal boundary $\partial \mathbb{H}$ is the unit circle in $\mathbb{C}$. Denote by $\mathfrak{j}_{\mathbb{H}}$ the complex structure on $\mathbb{H}$. The group $\mathsf{SL}_2(\mathbb{R})$ acts on $\mathbb{H}$ by orientation preserving isometries.

We require some basic definitions from Teichm\"{u}ller theory \cite{penner2012}. Given $b,b^{\prime} \in \partial \mathbb{H}$, let $(b, b^{\prime})$ be the oriented geodesic with limiting points $b$ at infinite negative time and $b^{\prime}$ at infinite positive time. For each $n \geq 0$ let $P(b_0, \dots, b_n)$ be the ideal $(n+1)$-gon in $\mathbb{H}$ with vertices $b_0, \dots, b_n \in \partial \mathbb{H}$ numbered compatibly with the canonical orientation of $\partial \mathbb{H}$. A decoration of $P(b_0, \dots, b_n)$ is the data of a horocycle $\xi_i \in \mathsf{Hor}_{b_i}$ with hyperbolic centre $b_i$ for each $i=0, \dots, n$. The group $\mathsf{SL}_2(\mathbb{R}) \times \mathbb{R}$ acts on the set of decorated ideal $(n+1)$-gons by the formula
\[
(g,a) \cdot (P(b_0, \dots, b_n), \xi_0, \dots, \xi_n) = (P(g\cdot b_0, \dots, g \cdot b_n), g \cdot \xi_0 + a, \dots, g \cdot \xi_n +a)
\]
where $g \in \mathsf{SL}_2(\mathbb{R})$ and $a \in \mathbb{R}$ and we have used the canonical identification of $\mathbb{R}$-torsors $g: \mathsf{Hor}_{b_i} \rightarrow \mathsf{Hor}_{g \cdot b_i}$. We will often omit the decoration from the notation if it will not cause confusion.

Given a decorated ideal $(n+1)$-gon, for each $0 \leq i < j \leq n$ let $m(\xi_i, \xi_j)$ be the midpoint of the unique geodesic connecting the points $\xi_i \cap (b_i,b_j)$ and $\xi_j \cap (b_i,b_j)$. This trivializes the $\mathbb{R}$-torsor $(b_i, b_j)$ via
\begin{equation}
\label{eq:geodesTriv}
\mathbb{R} \rightarrow (b_i, b_j), \qquad 0 \mapsto m(\xi_i,\xi_j).
\end{equation}
This trivialization is invariant under the action of $\mathbb{R} \subset \mathsf{SL}_2(\mathbb{R}) \times \mathbb{R}$.

Following \cite[\S 3.8]{dyckerhoff2012b} we construct from an almost complex manifold $(M,J)$ a semi-simplicial set $\widetilde{\mathbb{T}}_{\bullet}(M)$. Let $\widetilde{\mathbb{T}}_0(M) = M$ and, for each $n \geq 1$, let $\widetilde{\mathbb{T}}_n(M)$ be the set of equivalence classes of pairs consisting of a decorated ideal $(n+1)$-gon $P$ together with a continuous map $u: (P,\mathfrak{j}_{\mathbb{H}}) \rightarrow (M,J)$ which is pseudoholomorphic on the two dimensional interior of $P$. The equivalence relation is generated by the action of $\mathsf{SL}_2(\mathbb{R}) \times \mathbb{R}$ on decorated ideal polygons. Using the trivialization \eqref{eq:geodesTriv}, we can identify $\widetilde{\mathbb{T}}_1(M)$ with the set of continuous maps $[-\infty, \infty] \rightarrow M$. The face map $\partial_i: \widetilde{\mathbb{T}}_n(M) \rightarrow \widetilde{\mathbb{T}}_{n-1}(M)$ omits the ideal boundary point $b_i$ and forms the resulting decorated $n$-gon together with the restricted morphism to $M$. It is proved in \cite[Theorem 3.8.6]{dyckerhoff2012b} that $\widetilde{\mathbb{T}}_{\bullet}(M)$ is a $2$-Segal semi-simplicial set.

Consider now the relative setting. Recall that a real structure on $(M,J)$ is a smooth involution $\tau : M \rightarrow M$ which satisfies
\[
d\tau - J \circ d\tau \circ J =0.
\]
A map $u:(\Sigma, \mathfrak{j}, \sigma) \rightarrow (M, J, \tau)$ of almost complex manifolds with real structures is called real pseudoholomorphic if it is pseudoholomorphic and satisfies $u \circ \sigma = \tau \circ u$.

The real structure $\sigma: z \mapsto -\overline{z}$ on $\mathbb{C}$ induces a real structure on $\mathbb{H}$, again denoted by $\sigma$. The subgroup $\mathsf{SL}_2(\mathbb{R})^{\sigma} \leq \mathsf{SL}_2(\mathbb{R})$ which commutes with $\sigma$ is isomorphic to $\mathbb{R}^{\times} \rtimes \mathbb{Z}_2$, the generator of $\mathbb{Z}_2$ being rotation of $\mathbb{H}$ through an angle $\pi$. Define a real ideal boundary point of $\mathbb{H}$ to be either a point of the real locus $\partial \mathbb{H}^{\sigma} = \{\sqrt{-1},-\sqrt{-1}\}$ or a pair $\{b, \sigma(b)\}$ of distinct $\sigma$-conjugate ideal boundary points.

A real ideal $(n+1)$-gon $Q=Q(b_0, \dots, b_n, \dots)$ is an ideal polygon which has exactly $n+1$ real ideal boundary points, labelled so that only $b_0$ and $b_n$ may lie in $\partial \mathbb{H}^{\sigma}$. It follows that $Q$ is of one of the following four types:
\begin{enumerate}[label=(\roman*)]
\item $Q(b_0, \dots, b_n, \sigma(b_n),\sigma(b_{n-1}), \dots, \sigma(b_1), \sigma(b_0))$

\item $Q(b_0, \dots, b_n, \sigma(b_n),\sigma(b_{n-1}), \dots, \sigma(b_1))$

\item $Q(b_0, \dots, b_n, \sigma(b_{n-1}),\dots, \sigma(b_1), \sigma(b_0))$

\item $Q(b_0, \dots, b_n, \sigma(b_{n-1}), \dots, \sigma(b_1))$.
\end{enumerate}
Then $Q$ has exactly zero, one, one and two ideal vertices in $\partial \mathbb{H}^{\sigma}$ in cases (i)-(iv), respectively. A decoration of $Q$ is a decoration of its underlying ideal polygon for which the horocycle at a vertex $b$ is equal to that at $\sigma(b)$ under the canonical identification $\mathsf{Hor}_b \simeq \mathsf{Hor}_{\sigma(b)}$. The group $\mathsf{SL}_2(\mathbb{R})^{\sigma} \times \mathbb{R}$ acts on the set of real decorated ideal polygons.

With the above notation in place, we define a semi-simplicial set $\widetilde{\mathbb{T}}_{\bullet}^{\tau}(M)$ as follows. For each $n \geq 0$ let $\widetilde{\mathbb{T}}^{\tau}_n(M)$ be the set of equivalence classes of pairs $(Q,v)$ consisting of a real decorated ideal $(n+1)$-gon $Q$ together with a real continuous map $v: (Q,\mathfrak{j}_{\mathbb{H}},\sigma) \rightarrow (M,J,\tau)$ which is pseudoholomorphic on the two dimensional interior of $Q$. The equivalence relation is generated by the action of $\mathsf{SL}_2(\mathbb{R})^{\sigma} \times \mathbb{R}$ on real decorated ideal polygons. In particular, we have
\begin{equation}
\label{eq:pathDecomp}
\widetilde{\mathbb{T}}_0^{\tau}(M) \simeq M^{\tau} \sqcup C^0([-\infty,\infty],M)^{\tau}
\end{equation}
where $C^0([-\infty,\infty],M)^{\tau}$ denotes the set of real continuous maps $[-\infty, \infty] \rightarrow M$ with the compactified real line $\overline{\mathbb{R}}=[-\infty, \infty]$ given the $\mathbb{Z}_2$-action $x \mapsto -x$. Define face maps $\partial_i: \widetilde{\mathbb{T}}^{\tau}_n(M) \rightarrow \widetilde{\mathbb{T}}^{\tau}_{n-1}(M)$ by omitting the real ideal boundary point $\{b_i, \sigma(b_i)\}$. Then $\widetilde{\mathbb{T}}^{\tau}_{\bullet}(M)$ forms a semi-simplicial set. 

Given $(Q, v) \in \widetilde{\mathbb{T}}_n^{\tau}(M)$ write $Q=Q(b_0, \dots, b_n, \dots)$ as above and let $\widetilde{Q}=\widetilde{Q}(b_0, \dots, b_n)$ be the decorated ideal $(n+1)$-gon obtained from $Q$ by omitting the indicated vertices and their decorations. This defines a semi-simplicial morphism
\[
F_{\bullet}: \widetilde{\mathbb{T}}_{\bullet}^{\tau}(M) \rightarrow \widetilde{\mathbb{T}}_{\bullet}(M), \qquad (Q, v)  \mapsto (\widetilde{Q}, v_{\vert \widetilde{Q}}).
\]

We can now state the main result of this section.

\begin{Thm}
\label{thm:realPseudoPoly}
Let $(M, J, \tau)$ be an almost complex manifold with real structure. The morphism $F_{\bullet}: \widetilde{\mathbb{T}}_{\bullet}^{\tau}(M) \rightarrow \widetilde{\mathbb{T}}_{\bullet}(M)$ is a relative $2$-Segal semi-simplicial set.
\end{Thm}

\begin{proof}
For each $n \geq 2$ and $0 < i < n$ we construct an inverse of the $1$-Segal map
\[
\Xi_n: \widetilde{\mathbb{T}}_n^{\tau}(M) \rightarrow \widetilde{\mathbb{T}}_{\{0, \dots, i\}}^{\tau}(M) \times_{\widetilde{\mathbb{T}}_{\{i\}}^{\tau}(M)} \widetilde{\mathbb{T}}_{\{i, \dots, n\}}^{\tau}(M).
\]
Let $(Q^{\prime},v^{\prime}) \in \widetilde{\mathbb{T}}^{\tau}_{\{0, \dots, i\}}(M)$ and $(Q^{\prime \prime},v^{\prime  \prime}) \in \widetilde{\mathbb{T}}^{\tau}_{\{i, \dots, n\}}(M)$ with equal image in $\widetilde{\mathbb{T}}_{\{i\}}^{\tau}(M)$. Since $0 < i < n$ these images lie in the component $C^0([-\infty, \infty],M)^{\tau}$ in the decomposition \eqref{eq:pathDecomp}. There exists a unique $g \in \mathsf{SL}_2(\mathbb{R})^{\sigma}$ with the property that $g \cdot (b^{\prime}_i, \sigma(b^{\prime}_i)) =(b^{\prime \prime}_i, \sigma(b^{\prime \prime}_i))$ and the restrictions
\[
g \cdot v^{\prime}, v^{\prime \prime}: [-\infty, \infty] \rightarrow M
\]
are equal. Applying Morera's theorem, we see that $g \cdot (Q^{\prime},v^{\prime})$ and $(Q^{\prime \prime},v^{\prime  \prime})$ can be glued in a unique way so as to obtain a continuous map
\[
g \cdot v^{\prime} \cup v^{\prime \prime}: g\cdot Q^{\prime} \cup_{(b^{\prime \prime}_i, \sigma(b^{\prime \prime}_i))} Q^{\prime \prime} \rightarrow M
\]
which is pseudoholomorphic on the two dimensional interior. Since both $Q^{\prime}$ and $Q^{\prime  \prime}$ (resp. $v^{\prime}$ and $v^{\prime  \prime}$) are real, so too is $g\cdot Q^{\prime} \cup_{(b^{\prime \prime}_i, \sigma(b^{\prime \prime}_i))} Q^{\prime \prime}$ (resp. $g\cdot v^{\prime} \cup v^{\prime \prime}$). See Figure \ref{fig:1SegRealPoly}. This defines the required inverse of $\Xi_n$, showing that $\widetilde{\mathbb{T}}^{\tau}_{\bullet}(M)$ is $1$-Segal.
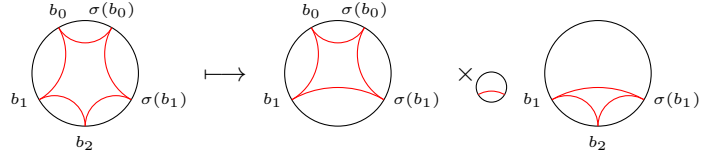
\begin{figure}
\begin{center}
\[
\begin{tikzpicture}[baseline=-0.65ex,scale=.7]
\draw[] (120:1) node[above]{\tiny $b_0$}; 
\draw[] (210:1) node[left]{\tiny $b_1$}; 
\draw[] (270:1) node[below]{\tiny $b_2$}; 
\draw[] (330:1) node[right]{\tiny $\sigma(b_1)$}; 
\draw[] (60:1) node[above]{\tiny $\sigma(b_0)$}; 
\draw (0,0) circle (1);
\clip (0,0) circle (1);
\hgline{120}{210}
\hgline{210}{270}
\hgline{270}{330}
\hgline{-30}{60}
\hgline{60}{120}
\end{tikzpicture}
\longmapsto
\begin{tikzpicture}[baseline=-0.65ex,scale=.7]
\draw[] (120:1) node[above]{\tiny $b_0$}; 
\draw[] (210:1) node[left]{\tiny $b_1$};
\draw[] (330:1) node[right]{\tiny $\sigma(b_1)$}; 
\draw[] (60:1) node[above]{\tiny $\sigma(b_0)$}; 
\draw (0,0) circle (1);
\clip (0,0) circle (1);
\hgline{120}{210}
\hgline{210}{330}
\hgline{-30}{60}
\hgline{60}{120}
\end{tikzpicture}
\times_{
\begin{tikzpicture}[scale=.2] 
\draw (0,0) circle (1);
\clip (0,0) circle (1);
\hgline{210}{330}
\end{tikzpicture}
}
\begin{tikzpicture}[baseline=-0.65ex,scale=.7]
\draw[] (210:1) node[left]{\tiny $b_1$}; 
\draw[] (270:1) node[below]{\tiny $b_2$}; 
\draw[] (330:1) node[right]{\tiny $\sigma(b_1)$}; 
\draw (0,0) circle (1);
\clip (0,0) circle (1);
\hgline{210}{270}
\hgline{270}{330}
\hgline{330}{-30}
\hgline{210}{330}
\end{tikzpicture}
\]
\caption{The map $\Xi_2: \widetilde{\mathbb{T}}_2^{\tau}(M) \rightarrow \widetilde{\mathbb{T}}^{\tau}_{\{0,1\}}(M) \times_{\widetilde{\mathbb{T}}^{\tau}_{\{1\}}(M)} \widetilde{\mathbb{T}}^{\tau}_{\{1,2\}}(M)$, omitting the data of the maps to $M$.}
\label{fig:1SegRealPoly}
\end{center}
\end{figure}

To verify the second of the relative $2$-Segal conditions, consider the map
\[
\Psi_n: \widetilde{\mathbb{T}}_n^{\tau}(M) \rightarrow \widetilde{\mathbb{T}}_n (M) \times_{\widetilde{\mathbb{T}}_{\{0,n\}}(M)} \widetilde{\mathbb{T}}_{\{0,n\}}^{\tau}(M), \qquad n \geq 1.
\]
Let $(P^{\prime}, u^{\prime}) \in \widetilde{\mathbb{T}}_n (M)$ and $(Q^{\prime \prime}, v^{\prime \prime}) \in \widetilde{\mathbb{T}}^{\tau}_{\{0, n\}}(M)$ with equal image in $\widetilde{\mathbb{T}}_{\{0, n\}}(M)$. Choose $g \in \mathsf{SL}_2(\mathbb{R})$ so that $g \cdot (b^{\prime}_0, b^{\prime}_n) = (b^{\prime \prime}_0, b^{\prime \prime}_n)$. As above, the reality condition on $(Q^{\prime \prime}, v^{\prime \prime})$ implies that, up to the action of $\mathsf{SL}_2(\mathbb{R})^{\sigma}$, the triple
\[
\{g \cdot (P^{\prime}, u^{\prime}), (Q^{\prime \prime}, v^{\prime \prime} ), g \cdot (\sigma(P^{\prime}), \tau \circ u^{\prime} \circ \sigma^{-1} )\}
\]
can be glued in a unique way so as to obtain an element of $\widetilde{\mathbb{T}}^{\tau}_n (M)$. See Figure \ref{fig:2SegRealPoly}. This defines the desired inverse of $\Psi_n$.
\begin{figure}
\begin{center}
\[
\begin{tikzpicture}[baseline=-0.65ex,scale=.7]
\draw[] (120:1) node[above]{\tiny $b_0$}; 
\draw[] (175:1) node[left]{\tiny $b_1$}; 
\draw[] (220:1) node[left]{\tiny $b_2$}; 
\draw[] (320:1) node[right]{\tiny $\sigma(b_2)$}; 
\draw[] (365:1) node[right]{\tiny $\sigma(b_1)$}; 
\draw[] (60:1) node[above]{\tiny $\sigma(b_0)$}; 
\draw (0,0) circle (1);
\clip (0,0) circle (1);
\hgline{120}{175}
\hgline{175}{220}
\hgline{220}{320}
\hgline{320}{365}
\hgline{365}{420}
\hgline{120}{60}
\end{tikzpicture}
\longmapsto
\begin{tikzpicture}[baseline=-0.65ex,scale=.7]
\draw[] (120:1) node[above]{\tiny $b_0$}; 
\draw[] (175:1) node[left]{\tiny $b_1$}; 
\draw[] (220:1) node[left]{\tiny $b_2$};
\draw (0,0) circle (1);
\clip (0,0) circle (1);
\hgline{120}{175}
\hgline{175}{220}
\hgline{220}{120}
\end{tikzpicture}
\times_{
\begin{tikzpicture}[scale=.2]
\draw (0,0) circle (1);
\clip (0,0) circle (1);
\hgline{220}{120}
\end{tikzpicture}
}
\begin{tikzpicture}[baseline=-0.65ex,scale=.7]
\draw[] (120:1) node[above]{\tiny $b_0$}; 
\draw[] (220:1) node[left]{\tiny $b_2$}; 
\draw[] (320:1) node[right]{\tiny $\sigma(b_2)$};
\draw[] (60:1) node[above]{\tiny $\sigma(b_0)$}; 
\draw (0,0) circle (1);
\clip (0,0) circle (1);
\hgline{220}{120}
\hgline{220}{320}
\hgline{120}{60}
\hgline{60}{-40}
\end{tikzpicture}
\vspace{4pt}
\]
\caption{The map $\Psi_2: \widetilde{\mathbb{T}}^{\tau}_2(M) \rightarrow \widetilde{\mathbb{T}}_2(M) \times_{\widetilde{\mathbb{T}}_{\{0,2\}}(M)} \widetilde{\mathbb{T}}^{\tau}_{\{0,2\}}(M)$, omitting the data of the maps to $M$.}
\label{fig:2SegRealPoly}
\end{center}
\end{figure}
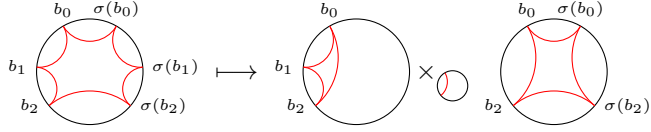
\end{proof}

\section{Relative higher Segal spaces from categories with dualities}
\label{sec:relSegDual}

\subsection{Categories with duality}
\label{sec:simpSub}

We recall some basics about categories with duality. For further details the reader is referred to \cite{schlichting2010}.

\begin{Def}
A category with (strong) duality is a triple $(\mathcal{C}, P, \Theta)$ consisting of a  category $\mathcal{C}$, a functor $P: \mathcal{C}^{\mathsf{op}} \rightarrow \mathcal{C}$ and a natural isomorphism $\Theta: \mathbf{1}_{\mathcal{C}} \Rightarrow P \circ P^{\mathsf{op}}$ such that $P(\Theta_U) \circ \Theta_{P(U)} = \mathbf{1}_{P(U)}$ for all objects $U \in \mathcal{C}$.
\end{Def}

When there is no risk of confusion we will omit $(P, \Theta)$ from the notation and simply refer to $\mathcal{C}$ as a category with duality.

A form functor $(T,\varphi): (\mathcal{C}, P, \Theta) \rightarrow (\mathcal{D}, Q, \Xi)$ between categories with duality consists of a functor $T: \mathcal{C} \rightarrow \mathcal{D}$ and a natural transformation $\varphi: T \circ P \Rightarrow Q \circ T^{\mathsf{op}}$ for which the diagram
\[
\begin{tikzpicture}[baseline= (a).base]
\node[scale=0.85] (a) at (0,0){
\begin{tikzcd}[column sep=small]
T(U) \arrow{r}[above]{\Xi_{T(U)}} \arrow{d}[left]{T(\Theta_U)} & [2.5em] Q^2T(U) \arrow{d}[right]{Q(\varphi_U)}  \\
[1em] TP^2(U) \arrow{r}[below]{\varphi_{P(U)}} & QTP(U)
\end{tikzcd}
};
\end{tikzpicture}
\]
commutes for all objects $U \in \mathcal{C}$. In particular, let $T: \mathcal{C} \rightarrow \mathcal{D}$ be a duality preserving functor, that is, a functor for which $T \circ P = Q \circ T^{\mathsf{op}}$ and $T \Theta = \Xi T$, where $T \Theta$ (resp. $\Xi T$) is the left (resp. right) whiskering of $T$ with $\Theta$ (resp. $\Xi$). Then $(T, \mathbf{1}_{T})$ is a form functor. Write $\mathsf{CatD}$ for the category of small categories with duality with form functors as morphisms.

A (nonsingular) symmetric form in $\mathcal{C}$ is a pair $(N, \psi_N)$, often just denoted by $N$, consisting of an object $N \in \mathcal{C}$ and an isomorphism $\psi_N: N \rightarrow P(N)$ which satisfies $P(\psi_N) \circ \Theta_N = \psi_N$. An isometry of symmetric forms $\phi: (M, \psi_M) \rightarrow (N, \psi_N)$ is an isomorphism $\phi: M \rightarrow N$ which satisfies $P(\phi) \circ  \psi_N \circ \phi = \psi_M$. The groupoid of symmetric forms and their isometries is denoted by $\mathcal{C}_h$ and is called the Hermitian groupoid of $\mathcal{C}$.

Recall that $\mathbf{n}$ denotes the totally ordered set $\{ 0 < \cdots < n < n^{\prime} <  \cdots < 0^{\prime}\}$. The morphisms $
[n] \rightarrow \mathbf{n}$, $i \mapsto i$, define the subdivision functor $\mathsf{sd}: \Delta \rightarrow \Delta$. The edgewise subdivision of a simplicial object $X_{\bullet}$ of a category $\mathcal{C}$ is then defined to be the composition
\[
X_{\bullet}^e: \Delta^{\mathsf{op}} \xrightarrow[]{\mathsf{sd}^{\mathsf{op}}} \Delta^{\mathsf{op}} \xrightarrow[]{X_{\bullet}} \mathcal{C}.
\]
Explicitly, $X_n^e = X_{\mathbf{n}}= X_{2n+1}$ with face maps
\[
\partial_i^e: X^e_n = X_{2n+1} \xrightarrow[]{\partial_i \circ \partial_{2n+1-i}} X_{2n-1} = X^e_{n-1},
\]
the degeneracy maps admitting a similar description. The morphisms $[n] \rightarrow \mathbf{n}$ define a simplicial morphism $X^e_{\bullet} \rightarrow X_{\bullet}$.

In the sections below we will repeatedly apply the following result (\textit{cf.} \cite[\S 1.5]{hornbostel2004}).

\begin{Lem}
\label{lem:subdivDuality}
Let $X_{\bullet}$ be a simplicial object of $\mathsf{Cat}$. For each $n\geq 0$ let $(P_n, \Theta_n)$, $n \geq 0$ be duality a structures on $X_n$. Suppose that
\[
P_{n-1} \circ \partial^{\mathsf{op}}_i = \partial_{n-i} \circ P_n, \qquad P_{n+1} \circ s^{\mathsf{op}}_i = s_{n-i} \circ P_n
\]
and
\[
\partial_i  \Theta_n = \Theta_{n-1}  \partial_i, \qquad s_i   \Theta_n = \Theta_{n+1}  s_i
\]
for all $n \geq 1$ and $0 \leq i \leq n$. Then $(X_{\bullet}^e, P_{\bullet}^e, \Theta^e_{\bullet})$ is a simplicial object of $\mathsf{CatD}$.
\end{Lem}

\begin{proof}
It is clear that $(X_n^e, P_n^e, \Theta^e_n) = (X_{2n+1}, P_{2n+1}, \Theta_{2n+1})$ is a category with duality. The equality $\partial_i \Theta_n = \Theta_{n-1}  \partial_i$ together with the calculation
\begin{eqnarray*}
\partial^e_i \circ P^e_n &=& \partial_i \circ \partial_{2n+1-i} \circ P_{2n+1} \\
&=& \partial_i \circ P_{2n} \circ \partial^{\mathsf{op}}_i \\
&=& P_{2n-1} \circ \partial^{\mathsf{op}}_{2n-i} \circ \partial^{\mathsf{op}}_i \\
&=& P_{2n-1} \circ \partial^{\mathsf{op}}_{i} \circ \partial^{\mathsf{op}}_{2n-1-i} = P_{n-1}^e \circ (\partial_i^e)^{\mathsf{op}}
\end{eqnarray*}
shows that the face map $(\partial_i^e, \mathbf{1}_{\partial_i^e}) : X_n^e \rightarrow X_{n-1}^e$ is a form functor. A similar calculation shows that $(s_i^e, \mathbf{1}_{s_i^e}) : X_n^e \rightarrow X_{n+1}^e$ is a form functor.
\end{proof}

\subsection{Unoriented categorified nerves}

\label{sec:unoriNerve}

As a warm-up for Section \ref{sec:unoriLoopStack} we adapt the categorified nerve construction to the relative setting.

Let $\mathcal{C}$ be a small category. Its categorified nerve $\mathcal{N}_{\bullet}(\mathcal{C})$ (see Section \ref{sec:1SegSpace}) is a $1$-Segal simplicial groupoid. Suppose that $(P, \Theta)$ is a duality structure on $\mathcal{C}$. Then the groupoids $\mathcal{N}_n(\mathcal{C})$, $n \geq 0$, inherit duality structures which are compatible in the sense of Lemma \ref{lem:subdivDuality}. The edgewise subdivision $\mathcal{N}^e_{\bullet} (\mathcal{C})$ is therefore a simplicial groupoid with duality.

\begin{Def}
The Hermitian groupoid $\mathcal{N}^e_{\bullet} (\mathcal{C})_h$, denoted by $\mathcal{U}_{\bullet}(\mathcal{C})$, is called the categorified unoriented nerve of $(\mathcal{C}, P, \Theta)$.
\end{Def}

Slightly abusively, write $(x_{\bullet}, \psi_{\bullet}) \in \mathcal{U}_n(\mathcal{C})$ for the object consisting of the diagram
\[
x_0 \rightarrow \cdots \rightarrow x_n \rightarrow x_{n^{\prime}} \rightarrow \cdots \rightarrow x_{0^{\prime}}
\]
in $\mathcal{C}$ together with isomorphisms $\psi_i : x_i \rightarrow P(x_{i^{\prime}})$, $0 \leq i \leq 0^{\prime}$, which satisfy $P(\psi_i) \circ \Theta_{x_{i^{\prime}}} = \psi_{i^{\prime}}$ and make the obvious squares commute. Here we use the convention that $i^{\prime \prime} = i$. The forgetful morphism $F_{\bullet}: \mathcal{U}_{\bullet}(\mathcal{C}) \rightarrow \mathcal{N}_{\bullet}(\mathcal{C})$ is given by
\[
\mathcal{U}_n(\mathcal{C}) \owns
(x_{\bullet}, \psi_{\bullet}) \mapsto (x_0 \rightarrow \cdots \rightarrow x_n) \in \mathcal{N}_n(\mathcal{C}).
\]

\begin{Prop}
\label{prop:catUnoriNerve}
For any small category with duality $(\mathcal{C}, P, \Theta)$, the morphism $F_{\bullet} : \mathcal{U}_{\bullet}(\mathcal{C}) \rightarrow \mathcal{N}_{\bullet}(\mathcal{C})$ is unital right relative $1$-Segal.
\end{Prop}

\begin{proof}
Fix $0 < i < n$. An object of the $2$-pullback $\mathcal{N}_{\{0, \dots, i\}}(\mathcal{C}) \times^{(2)}_{\mathcal{N}_{\{i\}}(\mathcal{C})} \mathcal{U}_{\{i, \dots, n\}}(\mathcal{C})$ is a triple
\begin{equation}
\label{eq:nerveFibEle}
(x_0 \rightarrow \cdots \rightarrow x_i, (x^{\prime}_i \rightarrow \cdots \rightarrow x^{\prime}_n \rightarrow x^{\prime}_{n^{\prime}} \rightarrow \cdots \rightarrow x^{\prime}_{i^{\prime}}, \psi_{\bullet}); \alpha)
\end{equation}
with $\alpha: x_i \rightarrow x_i^{\prime}$ an isomorphism. A morphism from the object \eqref{eq:nerveFibEle} to a second object, say
\[
(y_0 \rightarrow \cdots \rightarrow y_i, (y^{\prime}_i \rightarrow \cdots \rightarrow y^{\prime}_n \rightarrow y^{\prime}_{n^{\prime}} \rightarrow \cdots \rightarrow y^{\prime}_{i^{\prime}}, \mu_{\bullet}); \beta),
\]
is a tuple $((p_0, \dots, p_i),(q_i, \dots, q_n, q_{n^{\prime}}, \dots, q_{i^{\prime}}))$ where $p_j: x_j \rightarrow y_j$, $1 \leq j \leq i$, and $q_j: x^{\prime}_j \rightarrow y^{\prime}_j$, $i \leq j \leq i^{\prime}$, are isomorphisms which make the obvious nerve diagrams commute, the $q_j$ respect the symmetric forms $\psi_{\bullet}$ and $\mu_{\bullet}$ and the compatibility condition
\begin{equation}
\label{eq:morphismCompat}
q_i \circ \alpha = \beta \circ p_i
\end{equation}
holds.

We need to prove that the functor
\begin{eqnarray*}
\Psi_n: \mathcal{U}_n(\mathcal{C}) & \rightarrow & \mathcal{N}_{\{0, \dots, i\}}(\mathcal{C}) \times^{(2)}_{\mathcal{N}_{\{i\}}(\mathcal{C})} \mathcal{U}_{\{i, \dots, n\}}(\mathcal{C}) \\
(x_{\bullet}, \psi_{\bullet}) & \mapsto & (x_0 \rightarrow \cdots \rightarrow x_i, (x_i \rightarrow \cdots \rightarrow x_n \rightarrow x_{n^{\prime}} \rightarrow \cdots \rightarrow x_{i^{\prime}}, \psi_{\bullet}); \mathbf{1}_{x_i})
\end{eqnarray*}
is an equivalence. Consider a morphism $(p_{\bullet},q_{\bullet}) : \Psi_n (x_{\bullet}, \psi_{\bullet}) \rightarrow \Psi_n (y_{\bullet}, \mu_{\bullet})$. Since $\alpha$ and $\beta$ are the identity maps in this case, the compatibility condition \eqref{eq:morphismCompat} reduces to $q_i=p_i$. Any morphism $\Psi_n (x_{\bullet}, \psi_{\bullet}) \rightarrow \Psi_n (y_{\bullet}, \mu_{\bullet})$ is therefore the image of a unique isometry $(x_{\bullet}, \psi_{\bullet}) \rightarrow (y_{\bullet}, \mu_{\bullet})$, since for an isometry the map $x_j \rightarrow y_j$ uniquely determines the map $x_{j^{\prime}} \rightarrow y_{j^{\prime}}$. It follows that $\Psi_n$ is fully faithful. To show that $\Psi_n$ is essentially surjective suppose that we are given an object of the form \eqref{eq:nerveFibEle}. Consider the object of $\mathcal{U}_n(\mathcal{C})$ whose underlying diagram is the top row of
\[
\begin{tikzcd}[column sep = 0.45em]
x_0 \arrow{r} &  \cdots \arrow{r} & x_{i-1} \arrow{r} \arrow{d} & x^{\prime}_i \arrow{r} & \cdots \arrow{r} & x^{\prime}_n \arrow{r} & x^{\prime}_{n^{\prime}} \arrow{r} & \cdots \arrow{r} & x^{\prime}_{i^{\prime}} \arrow{r} \arrow{d}[left]{\psi_{i^{\prime}}} & P(x_{i-1}) \arrow{r} & \cdots \arrow{r} & P(x_0) \\
& & x_i \arrow{ru}[right]{\alpha} & & & & & & P(x^{\prime}_i) \arrow{ru} & & &
\end{tikzcd}
\]
and whose symmetric form is $(\{\Theta_{\bullet}\}_{[0,i]}, \{\psi_{\bullet}\}_{[i, i^{\prime}]}, \{\mathbf{1}_{\bullet}\}_{[i^{\prime},0^{\prime}]})$, the subscripts indicating the index intervals to which each map applies. In the above diagram the triangles define the corresponding horizontal morphisms and the unlabelled maps are the canonical ones. The image of this object under $\Psi_n$ is
\[
(x_0 \rightarrow \cdots \rightarrow x_{i-1} \rightarrow x^{\prime}_i, (x_i^{\prime} \rightarrow \cdots \rightarrow x^{\prime}_n \rightarrow x^{\prime}_{n^{\prime}} \rightarrow \cdots \rightarrow x^{\prime}_{i^{\prime}}, \psi_{\bullet}); \mathbf{1}_{x^{\prime}_i})
\]
which is isomorphic via $((\mathbf{1}_{x_0}, \dots, \mathbf{1}_{x_{i-1}}, \alpha^{-1}),(\mathbf{1}_{x^{\prime}_i}, \dots, \mathbf{1}_{x^{\prime}_{i^{\prime}}}))$ to the object \eqref{eq:nerveFibEle}.
\end{proof}

At the level of simplicial sets, the constructions of this section define a morphism $N^e_{\bullet}(\mathcal{C})_h \rightarrow N_{\bullet}(\mathcal{C})$. However, this morphism is neither left nor right relative $1$-Segal.

\subsection{Unoriented categorified twisted cyclic nerves}
\label{sec:unoriLoopStack}

Let $X$ be a topological space with a self-homeomorphism $T: X \rightarrow X$. The $T$-twisted loop space of $X$ is then defined to be
\[
L^T X = \{ \gamma \in C^0(\mathbb{R},X) \mid \gamma(t+1)=T(\gamma(t)) \}.
\]
The case of the identity map $T = \mathbf{1}_X$ recovers the ordinary loop space $LX$. Suppose that we are also given a continuous involution $p: X \rightarrow X$ which satisfies
\[
p=T \circ p \circ T.
\]
This, together with the orientation reversing involution $\sigma : \mathbb{R} \rightarrow \mathbb{R}$, $t \mapsto 1-t$, induces an involution
\[
p_L: L^{T}X \rightarrow L^{T}X, \qquad \gamma \mapsto p \circ \gamma \circ \sigma^{-1},
\]
the fixed point set of which is naturally interpreted as the unoriented loop space of the stack $[X \slash \langle T \rangle]$. The goal of this section is to construct a relative $2$-Segal simplicial space which is a categorical analogue of the unoriented loop space.

Let $\mathcal{C}$ be a small category with an endofunctor $T: \mathcal{C} \rightarrow \mathcal{C}$. The categorified $T$-twisted cyclic nerve of $\mathcal{C}$ (see \cite[\S 3.3]{dyckerhoff2012b}, \cite{drinfeld2004}) is the simplicial groupoid $\mathcal{N}C_{\bullet}^{T} (\mathcal{C})$ which assigns to $[n] \in \Delta$ the groupoid of all diagrams in $\mathcal{C}$ of the form
\begin{equation}
\label{eq:nerveObject}
x_{\bullet} = \{x_0 \xrightarrow[]{f_0} x_1 \xrightarrow[]{f_1} \cdots \xrightarrow[]{f_{n-2}} x_{n-1} \xrightarrow[]{f_{n-1}} x_n \xrightarrow[]{f_n} T(x_0) \}.
\end{equation}
A morphism $x_{\bullet} \rightarrow y_{\bullet}$ in $\mathcal{N}C_{\bullet}^{T} (\mathcal{C})$ is a collection of isomorphisms $x_i \rightarrow y_i$, $0 \leq i \leq n$, which make the obvious diagrams commute. For $i =1, \dots, n$ the face map $\partial_i : \mathcal{N}C_n^{T}(\mathcal{C}) \rightarrow \mathcal{N}C_{n-1}^{T}(\mathcal{C})$ omits $x_i$ and composes $f_i$ and $f_{i-1}$. The face map $\partial_0$ sends \eqref{eq:nerveObject} to
\[
x_1 \xrightarrow[]{f_1} \cdots \xrightarrow[]{f_{n-2}} x_{n-1} \xrightarrow[]{f_{n-1}} x_n \xrightarrow[]{T(f_0) \circ f_n} T(x_1).
\]
The degeneracy maps insert identity morphisms at appropriate spots.

It is proved in \cite[Theorem 3.2.3]{dyckerhoff2012b} that the (non-categorified) $T$-twisted cyclic nerve $NC_{\bullet}^{T} (\mathcal{C})$ is a unital $2$-Segal simplicial set. The analogous result holds also in the categorified setting.

\begin{Thm}
For any small category $\mathcal{C}$ and endofunctor $T: \mathcal{C} \rightarrow \mathcal{C}$, the simplicial groupoid $\mathcal{N}C_{\bullet}^T(\mathcal{C})$ is unital $2$-Segal.
\end{Thm}

\begin{proof}
Omitted. A similar result will be proved in Theorem \ref{thm:twNerveRel2Segal} below.
\end{proof}

In addition to the pair $(\mathcal{C},T)$, suppose now that we are given a duality structure $(P, \Theta)$ on $\mathcal{C}$ and a natural transformation
\[
\lambda: P \Rightarrow T \circ P \circ T^{\mathsf{op}}  
\]
which satisfies the compatibility condition
\begin{equation}
\label{eq:loopCompat}
TP(\lambda_x) \circ \lambda_{PT(x)} \circ \Theta_{T(x)} = T(\Theta_x), \qquad x \in \mathcal{C}.
\end{equation}
For example, the pair $(T, \lambda) = (\mathbf{1}_{\mathcal{C}}, \mathbf{1}_P)$ satisfies this condition.

For each $n \geq 0$ define a functor $P_n : \mathcal{N}C^{T}_n(\mathcal{C})^{\mathsf{op}} \rightarrow \mathcal{N}C^{T}_n(\mathcal{C})$ by sending the diagram \eqref{eq:nerveObject} to
\[
P(x_n) \xrightarrow[]{P(f_{n-1})} P(x_{n-1}) \xrightarrow[]{P(f_{n-2})} \cdots \xrightarrow[]{P(f_1)} P(x_1) \xrightarrow[]{P(f_0)} P(x_0) \xrightarrow[]{P_*(f_n)} T P(x_n)
\]
where $P_*(f_n) : P(x_0) \xrightarrow[]{\lambda_{x_0}} TPT(x_0) \xrightarrow[]{TP(f_n)} T P(x_n)$. The action of $P_n$ on morphisms is the obvious one.

\begin{Lem}
\label{lem:twistNerveDual}
The data $\Theta_{n,x_{\bullet}} = (\Theta_{x_0}, \dots, \Theta_{x_n}, T(\Theta_{x_0}))$, $x_{\bullet} \in \mathcal{N}C^{T}_n(\mathcal{C})$, define a natural isomorphism $\Theta_n: \mathbf{1}_{\mathcal{N}C^{T}_n(\mathcal{C})} \Rightarrow P_n \circ P^{\mathsf{op}}_n$, giving a triple $(\mathcal{N}C^{T}_{\bullet}(\mathcal{C}), P_{\bullet}, \Theta_{\bullet})$ which satisfies the hypotheses of Lemma \ref{lem:subdivDuality}.
\end{Lem}

\begin{proof}
We will prove that $\Theta_{n,x_{\bullet}}$ defines a morphism $x_{\bullet} \rightarrow P^2_n (x_{\bullet})$; the remaining statements of the lemma can be verified directly. Keeping the notation \eqref{eq:nerveObject}, we need to show that the diagram
\[
\begin{tikzpicture}[baseline= (a).base]
\node[scale=1] (a) at (0,0){
\begin{tikzcd}[column sep=small]
x_n \arrow{r}[above]{f_n} \arrow{d}[left]{\Theta_{x_n}} & [2.5em] T(x_0) \arrow{d}[right]{T(\Theta_{x_0})}  \\
[1em] P^2(x_n) \arrow{r}[below]{P_*^2(f_n)} & TP^2(x_0)
\end{tikzcd}
};
\end{tikzpicture}
\]
commutes. The definition of $P_*$ implies that
\[
P_*^2(f_n) = TP(\lambda_{x_0}) \circ TPTP(f_n) \circ \lambda_{P(x_n)}.
\]
The natural transformation $\lambda$ associates to $x_n \xrightarrow[]{f_n} T(x_0)$ the commutative diagram
\[
\begin{tikzpicture}[baseline= (a).base]
\node[scale=1] (a) at (0,0){
\begin{tikzcd}
TPTP(x_n) \arrow{r}[above]{TPTP(f_n)} & [3em] TPTPT(x_0) \\
[1em] P^2(x_n) \arrow{r}[below]{P^2(f_n)} \arrow{u}[left]{\lambda_{P(x_n)}} & P^2T(x_0). \arrow{u}[right]{\lambda_{PT(x_0)}}
\end{tikzcd}
};
\end{tikzpicture}
\]
Combining this with the equality $P^2(f_n) \circ \Theta_{x_n} =\Theta_{T(x_0)} \circ f_n$, we compute
\begin{eqnarray*}
P_*^2(f_0) \circ \Theta_{x_n} &=& TP(\lambda_{x_0}) \circ TPTP(f_n) \circ \lambda_{P(x_n)} \circ \Theta_{x_n}\\
 &=& TP(\lambda_{x_0}) \circ \lambda_{PT(x_0)} \circ \Theta_{T(x_0)} \circ f_n \\
 &=& T(\Theta_{x_0}) \circ f_n,
\end{eqnarray*}
where in the final step the compatibility condition \eqref{eq:loopCompat} was used.
\end{proof}

Lemmas \ref{lem:subdivDuality} and \ref{lem:twistNerveDual} imply that $\mathcal{N}C^{T,e}_{\bullet}(\mathcal{C})$ is a simplicial groupoid with duality.

\begin{Def}
The Hermitian groupoid $\mathcal{N}C^{T,e}_{\bullet}(\mathcal{C})_h$, denoted by $\mathcal{N}U^T_{\bullet}(\mathcal{C})$, is called the unoriented categorified $T$-twisted cyclic nerve of $(\mathcal{C}, P, \Theta; T, \lambda)$.
\end{Def}

Write $(x_{\bullet}, \psi_{\bullet}) \in \mathcal{N}U^T_n(\mathcal{C})$ for the diagram
\[
x_0 \xrightarrow[]{f_0} x_1 \xrightarrow[]{f_1} \cdots \xrightarrow[]{f_{n-1}} x_n \xrightarrow[]{f_n} x_{n^{\prime}} \xrightarrow[]{f_{n^{\prime}}} x_{(n-1)^{\prime}} \xrightarrow[]{f_{(n-1)^{\prime}}} \cdots \xrightarrow[]{f_{1^{\prime}}} x_{0^{\prime}} \xrightarrow[]{f_{0^{\prime}}} T(x_0)
\]
together with symmetric isomorphisms $\psi_{\bullet}$ which satisfy $\psi_{i+1} \circ f_i = P(f_{(i+1)^{\prime}}) \circ \psi_i$, $i=0, \dots, n-1$, and
\[
\psi_{n^{\prime}} \circ f_n = P(f_n) \circ \psi_n, \qquad T(\psi_0) \circ f_{0^{\prime}} = P_*(f_{0^{\prime}}) \circ \psi_{0^{\prime}}.
\]

The following statement gives a relative $2$-Segal space which plays the role of the unoriented loop space.

\begin{Thm}
\label{thm:twNerveRel2Segal}
For any small category with duality $(\mathcal{C}, P, \Theta)$ and endofunctor $T: \mathcal{C} \rightarrow \mathcal{C}$ with compatibility data $\lambda$ as above, the morphism $F_{\bullet} : \mathcal{N}U^T_{\bullet}(\mathcal{C}) \rightarrow  \mathcal{N}C^T_{\bullet}(\mathcal{C})$ is a unital relative $2$-Segal simplicial groupoid.
\end{Thm}

\begin{proof}
We need to prove that the $1$-Segal morphisms
\[
\Xi_n: \mathcal{N}U^T_n(\mathcal{C}) \rightarrow \mathcal{N}U^T_i(\mathcal{C}) \times^{(2)}_{\mathcal{N}U^T_{\{i\}}(\mathcal{C})} \mathcal{N}U^T_{n-i}(\mathcal{C})
\]
and the relative $2$-Segal morphisms
\[
\Psi_n: \mathcal{N}U^T_n(\mathcal{C}) \rightarrow \mathcal{N}C^T_n(\mathcal{C}) \times^{(2)}_{\mathcal{N}C^T_{\{0,n\}}(\mathcal{C})} \mathcal{N}U^T_{\{0,n\}}(\mathcal{C})
\]
are equivalences. Since the proofs are similar we will only prove the latter. Explicitly, the functor $\Psi_n$ sends $(x_{\bullet}, \psi_{\bullet}) \in \mathcal{N}U^T_n(\mathcal{C})$ to
\[
(x_0 \rightarrow \cdots \rightarrow x_n \rightarrow T(x_0), (x_0 \rightarrow x_n \rightarrow x_{n^{\prime}} \rightarrow  x_{0^{\prime}} \rightarrow T(x_0), \psi_{\bullet}); (\mathbf{1}_{x_0},\mathbf{1}_{x_n})).
\]

To see that $\Psi_n$ is fully faithful, let $\Psi_n (x_{\bullet}, \psi_{\bullet}) \rightarrow \Psi_n (y_{\bullet}, \mu_{\bullet})$ be a morphism determined by maps $((p_0, \dots, p_n),(q_0,q_n))$, the notation an obvious modification of that from the proof of Proposition \ref{prop:catUnoriNerve}. Since the isomorphisms $\alpha_0, \alpha_n$ and $\beta_0, \beta_n$ used to define arbitrary morphisms in $\mathcal{N}C^T_n(\mathcal{C}) \times^{(2)}_{\mathcal{N}C^T_{\{0,n\}}(\mathcal{C})} \mathcal{N}U^T_{\{0,n\}}(\mathcal{C})$ are the identities in this case, the compatibility conditions (analogous to equation \eqref{eq:morphismCompat}) imply $p_0=q_0$ and $p_n=q_n$. Using the symmetry conditions imposed by the isometry condition, any morphism $\Psi_n (x_{\bullet}, \psi_{\bullet}) \rightarrow \Psi_n  (y_{\bullet}, \mu_{\bullet})$ is therefore the image of a unique morphism $(x_{\bullet}, \psi_{\bullet}) \rightarrow (y_{\bullet}, \mu_{\bullet})$.

To prove that $\Psi_n$ is essentially surjective, let
\begin{equation}
\label{eq:twFibProObj}
(x_0 \rightarrow \cdots \rightarrow x_n \rightarrow T(x_0), (x^{\prime}_0 \rightarrow x^{\prime}_n \rightarrow x^{\prime}_{n^{\prime}} \rightarrow x^{\prime}_{0^{\prime}} \rightarrow T(x^{\prime}_0), \psi_{\bullet}); (\alpha_0, \alpha_n))
\end{equation}
be an arbitrary object of $\mathcal{N}C^T_n(\mathcal{C}) \times^{(2)}_{\mathcal{N}C^T_{\{0,n\}}(\mathcal{C})} \mathcal{N}U^T_{\{0,n\}}(\mathcal{C})$. Define an object of $\mathcal{N}U^T_n(\mathcal{C})$ by the top row of the diagram
\[
\begin{tikzcd}[column sep = 0.4em]
x^{\prime}_0 \arrow{r} \arrow{d}[left]{\alpha_0^{-1}}  &  x_1 \arrow{r} & \cdots \arrow{r} & x_{n-1} \arrow{dr} \arrow{r} & x^{\prime}_{n} \arrow{r} & x^{\prime}_{n^{\prime}} \arrow{r} \arrow{d}[left]{\psi_{n^{\prime}}} & P(x_{n-1}) \arrow{r} & \cdots \arrow{r} & P(x_1) \arrow{r} \arrow{d} & x^{\prime}_{0^{\prime}} \arrow{r} & T(x_0^{\prime}) \\
x_0 \arrow{ur} & & & & x_n \arrow{u}[right]{\alpha_n} & P(x^{\prime}_n) \arrow{ur} & & & P(x_0^{\prime}) \arrow{ur}[right]{\psi_{0^{\prime}}^{-1}} & &
\end{tikzcd}
\]
with symmetric form $(\{\Theta_{\bullet}\}_{[1, n-1]}, \{\psi_{\bullet}\}_{\{0,n,n^{\prime},0^{\prime}\}}, \{\mathbf{1}_{\bullet}\}_{[(n-1)^{\prime}, 1^{\prime}]})$. The image of this object under $\Psi_n$ is
\[
\begin{tikzcd}[column sep = 0.5em]
(x^{\prime}_0 \arrow{r} &  x_1 \arrow{r} & \cdots \arrow{r} & x_{n-1} \arrow{r} & x^{\prime}_{n}  \arrow{r} & T(x_0^{\prime}),
\end{tikzcd} \hspace{-7pt}
\begin{tikzcd}[column sep = 0.5em]
(x^{\prime}_0 \arrow{r} & \cdots \arrow{r} & x^{\prime}_{0^{\prime}}  \arrow{r} & T(x_0^{\prime}),
\end{tikzcd}
\psi_{\bullet}); (\mathbf{1}_{x^{\prime}_0}, \mathbf{1}_{x^{\prime}_n})).
\]
Then the tuple $(\alpha_0^{-1}, \mathbf{1}_{x_1}, \dots, \mathbf{1}_{x_{n-1}}, \alpha_n^{-1}), (\mathbf{1}_{x^{\prime}_0}, \dots, \mathbf{1}_{x^{\prime}_{0^{\prime}}}))$ defines an isomorphism from the previous object to the object \eqref{eq:twFibProObj}. Hence $\Psi_n$ is an equivalence.

The proof that relative unitality holds is similar.
\end{proof}

\begin{Rem}
The homotopical counterpart of the $\mathsf{SO(2)}$-action on $LX$ is the cyclic structure of $NC_{\bullet}(\mathcal{C})$ in the sense of Connes. This action extends to $\mathsf{O(2)}$ by allowing orientation reversal of loops and corresponds to the dihedral structure of $NC_{\bullet}(\mathcal{C})$ induced by a (strict) duality $P$ on $\mathcal{C}$. More generally, $NC^T_{\bullet}(\mathcal{C})$ has a paradihedral or, if $T^r = 1$ for some $r \geq 2$, an $r$-dihedral structure. Similar comments apply to categorified nerves. See \cite[\S I.6]{dyckerhoff2015} for further discussion.
\end{Rem}

\subsection{The \texorpdfstring{$\mathcal{R}_{\bullet}$}{}-construction}

The goal of this section is to prove that the $\mathcal{R}_{\bullet}$-construction (also called the Hermitian $\mathcal{S}_{\bullet}$-construction) from Grothendieck-Witt theory is relative $2$-Segal over the Waldhausen $\mathcal{S}_{\bullet}$-construction. We work in the proto-exact setting of Section \ref{sec:waldhausen}.

\begin{Def}
A proto-exact category with duality is a category with duality $(\mathcal{C}, P, \Theta)$ such that $\mathcal{C}=(\mathcal{C}, \mathfrak{I}, \mathfrak{D})$ is proto-exact and $P$ is an exact functor. Explicitly, the latter condition comprises of the following:
\begin{enumerate}[label=(\roman*)]
\item $P(0) \simeq 0$,
\item a morphism $U \xrightarrow[]{\phi} V$ is in $\mathfrak{I}$ if and only if $P(V) \xrightarrow[]{P(\phi)} P(U)$ is in $\mathfrak{D}$, and
\item $P$ sends biCartesian squares to biCartesian squares.
\end{enumerate}
\end{Def}

Let $N$ be a symmetric form in a proto-exact category with duality $\mathcal{C}$ and let $i : U \rightarrowtail N$ be an inflation. The orthogonal $U^{\perp}$ is defined to be the pullback
\[
\begin{tikzpicture}[baseline= (a).base]
\node[scale=1] (a) at (0,0){
\begin{tikzcd}
U^{\perp} \arrow[dashed,two heads]{d} \arrow[dashed,tail]{r} & N \arrow[two heads]{d} \\
0 \arrow[tail]{r} & P(U).
\end{tikzcd}
};
\end{tikzpicture}
\]
The inflation $i$ is called isotropic if the composition $P(i) \psi_N i$ is zero and the canonical monomorphism $U \rightarrow U^{\perp}$ is an inflation.

\begin{Ex}
\hspace{2em}
\begin{enumerate}
\item Let $\mathsf{Rep}_k(Q)$ be the abelian category of representations of quiver $Q$. Given a contravariant involution of $Q$ together with some combinatorial data, there is an associated exact duality structure on $\mathsf{Rep}_k(Q)$ whose symmetric forms are (generalizations of) the orthogonal or symplectic quiver representations of Derksen and Weyman \cite{derksen2002}. See also \cite[\S 1.4]{mbyoung2016b}. A similar construction applies when $\mathsf{Rep}_k(Q)$ is replaced with $\mathsf{Rep}_{\mathbb{F}_1}(Q)$.

\item Let $\mathsf{Vect}(X)$ be the exact category of vector bundles over a scheme $X$. Fix a line bundle $L \rightarrow X$ and a sign $s \in \{ \pm 1\}$. Then $(P, \Theta) =(\Hom_{\mathcal{O}_X}(-, L), s \cdot \mathsf{can})$ defines an exact duality structure on $\mathsf{Vect}(X)$ whose symmetric forms are vector bundles over $X$ with $L$-valued orthogonal ($s=1$) or symplectic ($s=-1$) forms. \qedhere
\end{enumerate}
\end{Ex}

In this section we will make the following assumption on $(\mathcal{C}, P, \Theta)$.

\begin{Assu}
Let $N$ be a symmetric form in $\mathcal{C}$ and let $U \rightarrowtail N$ be isotropic with orthogonal $k: U^{\perp} \rightarrowtail N$. Define $M \in \mathcal{C}$ to be the pushout
\[
\begin{tikzpicture}[baseline= (a).base]
\node[scale=1] (a) at (0,0){
\begin{tikzcd}
U \arrow[two heads]{d} \arrow[tail]{r} & U^{\perp} \arrow[dashed,two heads]{d}{\pi} \\
0 \arrow[dashed,tail]{r} & M.
\end{tikzcd}
};
\end{tikzpicture}
\]
Then there is a unique symmetric form $\psi_M$ on $M$ such that $
P(k) \psi_N k = P(\pi) \psi_M \pi$.
\end{Assu}

The symmetric form $(M, \psi_M)$ is called the isotropic reduction of $N$ by $U$ and is denoted by $N \git U$. When $\mathcal{C}$ is exact the Reduction Assumption is known to hold; see \cite[Lemma 5.2]{quebbemann1979}, \cite[Lemma 2.6]{schlichting2010}. In a number of (non-exact) proto-exact examples, such as $\mathsf{Rep}_{\mathbb{F}_1}(Q)$, the Reduction Assumption also holds.

The category $[n]$ has a strict duality structure given at the level of objects by $i \mapsto n-i$. Then $[\mathsf{Ar}_n, \mathcal{C}]$ and its full subcategory $\mathcal{W}_n(\mathcal{C})$ inherit duality structures. Moreover, the duality structures on $\mathcal{W}_{\bullet}(\mathcal{C})$ satisfy the assumptions of Lemma \ref{lem:subdivDuality} so that $\mathcal{W}^e_{\bullet}(\mathcal{C})$ is a simplicial object of $\mathsf{CatD}$. Explicitly, the dual of a diagram $\{A_{\{p,q\}} \}_{0 \leq p \leq q \leq 0^{\prime}} \in \mathcal{W}^e_n(\mathcal{C})$ is the diagram $\{P(A_{\{q^{\prime}, p^{\prime}\}})\}_{0 \leq p \leq q \leq 0^{\prime}}$ and the double dual identification at $(p,q)$ is $\Theta_{A_{\{ p,q \}}}$.

The $\mathcal{R}_{\bullet}$-construction of $(\mathcal{C}, P, \Theta)$, denoted by $\mathcal{R}_{\bullet}(\mathcal{C})$, is the Hermitian groupoid $\mathcal{W}^e_{\bullet}(\mathcal{C})_h$; see \cite{shapiro1996}, \cite{hornbostel2004}.\footnote{In \cite{hornbostel2004} the notation $\mathcal{R}_{\bullet}(\mathcal{C})$ is used for $\mathcal{W}^e_{\bullet}(\mathcal{C})$, whereas what we call $\mathcal{R}_{\bullet}(\mathcal{C})$ is denoted by $\mathcal{R}^h_{\bullet}(\mathcal{C})$.} Objects of $\mathcal{R}_n(\mathcal{C})$ are diagrams $\{A_{\{ p,q \}} \}_{0 \leq p \leq q \leq 0^{\prime}} \in \mathcal{W}^e_n(\mathcal{C})$ together with symmetric isomorphisms $\psi_{p,q} : A_{\{p,q\}} \rightarrow P(A_{\{q^{\prime}, p^{\prime}\}})$ which make all appropriate diagrams commute. In particular,
\begin{enumerate}[label=(\roman*)]
\item for every $1 \leq i \leq n$ the pair $(A_{\{i, i^{\prime}\}}, \psi_{i,i^{\prime}})$ is a symmetric form,

\item for every $0 \leq i \leq j \leq n$ the inflation $A_{\{ i,j \}} \rightarrowtail A_{\{ i,i^{\prime} \}}$ is isotropic with orthogonal $A_{\{i, j^{\prime}\}} \rightarrowtail A_{\{ i,i^{\prime} \}}$, and

\item for every $0 \leq i \leq j \leq n$ the symmetric form $(A_{\{j, j^{\prime}\}}, \psi_{j,j^{\prime}})$ is isometric to the reduction $(A_{\{i,i^{\prime}\}} , \psi_{i,i^{\prime}}) \git A_{\{i,j\}}$.
\end{enumerate}
For example, an object of $\mathcal{R}_0(\mathcal{C})$ is a diagram
\[
\begin{tikzpicture}[baseline= (a).base]
\node[scale=1] (a) at (0,0){
\begin{tikzcd}
0 \arrow[rightarrowtail]{r} & A_{\{0,0^{\prime}\}} \arrow[two heads]{d} \\
& 0
\end{tikzcd}
};
\end{tikzpicture}
\]
together with a symmetric form on $A_{\{0, 0^{\prime} \}}$. Hence $\mathcal{R}_0(\mathcal{C})$ is equivalent to the Hermitian groupoid $\mathcal{C}_h$. Similarly, an object of $\mathcal{R}_1(\mathcal{C})$ is a diagram
\[
\begin{tikzpicture}[baseline= (a).base]
\node[scale=1] (a) at (0,0){
\begin{tikzcd}
0 \arrow[tail]{r} & A_{\{0,1\}} \arrow[two heads]{d} \arrow[tail]{r} & A_{\{0,1^{\prime}\}} \arrow[tail]{r} \arrow[two heads]{d} & A_{\{0,0^{\prime}\}} \arrow[two heads]{d} \\
 & 0 \arrow[tail]{r} & A_{\{1,1^{\prime}\}}  \arrow[two heads]{d} \arrow[tail]{r} & A_{\{1,0^{\prime}\}} \arrow[two heads]{d} \\
 & & 0 \arrow[tail]{r} & A_{\{1^{\prime},0^{\prime}\}} \arrow[two heads]{d} \\
 & & & 0
\end{tikzcd}
};
\end{tikzpicture}
\]
all of whose squares are biCartesian, together with the data of
\begin{enumerate}[label=(\roman*)]
\item a symmetric form on $A_{\{0,0^{\prime}\}}$ such that $A_{\{0,1 \}} \rightarrowtail A_{\{0,0^{\prime}\}}$ is isotropic with orthogonal $A_{\{0,1^{\prime}\}} \rightarrowtail A_{\{0,0^{\prime}\}}$,

\item a symmetric form on $A_{\{1,1^{\prime}\}}$ presenting $A_{\{1,1^{\prime}\}}$ as $A_{\{0,0^{\prime}\}} \git A_{\{ 0,1\}}$, and

\item isomorphisms $A_{\{0,1\}}  \simeq P(A_{\{1^{\prime},0^{\prime} \}})$ and $A_{\{0,1^{\prime} \}}  \simeq P(A_{\{1,0^{\prime} \}})$ such that each morphism above the diagonal agrees with the corresponding morphism below the diagonal.
\end{enumerate}
In short, objects of $\mathcal{R}_n(\mathcal{C})$ are isotropic $n$-flags together with presentations of all subquotients and subreductions. The forgetful map
\[
F_{\bullet}: \mathcal{R}_{\bullet}(\mathcal{C}) \rightarrow \mathcal{S}_{\bullet}(\mathcal{C}), \qquad \mathcal{R}_n(\mathcal{C}) \owns \{A_{\{ p,q \}} , \psi_{p,q} \}_{0 \leq p \leq q \leq 0^{\prime}} \mapsto \{ A_{\{p,q\}}\}_{0 \leq p \leq q \leq n}
\]
is a morphism of simplicial groupoids.

\begin{Thm}
\label{thm:sd2Segal}
Let $\mathcal{C}$ be a proto-exact category with duality which satisfies the Reduction Assumption. Then the morphism $F_{\bullet}: \mathcal{R}_{\bullet}(\mathcal{C}) \rightarrow \mathcal{S}_{\bullet}(\mathcal{C})$ is a unital relative $2$-Segal simplicial groupoid.
\end{Thm}

\begin{proof}
For each $n \geq 0$ let $\mathcal{I}_n(\mathcal{C})$ be the groupoid of isotropic $n$-flags in $\mathcal{C}$. An object of $\mathcal{I}_n(\mathcal{C})$ is a diagram
\begin{equation}
\label{eq:isoFlag}
0 \rightarrowtail A_1 \rightarrowtail  \cdots \rightarrowtail A_n \rightarrowtail A_{n^{\prime}} \rightarrowtail \cdots \rightarrowtail A_{1^{\prime}} \rightarrowtail (A_{0^{\prime}}, \psi_{0^{\prime}})
\end{equation}
with $(A_{0^{\prime}}, \psi_{0^{\prime}})$ a symmetric form such that, for each $0 \leq i \leq n$, the inflation $A_i \rightarrowtail A_{0^{\prime}}$ is isotropic with orthogonal $A_{i^{\prime}} \rightarrowtail A_{0^{\prime}}$. We claim that the forgetful functor
\[
\nu_n: \mathcal{R}_n (\mathcal{C}) \rightarrow \mathcal{I}_n (\mathcal{C}), \qquad \{A_{\{i,j\}}, \psi_{i,j}\}_{0 \leq i \leq j \leq 0^{\prime}} \mapsto \{A_{\{ 0,j \}}, \psi_{0,0^{\prime}}\}_{0 \leq j \leq 0^{\prime}}
\]
is an equivalence. Construct a quasi-inverse $\eta_n$ of $\nu_n$ as follows. Given an isotropic flag \eqref{eq:isoFlag}, put $A_{\{0,k\}} = A_k$, $1 \leq k \leq 1^{\prime}$, and $(A_{\{0,0^{\prime}\}}, \psi_{0,0^{\prime}})=(A_{0^{\prime}}, \psi_{0^{\prime}})$ and let $A_{\{1, k\}}$ be the pushout
\[
\begin{tikzpicture}[baseline= (a).base]
\node[scale=1] (a) at (0,0){
\begin{tikzcd}
A_{\{0, 1\}} \arrow[two heads]{d} \arrow[tail]{r} & A_{\{0,k\}} \arrow[dashed,two heads]{d} \\
0 \arrow[dashed,tail]{r} & A_{\{1,k\}}.
\end{tikzcd}
};
\end{tikzpicture}
\]
By the Reduction Assumption the symmetric form $(A_{\{0,0^{\prime}\}}, \psi_{0,0^{\prime}})$ induces a unique compatible symmetric form on $A_{\{1,1^{\prime}\}}$. Define also $A_{\{1,0^{\prime}\}} = P(A_{\{0, 1^{\prime}\}})$ and let $A_{\{0, 0^{\prime}\}} \twoheadrightarrow A_{\{1^{\prime}, 0^{\prime}\}}$ be the composition
\[
A_{\{0,0^{\prime}\}} \xrightarrow[]{\psi_{0,0^{\prime}}} P(A_{\{0,0^{\prime}\}}) \twoheadrightarrow P(A_{\{0, 1^{\prime}\}}).
\]
This construction defines the top two rows of $\eta_n(A_{\{0, \bullet\}})$ and can be iterated to define the remaining $n-2$ rows. It is clear that $\nu_n$ is a quasi-inverse to $\eta_n$.

We can now prove the theorem. The $1$-Segal morphism for $\mathcal{R}_{\bullet}(\mathcal{C})$ is
\[
\Xi_n : \mathcal{R}_n(\mathcal{C}) \rightarrow \mathcal{R}_{\{0, \dots, i\}}(\mathcal{C})
\times^{(2)}_{\mathcal{R}_{\{i\}}(\mathcal{C})} \mathcal{R}_{\{i, \dots, n\}}(\mathcal{C}).
\]
Arguing as in the proof of Theorem \ref{thm:stabFrameSegal} and using that $\nu_n$ is an equivalence, to prove that $\Xi_n$ is an equivalence it suffices to prove that the functor
\[
\widetilde{\Xi}_n: \mathcal{I}_n(\mathcal{C}) \rightarrow \mathcal{I}_i(\mathcal{C}) \times^{(2)}_{\mathcal{I}_{\{i\}}(\mathcal{C})} \mathcal{I}_{n-i}(\mathcal{C})
\]
is an equivalence. A quasi-inverse of $\widetilde{\Xi}_n$ is defined by assigning to a pair
\[
( 0 \rightarrowtail A^{\prime}_1 \rightarrowtail  \cdots \rightarrowtail A^{\prime}_i \rightarrowtail A^{\prime}_{i^{\prime}} \rightarrowtail \cdots \rightarrowtail A^{\prime}_{1^{\prime}} \rightarrowtail (A^{\prime}_{0^{\prime}}, \psi^{\prime}_{0^{\prime}}) ) \in \mathcal{I}_i(\mathcal{C})
\]
and
\[
( 0 \rightarrowtail A^{\prime \prime}_{i+1} \rightarrowtail  \cdots \rightarrowtail A^{\prime \prime}_n \rightarrowtail A^{\prime \prime}_{n^{\prime}} \rightarrowtail \cdots \rightarrowtail A^{\prime \prime}_{(i+1)^{\prime \prime}} \rightarrowtail (A^{\prime \prime}_{i^{\prime}}, \psi^{\prime \prime}_{i^{\prime}}) ) \in \mathcal{I}_{n-i}(\mathcal{C})
\]
together with an isometry $(A^{\prime}_{0^{\prime}}, \psi^{\prime}_{0^{\prime}}) \git A^{\prime}_i \simeq (A^{\prime \prime}_{i^{\prime}}, \psi^{\prime \prime}_{i^{\prime}})$ the object
\[
(
0 \rightarrowtail A^{\prime}_1 \rightarrowtail  \cdots \rightarrowtail A^{\prime}_n \rightarrowtail A^{\prime}_{n^{\prime}} \rightarrowtail \cdots \rightarrowtail A^{\prime}_{1^{\prime}} \rightarrowtail (A^{\prime}_{0^{\prime}}, \psi^{\prime}_{0^{\prime}}) )\in \mathcal{I}_n(\mathcal{C})
\]
where $A^{\prime}_j$, $i + 1 \leq j \leq (i+1)^{\prime}$, is defined to be the pullback
\[
\begin{tikzpicture}[baseline= (a).base]
\node[scale=1] (a) at (0,0){
\begin{tikzcd}
A^{\prime}_j \arrow[dashed,two heads]{d} \arrow[dashed,tail]{r} & A^{\prime}_{i^{\prime}} \arrow[two heads]{d} \\
A^{\prime \prime}_j \arrow[tail]{r} & A^{\prime \prime}_{i^{\prime}}.
\end{tikzcd}
};
\end{tikzpicture}
\]
That this is indeed a quasi-inverse follows from the Hermitian variant of the Second Isomorphism Theorem \cite[Proposition 6.5]{quebbemann1979}, which generalizes to the proto-exact setting under the Reduction Assumption.

Arguing in the same way, to prove that the functor
\[
\Psi_n: \mathcal{R}_n(\mathcal{C}) \rightarrow \mathcal{S}_n (\mathcal{C}) \times^{(2)}_{\mathcal{S}_{\{0,n\}}(\mathcal{C})} \mathcal{R}_{\{0, n\}}(\mathcal{C})
\]
is an equivalence it suffices to prove that the functor
\[
\widetilde{\Psi}_n : \mathcal{I}_n(\mathcal{C}) \rightarrow \mathcal{F}_n (\mathcal{C}) \times^{(2)}_{\mathcal{F}_{\{0,n\}}(\mathcal{C})} \mathcal{I}_{\{0, n\}}(\mathcal{C})
\]
is an equivalence, which is obvious.

Finally, relative unitality is the condition that, for each $0 \leq i \leq n-1$, the functor
\[
\Upsilon_n: \mathcal{R}_{n-1}(\mathcal{C}) \rightarrow \mathcal{S}_{\{i\}}(\mathcal{C}) \times^{(2)}_{\mathcal{S}_{\{i, i+1\}}(\mathcal{C})} \mathcal{R}_n(\mathcal{C})
\]
is an equivalence. Note that $\mathcal{S}_0(\mathcal{C})$ is a point, $\mathcal{S}_1(\mathcal{C})$ is the maximal groupoid of $\mathcal{C}$ and the map $\mathcal{S}_0(\mathcal{C}) \xrightarrow[]{s_0} \mathcal{S}_1(\mathcal{C})$ sends the point to the zero object. An object of $\mathcal{S}_{\{i\}}(\mathcal{C}) \times^{(2)}_{\mathcal{S}_{\{i, i+1\}}(\mathcal{C})} \mathcal{R}_n(\mathcal{C})$ is therefore an object of $\mathcal{R}_n(\mathcal{C})$ whose maps between the $i$th and $(i+1)$st rows/columns are the identities. It follows from this that $\Upsilon_n$ is an equivalence.
\end{proof}

\begin{Rem}
For comparison, the Waldhausen simplicial groupoid $\mathcal{S}_{\bullet}(\mathcal{C})$ of an exact category is $1$-Segal if and only if all short exact sequence in $\mathcal{C}$ are split.
\end{Rem}

When $\mathcal{C}$ is exact, the morphism $F_{\bullet}: \mathcal{R}_{\bullet}(\mathcal{C}) \rightarrow \mathcal{S}_{\bullet}(\mathcal{C})$ is closely related to the higher Grothendieck-Witt theory of $\mathcal{C}$. Indeed, $GW_i(\mathcal{C})$ is the $i$th homotopy group of the homotopy fibre over $0 \in \mathcal{C}$ of the map $\vert BF_{\bullet} \vert: \vert B\mathcal{R}_{\bullet}(\mathcal{C}) \vert  \rightarrow \vert B\mathcal{S}_{\bullet}(\mathcal{C}) \vert$ \cite{schlichting2010b}.

\section{Applications to Hall algebras}

\subsection{Categorical Hall algebra representations}
\label{sec:catHallAlg}

After reviewing the relationship between $2$-Segal spaces and categorical Hall algebras \cite{dyckerhoff2012b}, in this section we construct from a relative $2$-Segal space a categorical representation of the Hall algebra of the base.

Fix a combinatorial model category $\mathcal{C}$, such as $\mathbb{S}$ with its Kan model structure or $\mathsf{Grpd}$ with its Bousfield model structure; the latter will be the main source of examples. While the Quillen model structure on $\mathsf{Top}$ is not combinatorial, it is Quillen equivalent to $\mathbb{S}$ so that the results of this section, for all intents and purposes, also apply to simplicial spaces. We write $\pt \in \mathcal{C}$ for the final object.

Let $\mathsf{Span}(\mathcal{C})$ be the bicategory of spans in $\mathcal{C}$. Objects of $\mathsf{Span}(\mathcal{C})$ are simply objects of $\mathcal{C}$ while $1$-morphisms $A \rightarrow B$ in $\mathsf{Span}(\mathcal{C})$ are spans from $A$ to $B$, that is, diagrams in $\mathcal{C}$ of the form
\[
A \leftarrow X \rightarrow B.
\]
Composition of spans is given by homotopy pullback. A $2$-morphism in $\mathsf{Span}(\mathcal{C})$ is a homotopy commutative diagram
\[
\xymatrixrowsep{1.5em}
\xymatrixcolsep{1.5em}
\xymatrix{
& X \ar[rd] \ar[ld] \ar[dd] & \\
A & & B \\
& X^{\prime} \ar[ru] \ar[lu] &
}
\]
in $\mathcal{C}$. Give $\mathsf{Span}(\mathcal{C})$ the Cartesian monoidal structure.

Associated to the bicategory $\mathsf{Span}(\mathcal{C})$ is the category $\mathsf{Span}(\mathcal{C})^{\sim}$, having the same objects as $\mathsf{Span}(\mathcal{C})$ but with morphisms being the $2$-isomorphism classes spans. The constructions of this section are simplified if one uses $\mathsf{Span}(\mathcal{C})^{\sim}$ in place of $\mathsf{Span}(\mathcal{C})$, the downside being that less of the (relative) $2$-Segal structure is used.

\begin{Def}[{\cite[\S 8.1]{dyckerhoff2012b}}]
A transfer structure on $\mathcal{C}$ is a pair $(\mathcal{S}, \mathcal{P})$ consisting of collections of morphisms $\mathcal{S}$ and $\mathcal{P}$ in $\mathcal{C}$ called smooth and proper, respectively, which satisfy the following properties:
\begin{enumerate}
\item Both collections $\mathcal{S}$ and $\mathcal{P}$ are closed under composition.

\item Given a homotopy Cartesian diagram
\begin{equation}
\begin{gathered}
\label{eq:pushPullDiag}
\xymatrix{
X \ar[d]_{p} \ar[r]^{s} & Y \ar[d]^{p^{\prime}} \\
X^{\prime} \ar[r]_{s^{\prime}} & Y^{\prime}
}
\end{gathered}
\end{equation}
in $\mathcal{C}$ with $s^{\prime} \in \mathcal{S}$ and $p^{\prime} \in \mathcal{P}$, then $s \in \mathcal{S}$ and $p \in \mathcal{P}$.
\end{enumerate}
\end{Def}

\begin{Ex}
The pair $(\mathcal{S}, \mathcal{P})=(\mathcal{C},\mathcal{C})$ is called the trivial transfer structure on $\mathcal{C}$.
\end{Ex}

Fix a transfer structure $(\mathcal{S}, \mathcal{P})$. Spans of the form $A \xleftarrow[]{s} X \xrightarrow[]{p} B$ with $s \in \mathcal{S}$ and $p \in \mathcal{P}$ are called $(\mathcal{S}, \mathcal{P})$-admissible and form a subbicategory $\mathsf{Span}(\mathcal{S}, \mathcal{P}) \subset \mathsf{Span}(\mathcal{C})$.

\begin{Thm}[{\cite[Proposition 8.1.7]{dyckerhoff2012b}}]
\label{thm:univHallAlg}
Let $X_{\bullet}$ be a $2$-Segal object of $\mathcal{C}$. Assume that the span
\begin{equation}
\label{eq:multSpan}
m_X = \left\{ X_{\{ 0,1\}} \times X_{\{1,2\}} \xleftarrow[]{(\partial_2, \partial_0)} X_{\{ 0,1,2\}} \xrightarrow[]{\partial_1} X_{\{ 0,2 \}} \right\} \in \Hom_{\mathsf{Span}(\mathcal{C})}(X_1 \otimes X_1, X_1)
\end{equation}
is $(\mathcal{S}, \mathcal{P})$-admissible. Then $(X_1, m_X)$ is a semigroup in $\mathsf{Span}(\mathcal{S},\mathcal{P})$. Moreover, if $X_{\bullet}$ is unital and the span
\[
I_X = \left\{ \pt \xleftarrow[]{\mathsf{can}}  X_0 \xrightarrow[]{s_0} X_1 \right\} \in \Hom_{\mathsf{Span}(\mathcal{C})}(\pt, X_1)
\]
is $(\mathcal{S}, \mathcal{P})$-admissible, then $(X_1, m_X, I_X)$ is a monoid in $\mathsf{Span}(\mathcal{S},\mathcal{P})$.
\end{Thm}

The pair $\mathcal{H}(X_{\bullet})= (X_1, m_X)$ is called the $(\mathcal{S},\mathcal{P})$-universal Hall algebra of $X_{\bullet}$. Slightly abusively, the associativity isomorphism $a_X$ for $m_X$ is omitted from the notation. Note that at the universal level the transfer structure $(\mathcal{S},\mathcal{P})$ simply determines the subbicategory $\mathsf{Span}(\mathcal{S}, \mathcal{P}) \subset \mathsf{Span}(\mathcal{C})$ to which $\mathcal{H}(X_{\bullet})$ belongs. In particular, we can always take the trivial transfer structure. A more important role is played by the transfer structure when passing from the universal Hall algebra to its concrete realizations. To explain this procedure, fix a monoidal category $(\mathcal{V},\otimes, \mathbf{1}_{\mathcal{V}})$.

\begin{Def}[{\cite[\S 8.1]{dyckerhoff2012b}}]
A $\mathcal{V}$-valued theory with transfer on $\mathcal{C}$ is the data of
\begin{enumerate}
\item a transfer structure $(\mathcal{S}, \mathcal{P})$ on $\mathcal{C}$,
\item a contravariant functor $(-)^* : \mathcal{S} \rightarrow \mathcal{V}$ and a covariant functor $(-)_* : \mathcal{P} \rightarrow \mathcal{V}$ with common values on objects, denoted by $\mathfrak{h}$, and which map weak equivalences to isomorphisms, and

\item an isomorphism $\mathfrak{h}(\pt) \simeq \mathbf{1}_{\mathcal{V}}$ and multiplicativity data for $\mathfrak{h}$, that is, natural maps 
\[
\mathfrak{h}(X) \otimes \mathfrak{h}(Y) \rightarrow \mathfrak{h}(X \otimes Y)
\]
which satisfy associativity and unitality conditions
\end{enumerate}
such that for a homotopy Cartesian diagram of the form \eqref{eq:pushPullDiag} with $s, s^{\prime} \in \mathcal{S}$ and $p,p^{\prime} \in \mathcal{P}$, we have $
p^{\prime}_* \circ s^* = s^{\prime *} \circ p_*$.
\end{Def}

Applying a $\mathcal{V}$-valued theory with transfer $\mathfrak{h}$ to the $(\mathcal{S},\mathcal{P})$-universal Hall algebra $\mathcal{H}(X_{\bullet})$ gives a semigroup in $\mathcal{V}$, denoted by
\[
\mathcal{H}(X_{\bullet}; \mathfrak{h}) = (\mathfrak{h}(X_1), \partial_{1*} \circ (\partial_2 \times \partial_0)^*).
\]

Dually, if the opposite span $m^{\mathsf{op}}_X \in  \Hom_{\mathsf{Span}(\mathcal{C})}(X_1, X_1 \otimes X_1)$ is $(\mathcal{S}, \mathcal{P})$-admissible, then $(X_1,m^{\mathsf{op}}_X)$ is a cosemigroup in $\mathsf{Span}(\mathcal{S},\mathcal{P})$, called the $(\mathcal{S}, \mathcal{P})$-universal Hall coalgebra, and passing to theories with transfer gives cosemigroups in monoidal categories. These statements can be proved in the same way as Theorem \ref{thm:univHallAlg}.

\begin{Ex}
Let $\mathcal{S}_{\bullet}(\mathcal{C})$ be the Waldhausen groupoid of an essentially small exact category $\mathcal{C}$. Suppose that $\mathcal{C}$ is finitary in the sense that $\Ext_{\mathsf{env}(\mathcal{C})}^n(U,V)$, $n=0,1$, is finite for all $U,V \in \mathcal{C}$. Here $\mathsf{env}(\mathcal{C})$ denotes the abelian envelope of $\mathcal{C}$ \cite{quillen1973}. Then $\mathcal{H}(\mathcal{S}_{\bullet}(\mathcal{C}))$ categorifies the Hall algebra of $\mathcal{C}$, as defined in various contexts in \cite{ringel1990}, \cite{hubery2006}, \cite{schiffmann2012b}. To see this, let $k$ be a field of characteristic zero. Consider the transfer structure on $\mathsf{Grpd}$ in which $\mathcal{S}$ and $\mathcal{P}$ are the collections of weakly proper and locally proper morphisms, respectively (see \cite[\S 8.2]{dyckerhoff2012b}). A $\mathsf{Vect}_k$-valued theory with transfer $\mathfrak{F}_0$ on $\mathsf{Grpd}$ is then be defined by taking finitely supported $k$-valued functions which are constant on isomorphism classes. The resulting $k$-algebra $\mathcal{H}(\mathcal{S}_{\bullet}(\mathcal{C}); \mathfrak{F}_0)$ is the standard Hall algebra of $\mathcal{C}$. Explicitly, $\mathcal{H}(\mathcal{S}_{\bullet}(\mathcal{C}); \mathfrak{F}_0)$ is the $k$-vector space with basis $\{\mathbf{1}_U\}_{U \in \pi_0(\mathcal{S}_1(\mathcal{C}))}$ and multiplication
\[
\mathbf{1}_U \cdot \mathbf{1}_V = \sum_{W \in \pi_0(\mathcal{S}_1(\mathcal{C}))}  F^W_{U,V} \mathbf{1}_W,
\]
where $F^W_{U,V}$ is the number of admissible subobjects of $W$ which are isomorphic to $U$ and have quotient isomorphic to $V$.

For certain categories $\mathcal{C}$, such as $\mathsf{Coh}(X)$ for a smooth projective variety $X$ or $\mathsf{mod}(A)$ for a finitely generated algebra $A$, the $\mathcal{S}_{\bullet}$-construction defines a $2$-Segal simplicial Artin stack. This allows to recover the perverse sheaf theoretic \cite{lusztig1991}, motivic \cite{joyce2007}, \cite{kontsevich2008} and cohomological \cite{kontsevich2011} Hall algebras of $\mathcal{C}$. See \cite[\S 8.5]{dyckerhoff2012b}.
\end{Ex}

By using relative $2$-Segal spaces we can easily modify the above results to construct module objects.

\begin{Thm}
\label{thm:univHallMod}
Let $X_{\bullet}$ be as in Theorem \ref{thm:univHallAlg} and let $F_{\bullet} : Y_{\bullet} \rightarrow X_{\bullet}$ be a relative $2$-Segal object of $\mathcal{C}$. Assume that the left action span
\begin{equation}
\label{eq:leftModuleSpan}
\mu^l_Y = \left\{ X_{\{ 0,1\}} \times Y_{\{ 1 \}} \xleftarrow[]{(F_1 , \partial_0)} Y_{\{ 0,1\}} \xrightarrow[]{\partial_1}  Y_{\{ 1 \}}
\right\} \in \Hom_{\mathsf{Span}(\mathcal{C})}(X_1 \otimes Y_0, Y_0)
\end{equation}
is $(\mathcal{S}, \mathcal{P})$-admissible. Then $(Y_0, \mu^l_Y)$ is a left $(X_1, m_X)$-module in $\mathsf{Span}(\mathcal{S},\mathcal{P})$. Moreover, if $F_{\bullet}$ is unital, then $(Y_0, \mu^l_Y)$ is a left $(X_1, m_X,I_X)$-module. Analogous statements hold for the right action span
\[
\mu^r_Y = \left\{ Y_{\{ 0 \}} \times X_{\{ 0,1\}} \xleftarrow[]{\partial_1 \times F_1} Y_{\{ 0,1\}} \xrightarrow[]{\partial_0}  Y_{\{ 1 \}}
\right\} \in \Hom_{\mathsf{Span}(\mathcal{C})}(Y_0 \otimes X_1, Y_0).
\]
\end{Thm}

\begin{proof}
We will prove the theorem for the left action span. It follows directly from the definitions that we have
\[
\mu^l_Y \circ (m_X \otimes \mathbf{1}_{Y_0}) = \left\{
\begin{gathered}
\begin{tikzpicture}
  \matrix (m) [matrix of math nodes,row sep=1.3em,column sep=0.05em] {
    \left( X_{\{0,1,2\}} \times Y_{\{2\}} \right) \times^R_{X_{\{0,2\}} \times Y_{\{2\}}} Y_{\{0,2\}}  & Y_{\{0,2\}} & Y_{\{0 \}} \\
     X_{\{0,1,2\}}  \times Y_{\{2\}} & X_{\{0,2\}}  \times Y_{\{2\}} & \\
    X_{\{0,1\}} \times X_{\{1,2\}} \times Y_{\{2\}} & & \\};     
  \path[-stealth]
    (m-1-1) edge  (m-1-2);    
  \path[-stealth]
    (m-1-1) edge  (m-2-1);               
  \path[-stealth]
    (m-1-2) edge  (m-1-3);    
   \path[-stealth]
    (m-1-2) edge (m-2-2);   
      \path[-stealth]
    (m-2-1) edge (m-2-2);    
      \path[-stealth]
    (m-2-1) edge (m-3-1);
\end{tikzpicture}
\end{gathered}
\right\}
\]
and
\[
\mu^l_Y \circ ( \mathbf{1}_{X_1} \otimes \mu^l_Y) = 
\left\{
\begin{gathered}
\begin{tikzpicture}
  \matrix (m) [matrix of math nodes,row sep=1.3em,column sep=0.05em] {
    \left( X_{\{0,1\}} \times Y_{\{1,2\}} \right) \times^R_{X_{\{0,1\}} \times Y_{\{1\}}} Y_{\{0,1\}} & Y_{\{0,1\}} & Y_{\{0 \}} \\
     X_{\{0,1\}}  \times Y_{\{1,2\}} & X_{\{0,1\}}  \times Y_{\{1\}} & \\
    X_{\{0,1\}} \times X_{\{1,2\}} \times Y_{\{2\}} & & \\};     
  \path[-stealth]
    (m-1-1) edge  (m-1-2);    
  \path[-stealth]
    (m-1-1) edge  (m-2-1);               
  \path[-stealth]
    (m-1-2) edge  (m-1-3);    
   \path[-stealth]
    (m-1-2) edge (m-2-2);   
      \path[-stealth]
    (m-2-1) edge (m-2-2);    
      \path[-stealth]
    (m-2-1) edge (m-3-1);
\end{tikzpicture}
\end{gathered}
\right\}
\]
as spans $X_1 \otimes X_1 \otimes Y_0 \rightarrow Y_0$. We claim that both $\mu^l_Y \circ (m_X \otimes \mathbf{1}_{Y_0})$ and $\mu^l_Y \circ ( \mathbf{1}_{X_1} \otimes \mu^l_Y)$ are $2$-isomorphic to the span
\[
\sigma^l = \left\{ X_{\{ 0,1\}} \times X_{\{ 1,2 \}} \times Y_{\{ 2 \}} \xleftarrow[]{(F_{\{0,1\}} \circ \partial_2,  F_{\{1,2\}} \circ \partial_0 , \partial_0 \circ \partial_1)} Y_{\{ 0,1,2 \}} \xrightarrow[]{\partial_1 \circ \partial_2} Y_{\{ 0 \}}
\right\}.
\]
Indeed, the composition
\begin{eqnarray*}
Y_{\{0,1,2\}} & \xrightarrow[]{\alpha_1} & \left( X_{\{0,1,2\}} \times Y_{\{2\}} \right) \times^R_{X_{\{0,2\}} \times Y_{\{2\}}} Y_{\{0,2\}} \\
& \rightarrow & X_{\{0,1,2\}} \times^R_{X_{\{0,2\}} } Y_{\{0,2\}}
\end{eqnarray*}
is a weak equivalence by the relative $2$-Segal condition on $F_{\bullet}$. Since the second functor is a weak equivalence for trivial reasons, it follows that $\alpha_1$ is a weak equivalence and so defines a $2$-isomorphism $\sigma^l \rightarrow \mu^l_Y \circ (m_X \otimes \mathbf{1}_{Y_0})$. Similarly, the composition
\begin{eqnarray*}
Y_{\{0,1,2\}} & \xrightarrow[]{\alpha_2} & \left( X_{\{0,1\}} \times Y_{\{1,2\}} \right) \times^R_{X_{\{0,1\}} \times Y_{\{1\}}} Y_{\{0,1\}} \\
& \rightarrow & Y_{\{1,2\}}  \times^R_{ Y_{\{1\}}} Y_{\{0,1\}}
\end{eqnarray*}
is a weak equivalence by the $1$-Segal condition on $Y_{\bullet}$ and $\alpha_2$ defines a $2$-isomorphism $\sigma^l \rightarrow \mu^l_Y \circ ( \mathbf{1}_{X_1} \otimes \mu^l_Y)$. We claim that the composition
\[
\alpha^l_Y: \mu^l_Y \circ (m_X \otimes \mathbf{1}_{Y_0}) \xrightarrow[]{\alpha_1^{-1}} \sigma^l \xrightarrow[]{\alpha_2} \mu^l_Y \circ ( \mathbf{1}_{X_1} \otimes \mu^l_Y)
\]
is a module associator for the left action of $(X_1,m_X)$ on $Y_0$ determined by $\mu^l_Y$. Without the unital assumption, this amounts to verifying that $\alpha^l_Y$ satisfies module theoretic Mac Lane coherence; see \cite[\S 2.3]{ostrik2003} or diagram \eqref{eq:macCoher} below. We will verify this using the setting of Section \ref{sec:symmSubdivisions}. The poset of the 11 symmetric polyhedral subdivisions of the symmetric octagon $P_{\mathbf{3}}$, ordered by refinement, is illustrated in Figure \ref{fig:macCoher}. At the level of the map $F_{\bullet}: Y_{\bullet} \rightarrow X_{\bullet}$, each node $\mathcal{P}$ of Figure \ref{fig:macCoher} defines a span
\[
X_1 \times X_1 \times X_1 \times Y_0 \leftarrow (\Delta^{\mathcal{P}}, F_{\bullet})_R \rightarrow Y_0.
\]
Similarly, each arrow of Figure \ref{fig:macCoher} defines a $2$-isomorphism of spans. The spans associated to the five vertices of the pentagon are precisely the spans appearing in the diagram expressing module theoretic Mac Lane coherence while the composed $2$-isomorphisms along the edges of the pentagon are precisely the arrows in the coherence diagram. It follows that module theoretic Mac Lane coherence holds. Hence $(Y_0,\mu^l_Y)$ is a left $(X_1,m_X)$-module.

\begin{figure}
\begin{center}
\[
\begin{tikzpicture}[scale=0.95]
\node (origin) at (0,0) (origin) {$\begin{tikzpicture}
\node[draw,minimum size=1cm,regular polygon,regular polygon sides=8] (a) {};
\end{tikzpicture}$}; 

\node (P0) at (1*18:3.25cm) { $
\begin{tikzpicture}
\node[draw,minimum size=1cm,regular polygon,regular polygon sides=8] (a) {};

\path[-,font=\scriptsize,color=red]
    (a.corner 2) edge[-]  (a.corner 5);  
\path[-,font=\scriptsize,color=red]
    (a.corner 1) edge[-]  (a.corner 6);    
\path[-,font=\scriptsize,color=red]
    (a.corner 2) edge[-]  (a.corner 4);
\path[-,font=\scriptsize,color=red]
    (a.corner 1) edge[-]  (a.corner 7);
\end{tikzpicture}
$}; 

\node (P2) at (1*18+1*72:3.25cm) {$
\begin{tikzpicture}
\node[draw,minimum size=1cm,regular polygon,regular polygon sides=8] (a) {};

\path[-,font=\scriptsize,color=red]
    (a.corner 1) edge[-]  (a.corner 6);   
\path[-,font=\scriptsize,color=red]
    (a.corner 2) edge[-]  (a.corner 5);
\path[-,font=\scriptsize,color=red]
    (a.corner 3) edge[-]  (a.corner 5);
\path[-,font=\scriptsize,color=red]
    (a.corner 8) edge[-]  (a.corner 6);
\end{tikzpicture}
$};

 \node (P4) at (1*18+2*72:3.25cm) {$
 \begin{tikzpicture}
\node[draw,minimum size=1cm,regular polygon,regular polygon sides=8] (a) {};

\path[-,font=\scriptsize,color=red]
    (a.corner 3) edge[-]  (a.corner 8);
\path[-,font=\scriptsize,color=red]
    (a.corner 3) edge[-]  (a.corner 5);  
\path[-,font=\scriptsize,color=red]
    (a.corner 8) edge[-]  (a.corner 6);
\end{tikzpicture}
$}; 

 \node (P1) at (1*18+3*72:3.25cm) {$
\begin{tikzpicture}
\node[draw,minimum size=1cm,regular polygon,regular polygon sides=8] (a) {};

\path[-,font=\scriptsize,color=red]
    (a.corner 3) edge[-]  (a.corner 8);
\path[-,font=\scriptsize,color=red]
    (a.corner 4) edge[-]  (a.corner 7);
\end{tikzpicture}
$}; 

\node (P3) at (1*18+4*72:3.25cm) {$
\begin{tikzpicture}
\node[draw,minimum size=1cm,regular polygon,regular polygon sides=8] (a) {};

\path[-,font=\scriptsize,color=red]
    (a.corner 4) edge[-]  (a.corner 7);   
\path[-,font=\scriptsize,color=red]
    (a.corner 2) edge[-]  (a.corner 4);    
\path[-,font=\scriptsize,color=red]
    (a.corner 1) edge[-]  (a.corner 7);
\end{tikzpicture}
$}; 

  \path (P0) -- (P2)
  node [midway] (P02) {$
\begin{tikzpicture}
\node[draw,minimum size=1cm,regular polygon,regular polygon sides=8] (a) {};

\path[-,font=\scriptsize,color=red]
    (a.corner 2) edge[-]  (a.corner 5); 
\path[-,font=\scriptsize,color=red]
    (a.corner 1) edge[-]  (a.corner 6);     
\end{tikzpicture}
$}; 

  \path (P2) -- (P4) node [midway] (P24) {$
\begin{tikzpicture}
\node[draw,minimum size=1cm,regular polygon,regular polygon sides=8] (a) {};
\path[-,font=\scriptsize,color=red]
    (a.corner 3) edge[-]  (a.corner 5);
\path[-,font=\scriptsize,color=red]
    (a.corner 6) edge[-]  (a.corner 8); 
\end{tikzpicture}
$}; 

  \path (P4) -- (P1) node [midway] (P14) {$
  \begin{tikzpicture}
\node[draw,minimum size=1cm,regular polygon,regular polygon sides=8] (a) {};

\path[-,font=\scriptsize,color=red]
    (a.corner 3) edge[-]  (a.corner 8);
\end{tikzpicture}
$ };

\path (P1) -- (P3) node [midway] (P13) {$
\begin{tikzpicture}
\node[draw,minimum size=1cm,regular polygon,regular polygon sides=8] (a) {};

\path[-,font=\scriptsize,color=red]
    (a.corner 4) edge[-]  (a.corner 7); 
\end{tikzpicture}
$};

\path (P0) -- (P3) node [midway] (P03) {$
\begin{tikzpicture}
\node[draw,,minimum size=1cm,regular polygon,regular polygon sides=8] (a) {};

\path[-,font=\scriptsize,color=red]
    (a.corner 2) edge[-]  (a.corner 4);    
\path[-,font=\scriptsize,color=red]
    (a.corner 1) edge[-]  (a.corner 7);
\end{tikzpicture}
$}; 
\draw [-latex, thick] (P0) -- (P02);
\draw [-latex, thick] (P2) -- (P02);
\draw [-latex, thick] (P0) -- (origin);
\draw [-latex, thick] (P1) -- (origin);
\draw [-latex, thick] (P2) -- (origin);
\draw [-latex, thick] (P3) -- (origin);
\draw [-latex, thick] (P4) -- (origin);
\draw [-latex, thick] (P02) -- (origin);
\draw [-latex, thick] (P24) -- (origin);
\draw [-latex, thick] (P14) -- (origin);
\draw [-latex, thick] (P13) -- (origin);
\draw [-latex, thick] (P03) -- (origin);
\draw [-latex, thick] (P2) -- (P24); 
\draw [-latex, thick] (P4) -- (P24); 
\draw [-latex, thick] (P4) -- (P14); 
\draw [-latex, thick] (P1) -- (P14);
\draw [-latex, thick] (P1) -- (P13); 
\draw [-latex, thick] (P3) -- (P13);  
\draw [-latex, thick] (P0) -- (P03); 
\draw [-latex, thick] (P3) -- (P03); 
\end{tikzpicture}
\]
\caption{The poset of symmetric polyhedral subdivisions of $P_{\mathbf{3}}$.}
\label{fig:macCoher}
\end{center}
\end{figure}
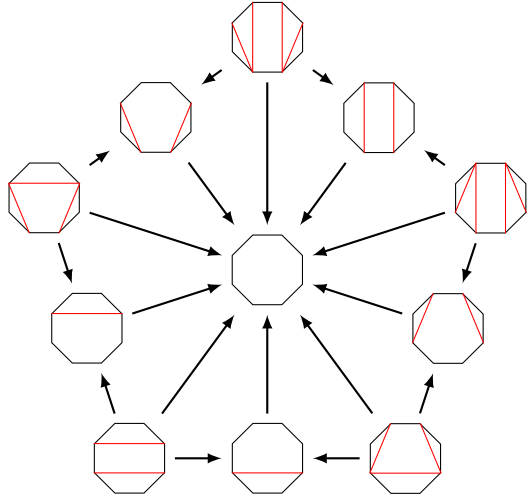

Finally, in the unital setting the action of $I_X$ on $Y_0$ is given by the span
\[
\begin{tikzpicture}
  \matrix (m) [matrix of math nodes,row sep=1.5em,column sep=1.5em,minimum width=1.5em] {
    \left( X_{\{0\}} \times Y_{\{ 1 \}} \right)  \times^R_{X_{\{0,1\}} \times Y_{\{1\}}} Y_{\{0,1\}} & Y_{\{0,1 \}} & Y_{\{0 \}} \\
     X_{\{0\}}  \times Y_{\{1\}} & X_{\{0,1\}}  \times Y_{\{1\}} & \\
    \pt \times Y_{\{1\}} & & \\}; 
  \path[-stealth]
    (m-1-1) edge  (m-1-2);    
  \path[-stealth]
    (m-1-1) edge  (m-2-1);               
  \path[-stealth]
    (m-1-2) edge  (m-1-3);    
  \path[-stealth]
    (m-1-2) edge (m-2-2);    
   \path[-stealth]
    (m-2-1) edge (m-2-2);    
   \path[-stealth]
    (m-2-1) edge (m-3-1);
\end{tikzpicture}
\]
which is $2$-isomorphic to the identity span $Y_0 \xleftarrow[]{\mathbf{1}_{Y_0}} Y_0 \xrightarrow[]{\mathbf{1}_{Y_0}} Y_0$. Indeed, the composition of functors
\begin{eqnarray*}
Y_{\{0\}} & \xrightarrow[]{I_Y} & \left( X_{\{0\}} \times Y_{\{ 1 \}} \right)  \times^R_{X_{\{0,1\}} \times Y_{\{1\}}} Y_{\{0,1\}} \\
& \rightarrow & X_{\{0\}}   \times^R_{X_{\{0,1\}} } Y_{\{0,1\}}
\end{eqnarray*}
is an equivalence by the unital relative $2$-Segal condition while the second functor is trivially an equivalence. To complete the proof we need to verify that $I_Y$ is compatible with the module associator in the sense that diagram \eqref{eq:unitCompat} below commutes. This is a straightforward exercise which can be completed in much the same way as the above verification of Mac Lane coherence.
\end{proof}

The left $\mathcal{H}(X_{\bullet})$-module $\mathcal{M}(Y_{\bullet})=(Y_0,\mu^l_Y)$ is called the $(\mathcal{S},\mathcal{P})$-universal left Hall module of $F_{\bullet}$. From a $\mathcal{V}$-valued theory with transfer $\mathfrak{h}$ we obtain a left $\mathcal{H}(X_{\bullet} ; \mathfrak{h})$-module $\mathcal{M}(Y_{\bullet} ; \mathfrak{h})$ in $\mathcal{V}$. In the same way we get right modules over $\mathcal{H}(X_{\bullet})$ and $\mathcal{H}(X_{\bullet};\mathfrak{h})$. Note that the left and right module structures do not define a bimodule.

\begin{Ex}
Let $X_{\bullet}$ be $2$-Segal object of $\mathcal{C}$. By Proposition \ref{prop:relPathSpace} the right path space $F^{\rhd}_{\bullet}: P^{\rhd}X_{\bullet} \rightarrow X_{\bullet}$ is relative $2$-Segal. Since $F^{\rhd}_n = \partial_{n^{\prime}}$, the spans $\mu^l_{P^{\rhd}X}$ and $m_X$ are equal. Hence $\mathcal{M}(X^{\rhd}_{\bullet}) = \mathcal{H}(X_{\bullet})$ as left $\mathcal{H}(X_{\bullet})$-modules. On the other hand, the right $\mathcal{H}(X_{\bullet})$-module structure on $\mathcal{M}(X^{\rhd}_{\bullet})$ is closely related to the coproduct $\Delta$ on $\mathcal{H}(X_{\bullet})$. Consider for example $\mathcal{H}(\mathcal{S}_{\bullet}(\mathcal{C}); \mathfrak{F}_0)$ for a finitary exact category $\mathcal{C}$. Recall that the Green bilinear form on $\mathcal{H}(\mathcal{S}_{\bullet}(\mathcal{C}); \mathfrak{F}_0)$, defined by
\[
(\mathbf{1}_U, \mathbf{1}_V) = \frac{\delta_{U,V}}{\vert \Aut(U) \vert},
\]
satisfies the Hopf property $(\mathbf{1}_U \otimes \mathbf{1}_V, \Delta \mathbf{1}_W) = (\mathbf{1}_U \cdot \mathbf{1}_V, \mathbf{1}_W)$. Then the right and left $\mathcal{H}(\mathcal{S}_{\bullet}(\mathcal{C}); \mathfrak{F}_0)$-actions are adjoint with respect to $(-,-)$.
\end{Ex}

\begin{Rems}
\hspace{2em}
\begin{enumerate}
\item Generalizing the previous observation, it is proved in \cite[Proposition 5.6.10]{walde2016} that for a relative $2$-Segal simplicial groupoid $Y_{\bullet} \rightarrow X_{\bullet}$, the right and left $\mathcal{H}(X_{\bullet})$-module structures of $\mathcal{M}(Y_{\bullet})$ are related via a categorical Green bilinear form.

\item The universal Hall algebra $\mathcal{H}(X_{\bullet})$ is itself a $\mathcal{H}(X_{\bullet})$-bimodule which, however, does not arise from a relative $2$-Segal space over $X_{\bullet}$.
\end{enumerate}
\end{Rems}

\begin{Ex}
Let $\mathcal{S}^{\mathsf{st} \mhyphen \mathsf{fr}}_{\bullet}(\mathcal{C}^{Z \mhyphen \mathsf{ss}}_{\phi}) \rightarrow \mathcal{S}_{\bullet}(\mathcal{C}^{Z \mhyphen \mathsf{ss}}_{\phi})$ be the relative $2$-Segal groupoid associated to a stability function $Z$ and a framing $\Phi$ on an abelian category $\mathcal{C}$. The universal left $\mathcal{H}(\mathcal{S}_{\bullet}(\mathcal{C}^{Z \mhyphen \mathsf{ss}}_{\phi}))$-module $\mathcal{M}(\mathcal{S}^{\mathsf{st} \mhyphen \mathsf{fr}}_{\bullet}(\mathcal{C}^{Z \mhyphen \mathsf{ss}}_{\phi}))$ categorifies the stable framed Hall algebra representations of \cite{soibelman2016}, \cite{franzen2016} which appear in framed Donaldson-Thomas theory.
\end{Ex}

\begin{Ex}
Let $\mathcal{R}_{\bullet}(\mathcal{C}) \rightarrow \mathcal{S}_{\bullet}(\mathcal{C})$ be the relative $2$-Segal groupoid associated to a proto-exact category with duality via the $\mathcal{R}_{\bullet}$-construction. For finitary exact $\mathcal{C}$, the universal left $\mathcal{H}(\mathcal{S}_{\bullet}(\mathcal{C}))$-module $\mathcal{M}(\mathcal{R}_{\bullet}(\mathcal{C}))$ categorifies the Hall algebra representations of \cite{mbyoung2016}. Explicitly, $\mathcal{M}(\mathcal{R}_{\bullet}(\mathcal{C}); \mathfrak{F}_0)$ is the $k$-vector space with basis $\{\mathbf{1}_{(M,\psi_M)}\}_{(M,\psi_M) \in \pi_0(\mathcal{R}_0(\mathcal{C}))}$ and left $\mathcal{H}(\mathcal{S}_{\bullet}(\mathcal{C}); \mathfrak{F}_0)$-module structure
\[
\mathbf{1}_U \star \mathbf{1}_{(M,\psi_M)} = \sum_{(N,\psi_N) \in \pi_0(\mathcal{R}_0(\mathcal{C}))} G^N_{U,M} \mathbf{1}_{(N,\psi_N)}.
\]
The structure constant $G^N_{U,M}$ is the number of isotropic subobjects of $N$ which are isomorphic to $U$ and have reduction isometric to $M$.

Similarly, the simplicial stack version of $\mathcal{R}_{\bullet}(\mathcal{C})$ is relative $2$-Segal over $\mathcal{S}_{\bullet}(\mathcal{C})$ and recovers the perverse sheaf theoretic \cite{enomoto2009}, motivic and cohomological \cite{mbyoung2016b} Hall algebra representations which appear in the representation theory of quantum enveloping algebras and orientifold Donaldson-Thomas theory.
\end{Ex}

\subsection{Hall monoidal module categories}

\label{sec:hallMonMod}

We briefly describe a variant of the constructions of Section \ref{sec:catHallAlg}. Let $k$ be a field of characteristic zero. Let $X_{\bullet}$ be a unital $2$-Segal groupoid\footnote{With minor modifications $X_{\bullet}$ could be a simplicial object of a combinatorial model category.} which is admissible with respect to weakly proper and locally proper maps. Write $\mathsf{Fun}_0(X_1)$ for the (abelian) category of finitely supported functors $X_1 \rightarrow  \mathsf{Vect}_k$. The span $m_X$ from Theorem \ref{thm:univHallAlg} induces a bifunctor
\[
\otimes^X = \partial_{1*} \circ (\partial_2 \times \partial_0)^* : \mathsf{Fun}_0(X_1) \otimes \mathsf{Fun}_0(X_1) \rightarrow \mathsf{Fun}_0(X_1).
\]
It is proved in \cite[Theorem 2.49]{dyckerhoff2015b} that the triple $\mathcal{H}^{\otimes}(X_{\bullet}) = (\mathsf{Fun}_0(X_1), \otimes^X, I_X)$ is a monoidal category. More generally, each componentwise cocontinuous monoidal left derivator of groupoids $\mathbb{D}$ determines a monoidal category structure on $\mathbb{D}(X_1)$; see \cite[Theorem 5.0.1.(1)]{walde2016}. The example $\mathcal{H}^{\otimes}(X_{\bullet})$ arises for $\mathbb{D}=\mathsf{Fun}_0(-, \mathsf{Vect}_k)$.

In a similar way, we have the following relative construction. The same statement is proved in \cite[Proposition 5.2.6.(2)]{walde2016}, as a special case of \cite[Theorem 5.0.1.(2)]{walde2016}.

\begin{Thm}
\label{thm:hallMonModCat}
Assume that $F_{\bullet} : Y_{\bullet} \rightarrow X_{\bullet}$ is an admissible unital relative $2$-Segal groupoid. Then the left action span $\mu^l_Y$ defines a bifunctor
\[
\otimes^{Y}= \partial_{1*} \circ (F_1 \times \partial_0)^* : \mathsf{Fun}_0(X_1) \otimes \mathsf{Fun}_0(Y_0) \rightarrow \mathsf{Fun}_0(Y_0)
\]
which gives $\mathcal{M}^{\otimes}(Y_{\bullet}) = (\mathsf{Fun}_0(Y_0), \otimes^{Y})$ the structure of a left $\mathcal{H}^{\otimes}(X_{\bullet})$-module category in the sense of \cite{ostrik2003}.
\end{Thm}

\begin{proof}
The proof is very similar to that Theorem \ref{thm:univHallMod} and is therefore omitted.
\end{proof}

We give two related instances of Theorem \ref{thm:hallMonModCat}.

\begin{Ex}
Let $\mathsf{Vect}_{\mathbb{F}_1}$ be the proto-exact category of finite dimensional vector spaces over $\mathbb{F}_1$. Objects of $\mathsf{Vect}_{\mathbb{F}_1}$ are finite pointed sets and morphisms are partial bijections. Let $\mathsf{Vect}^{\mathsf{sk}}_{\mathbb{F}_1} \subset \mathsf{Vect}_{\mathbb{F}_1}$ be the skeleton of standard ordinals. Writing $\mathfrak{S}_n$ for the symmetric group on $n$ letters, we have an equivalence of groupoids
\[
\mathcal{S}_1(\mathsf{Vect}^{\mathsf{sk}}_{\mathbb{F}_1}) \simeq \bigsqcup_{n \geq 0} B \mathfrak{S}_n.
\]
Objects of $\mathcal{H}^{\otimes}(\mathcal{S}_{\bullet}(\mathsf{Vect}^{\mathsf{sk}}_{\mathbb{F}_1}))$ (henceforth denoted by $\mathcal{H}^{\otimes}(\mathsf{Vect}^{\mathsf{sk}}_{\mathbb{F}_1})$) are thus sequences of finite dimensional representations of symmetric groups over $k$, only finitely many of which are non-trivial. The monoidal product is induction of representations. Using the results of \cite{macdonald1980} it follows that $\mathcal{H}^{\otimes}(\mathsf{Vect}^{\mathsf{sk}}_{\mathbb{F}_1})$ is equivalent to the category $\mathbf{P}$ of polynomial functors $\mathsf{Vect}_k \rightarrow \mathsf{Vect}_k$ \cite[\S 2.5.1]{dyckerhoff2015b}.

The functor $P=\Hom_{\mathsf{Vect}_{\mathbb{F}_1}}(-, \{*,1\})$ defines a strict exact duality structure on $\mathsf{Vect}_{\mathbb{F}_1}$. Note that each object of $\mathsf{Vect}_{\mathbb{F}_1}$ is canonically isomorphic to its dual. In particular, $P$ preserves $\mathsf{Vect}^{\mathsf{sk}}_{\mathbb{F}_1}$. A symmetric form on $\mathbb{F}_1^n$ is an element $\pi \in \mathfrak{S}_n$ which squares to the identity; conjugate such elements determine isometric symmetric forms. It follows that symmetric forms on $\mathbb{F}_1^n$ are determined uniquely by their Witt index $0 \leq w \leq \lfloor \frac{n}{2} \rfloor$, the number of $2$-cycles of any representative $\pi$. The isometry group of $\pi$, which is its centralizer in $\mathfrak{S}_n$, is isomorphic to $(\mathbb{Z}_2 \wr \mathfrak{S}_w) \times \mathfrak{S}_{n-2w}$. We therefore obtain an equivalence of groupoids
\[
\mathcal{R}_0(\mathsf{Vect}^{\mathsf{sk}}_{\mathbb{F}_1}) \simeq \bigsqcup_{w \geq 0} \bigsqcup_{d \geq 0} B ((\mathbb{Z}_2 \wr \mathfrak{S}_w) \times \mathfrak{S}_d)
\]
and we see that objects of $\mathcal{M}^{\otimes}(\mathsf{Vect}^{\mathsf{sk}}_{\mathbb{F}_1})$ are finite sequences of representations of groups of the form $(\mathbb{Z}_2 \wr \mathfrak{S}_w) \times \mathfrak{S}_d$. The left $\mathcal{H}^{\otimes}(\mathsf{Vect}^{\mathsf{sk}}_{\mathbb{F}_1})$-action on $\mathcal{M}^{\otimes}(\mathsf{Vect}^{\mathsf{sk}}_{\mathbb{F}_1})$ is induction of representations along subgroups of the form
\[
\mathfrak{S}_n \times \left( (\mathbb{Z}_2 \wr \mathfrak{S}_w) \times \mathfrak{S}_d \right) \leq (\mathbb{Z}_2 \wr \mathfrak{S}_{n+w}) \times \mathfrak{S}_d.
\]
From this description we see that
\[
\mathcal{M}^{\otimes}(\mathsf{Vect}^{\mathsf{sk}}_{\mathbb{F}_1}) = \bigoplus_{d=0}^{\infty} \mathcal{M}^{\otimes}(\mathsf{Vect}^{\mathsf{sk}}_{\mathbb{F}_1};d)
\]
as $\mathcal{H}^{\otimes}(\mathsf{Vect}^{\mathsf{sk}}_{\mathbb{F}_1})$-modules, the index $d$ labelling the fixed difference between the dimension and twice the Witt index. Moreover, we have
\[
\mathcal{M}^{\otimes} (\mathsf{Vect}^{\mathsf{sk}}_{\mathbb{F}_1};d) \simeq \mathcal{M}^{\otimes} (\mathsf{Vect}^{\mathsf{sk}}_{\mathbb{F}_1};0) \times \mathbf{P}^d
\]
where $\mathbf{P}^d \subset \mathbf{P}$ is the full subcategory of degree $d$ homogeneous polynomial functors. Using again results of \cite{macdonald1980} we find that $\mathcal{M}^{\otimes}(\mathsf{Vect}^{\mathsf{sk}}_{\mathbb{F}_1};0)$ is equivalent to $\mathbf{P}_{\mathbb{Z}_2}$, the category of polynomial functors $\mathsf{Vect}^{\mathbb{Z}_2 \mhyphen \mathsf{gr}}_k \rightarrow \mathsf{Vect}_k$, where $\mathsf{Vect}^{\mathbb{Z}_2 \mhyphen \mathsf{gr}}_k$ denotes the category of $\mathbb{Z}_2$-graded finite dimensional vector spaces over $k$ and we view $\mathbf{P}_{\mathbb{Z}_2}$ as a $\mathbf{P}$-module category via the forgetful functor $\mathsf{Vect}_k^{\mathbb{Z}_2 \mhyphen \mathsf{gr}} \rightarrow \mathsf{Vect}_k$.

Upon passing to Grothendieck groups we obtain an isomorphism
\[
K_0(\mathcal{M}^{\otimes}(\mathsf{Vect}^{\mathsf{sk}}_{\mathbb{F}_1};d)) \simeq \left( \bigoplus_{w=0}^{\infty} R_k(\mathbb{Z}_2 \wr \mathfrak{S}_w)  \right) \otimes_{\mathbb{Z}} R_k(\mathfrak{S}_d)
\]
as modules over the algebra
\[
K_0(\mathcal{H}^{\otimes}(\mathsf{Vect}^{\mathsf{sk}}_{\mathbb{F}_1})) \simeq \bigoplus_{n=0}^{\infty} R_k(\mathfrak{S}_n),
\]
where $R_k(-)$ denotes the representation ring over $k$. In the case $k= \mathbb{C}$ these modules have been studied in \cite{shelley2016} and are closely related to the work of Zelevinsky \cite{zelevinsky1981}, who studied the algebra structure on $K_0(\mathcal{M}^{\otimes}(\mathsf{Vect}^{\mathsf{sk}}_{\mathbb{F}_1};0))$ arising from induction of representations of wreath symmetric products.
\end{Ex}

We now give a sort of quantization of the previous example.

\begin{Ex}
Let $q$ be a prime power and consider the exact category $\mathsf{Vect}_{\mathbb{F}_q}$. Let $\mathsf{Vect}^{\mathsf{sk}}_{\mathbb{F}_q} \subset \mathsf{Vect}_{\mathbb{F}_q}$ be the skeleton consisting of the vector spaces $\mathbb{F}_q^n$, $n \geq 0$. We have an equivalence of groupoids
\[
S_1(\mathsf{Vect}^{\mathsf{sk}}_{\mathbb{F}_q}) \simeq \bigsqcup_{n \geq 0} B \mathsf{GL}_n(\mathbb{F}_q).
\]
The category $\mathcal{H}^{\otimes}(\mathsf{Vect}^{\mathsf{sk}}_{\mathbb{F}_q})$, whose monoidal product is parabolic induction of finite dimensional representations of $\mathsf{GL}$, has appeared in the work of Joyal and Street \cite{joyal1995}. When $k= \mathbb{C}$ the associated algebra $K_0(\mathcal{H}^{\otimes}(\mathsf{Vect}^{\mathsf{sk}}_{\mathbb{F}_q}))$, which is the complex representation ring of the tower of general linear groups over $\mathbb{F}_q$, has been studied by Green \cite{green1955} and Zelevinsky \cite{zelevinsky1981}.

Assume now that $q$ is odd and fix a sign $s \in \{\pm 1\}$. Take the exact duality structure with $P=\Hom_{\mathsf{Vect}_{\mathbb{F}_q}}(-, \mathbb{F}_q)$ with $\Theta = s \cdot \mathsf{can}$. Symmetric forms in $\mathsf{Vect}_{\mathbb{F}_q}$ are orthogonal or symplectic vector spaces. Identify the dual of $\mathbb{F}_q^n$ with itself via the dual basis. It follows that we have an equivalence of groupoids
\[
R_0(\mathsf{Vect}^{\mathsf{sk}}_{\mathbb{F}_q}) \simeq \bigsqcup_{n \geq 0} \bigsqcup_{\varepsilon \in \mathsf{W}_n} B \mathsf{G}^{\varepsilon}_n
\]
where $\mathsf{W}_n$ is the Witt group of $\mathbb{F}^n_q$ ($\mathbb{Z}_2$ if $s=1$ and trivial if $s=-1$) and $\mathsf{G}^{\varepsilon}_n = \mathsf{O}^{\varepsilon}_n(\mathbb{F}_q)$ if $s=+1$ and $\mathsf{G}_n = \mathsf{Sp}_{2n}(\mathbb{F}_q)$ if $s=-1$. The left $\mathcal{H}^{\otimes}(\mathsf{Vect}^{\mathsf{sk}}_{\mathbb{F}_q})$-action on $\mathcal{M}^{\otimes}(\mathsf{Vect}^{\mathsf{sk}}_{\mathbb{F}_q})$ is given by parabolic induction between $\mathsf{GL}$ and $\mathsf{G}$ representations. When $k= \mathbb{C}$ the $K_0(\mathcal{H}^{\otimes}(\mathsf{Vect}^{\mathsf{sk}}_{\mathbb{F}_q}))$-modules $K_0(\mathcal{M}^{\otimes}(\mathsf{Vect}^{\mathsf{sk}}_{\mathbb{F}_q}))$ were studied by van Leeuwen  \cite{leeuwen1991}, who showed that they are generated by cuspidal elements with respect to the natural Hall comodule structure.

It is natural to regard the module $\mathcal{M}^{\otimes}(\mathsf{Vect}^{\mathsf{sk}}_{\mathbb{F}_q})$ as a sort of $q$-analogue of the modules $\mathcal{M}^{\otimes}(\mathsf{Vect}^{\mathsf{sk}}_{\mathbb{F}_1};d)$, $d =0,1$, consistent with the philosophy that the $q \rightarrow 1$ limit of $\mathsf{G}(\mathbb{F}_q)$ is its Weyl group, namely $\mathbb{Z}_2 \wr \mathfrak{S}_n$ for the symplectic $\mathsf{Sp}_{2n}$ and even orthogonal $\mathsf{O}_{2n}$ groups and $(\mathbb{Z}_2 \wr \mathfrak{S}_n) \times \mathbb{Z}_2$ for the odd orthogonal group $\mathsf{O}_{2n+1}$. The extra factor of $\mathbb{Z}_2$ for $\mathsf{O}_{2n+1}$ reflects the fact that the subcategories of $\mathcal{M}^{\otimes}(\mathsf{Vect}^{\mathsf{sk}}_{\mathbb{F}_q})$ with fixed Witt type $\varepsilon$ are isomorphic as left $\mathcal{H}^{\otimes}(\mathsf{Vect}^{\mathsf{sk}}_{\mathbb{F}_q})$-modules, the same statement holding also at $q =1$. Working instead with Ringel-style Hall algebras and modules allows for the following more precise statement, which can be verified directly: the $q=1$ specialization of the $\mathcal{H}(\mathsf{Vect}^{\mathsf{sk}}_{\mathbb{F}_q}; \mathfrak{F}_0)$-module $\mathcal{M}(\mathsf{Vect}^{\mathsf{sk}}_{\mathbb{F}_q}; \mathfrak{F}_0)$ is isomorphic to the direct sum of the $\mathcal{H}(\mathsf{Vect}^{\mathsf{sk}}_{\mathbb{F}_1}; \mathfrak{F}_0)$-modules $\mathcal{M}(\mathsf{Vect}^{\mathsf{sk}}_{\mathbb{F}_1}; \mathfrak{F}_0)$, $d=0,1$.
\end{Ex}

\subsection{Modules over multivalued categories}

\label{sec:multiCat}

In this section and the next we describe two higher categorical interpretations of relative $2$-Segal simplicial sets; the first is in terms of modules over multivalued categories while the second is of a Hall algebraic nature.

Let $2 \mhyphen \mathsf{Seg} \mathbb{S} \subset \mathbb{S}$ be the full subcategory of unital $2$-Segal simplicial sets. Let also $\mu \mathsf{Cat}$ denote the category of small multivalued categories. Objects of $\mu \mathsf{Cat}$ are tuples $\mathfrak{X} = (\mathfrak{X}_0,\mathfrak{X}_1, m_{\mathfrak{X}}, a_{\mathfrak{X}}, e_{\mathfrak{X}}, i^l_{\mathfrak{X}}, i^r_{\mathfrak{X}})$ consisting of sets $\mathfrak{X}_0$ and $\mathfrak{X}_1$ with source and target maps $\partial_1, \partial_0: \mathfrak{X}_1 \rightarrow \mathfrak{X}_0$, a composition law $m_{\mathfrak{X}} \in \Hom_{\mathsf{Span}(\mathsf{Set})}(\mathfrak{X}_1 \times_{\mathfrak{X}_0} \mathfrak{X}_1, \mathfrak{X}_1)$, an associator isomorphism
\[
(m_{\mathfrak{X}} \circ (m_{\mathfrak{X}} \times \mathbf{1}_{\mathfrak{X}_1}) \xrightarrow[]{a_{\mathfrak{X}}} m_{\mathfrak{X}} \circ (\mathbf{1}_{\mathfrak{X}_1} \times m_{\mathfrak{X}}) ) \in \Hom_{\mathsf{Span}(\mathsf{Set})}(\mathfrak{X}_1 \times_{\mathfrak{X}_0} \mathfrak{X}_1 \times_{\mathfrak{X}_0} \mathfrak{X}_1, \mathfrak{X}_1)
\]
which satisfies Mac Lane coherence, a unit map $e_{\mathfrak{X}}: \mathfrak{X}_0 \rightarrow \mathfrak{X}_1$ and compatible left and right unit isomorphisms $i^l_{\mathfrak{X}}, i^r_{\mathfrak{X}} \in \Hom_{\mathsf{Span}(\mathsf{Set})}(\mathfrak{X}_1, \mathfrak{X}_1)$. A morphism $\varphi: \mathfrak{X} \rightarrow \mathfrak{X}^{\prime}$ is the data of maps $\varphi_i : \mathfrak{X}_i \rightarrow \mathfrak{X}_i^{\prime}$, $i=0,1$, compatible with the source, target and unit maps and a morphism of spans $\tilde{\varphi}_2: \varphi_1 \circ m_{\mathfrak{X}} \rightarrow m_{\mathfrak{X}^{\prime}} \circ (\varphi_1 \times_{\varphi_0} \varphi_1)$.

\begin{Thm}[{\cite[Theorem 3.3.6]{dyckerhoff2012b}}]
\label{thm:multiCat}
The generalized nerve construction defines an equivalence of categories $\mathsf{N}_{\bullet}: \mu \mathsf{Cat} \xrightarrow[]{\sim} 2\mhyphen \mathsf{Seg}\mathbb{S}$.
\end{Thm}

In the relative setting we will require the notion of a module over a multivalued category.

\begin{Def}
Let $\mathfrak{X}$ be a multivalued category. A unital left $\mathfrak{X}$-module is a tuple $\mathfrak{Y} =(\mathfrak{Y}_0, \mathfrak{F}_0, \mu_{\mathfrak{Y}}, \alpha_{\mathfrak{Y}}, \iota_{\mathfrak{Y}})$ consisting of
\begin{enumerate}[label=(\roman*)]
\item a set $\mathfrak{Y}_0$ together with a map $\mathfrak{F}_0: \mathfrak{Y}_0 \rightarrow \mathfrak{X}_0$,

\item a left action span $\mu_{\mathfrak{Y}} \in \Hom_{\mathsf{Span}(\mathsf{Set})}( \mathfrak{X}_1 \times_{\mathfrak{X}_0} \mathfrak{Y}_0, \mathfrak{Y}_0)$,

\item a module associator isomorphism
\[
( \mu_{\mathfrak{Y}} \circ (m_{\mathfrak{X}} \times \mathbf{1}_{\mathfrak{Y}_0}) \xrightarrow[]{\alpha_{\mathfrak{Y}}} \mu_{\mathfrak{Y}} \circ (\mathbf{1}_{\mathfrak{X}_1} \times \mu_{\mathfrak{Y}}))
 \in \Hom_{\mathsf{Span}(\mathsf{Set})}(\mathfrak{X}_1 \times_{\mathfrak{X}_0} \mathfrak{X}_1 \times_{\mathfrak{X}_0} \mathfrak{Y}_0, \mathfrak{Y}_0)
\]
which satisfies module theoretic Mac Lane coherence: the diagram
\begin{equation}
\label{eq:macCoher}
\begin{gathered}
\xymatrixrowsep{0.35in}
\xymatrixcolsep{0.02in}
\footnotesize
\xymatrix{
& \mu \circ ( m \circ ( m \times \mathbf{1}_{\mathfrak{X}_1}) \times \mathbf{1}_{\mathfrak{Y}_0})
\ar[dr]^{\mu \circ (a \times \mathbf{1}_{\mathfrak{Y}_0})}
\ar[dl]_{\alpha \circ (\mu \times \mathbf{1}_{\mathfrak{X}_1} \times \mathbf{1}_{\mathfrak{Y}_0})}
& \\
\mu \circ (m \times \mu) \ar[d]_{\alpha \circ (\mathbf{1}_{\mathfrak{X}_1} \times \mathbf{1}_{\mathfrak{X}_1} \times \mu)}& & \mu \circ (( m \circ (\mathbf{1}_{\mathfrak{X}_1} \times m) \times \mathbf{1}_{\mathfrak{Y}_0}) \ar[d]^{\alpha \circ (\mathbf{1}_{\mathfrak{X}_1} \times m \times \mathbf{1}_{\mathfrak{Y}_0})} \\
 \mu \circ ( \mathbf{1}_{\mathfrak{X}_1} \times \mu \circ ( \mathbf{1}_{\mathfrak{X}_1} \times \mu))  && \mu \circ (\mathbf{1}_{\mathfrak{X}_1} \times (\mu \circ (m \circ \mathbf{1}_{\mathfrak{Y}_0}))) \ar[ll]^{\mu \circ (\mathbf{1}_{\mathfrak{X}_1} \times \alpha)}
}
\end{gathered}
\end{equation}
commutes in $\Hom_{\mathsf{Span}(\mathsf{Set})}(\mathfrak{X}_1 \times_{\mathfrak{X}_0} \mathfrak{X}_1 \times_{\mathfrak{X}_0} \mathfrak{X}_1 \times_{\mathfrak{X}_0} \mathfrak{Y}_0, \mathfrak{Y}_0)$, and
\item a unit isomorphism $\iota_{\mathfrak{Y}}: \mu_{\mathfrak{Y}} \circ (e_{\mathfrak{X}} \times \mathbf{1}_{\mathfrak{Y}_0}) \rightarrow \mathbf{1}_{\mathfrak{Y}_0}$ in $\Hom_{\mathsf{Span}(\mathsf{Set})}(\mathfrak{Y}_0, \mathfrak{Y}_0)$ for which the diagram
\begin{equation}
\label{eq:unitCompat}
\begin{gathered}
\begin{tikzcd}[column sep=3.5em]
\mu \circ (m \circ (\mathbf{1}_{\mathfrak{X}_1} \times e) \times \mathbf{1}_{\mathfrak{Y}_0}) \arrow{rr}{\alpha \circ (\mathbf{1}_{\mathfrak{X}_1} \times e \times \mathbf{1}_{\mathfrak{Y}_0})} \arrow[swap]{rd}{\mu \circ (i^l_{\mathfrak{X}} \times \mathbf{1}_{\mathfrak{Y}_0})} &  & \mu \circ (\mathbf{1}_{\mathfrak{X}_1} \times \mu \circ ( e \times \mathbf{1}_{\mathfrak{Y}_0})) \arrow{ld}{\mu \circ (\mathbf{1}_{\mathfrak{X}_1} \times \iota_{\mathfrak{Y}})} \\
&\mu & 
\end{tikzcd}
\end{gathered}
\end{equation}
commutes in $\Hom_{\mathsf{Span}(\mathsf{Set})}( \mathfrak{X}_1 \times_{\mathfrak{X}_0} \mathfrak{Y}_0, \mathfrak{Y}_0)$.
\end{enumerate}
\end{Def}

Define a category $\mu \mathsf{Cat} \mhyphen \mathsf{mod}$ as follows. Objects are unital left modules over multivalued categories. A morphism $(\varphi, \vartheta): (\mathfrak{Y} \rightarrow \mathfrak{X}) \rightarrow (\mathfrak{Y}^{\prime} \rightarrow \mathfrak{X}^{\prime})$ consists of a morphism $\varphi: \mathfrak{X} \rightarrow \mathfrak{X}^{\prime}$, a map $\vartheta_0: \mathfrak{Y}_0 \rightarrow \mathfrak{Y}_0^{\prime}$ which satisfies $\mathfrak{F}^{\prime}_0 \circ \vartheta_0 = \varphi_0 \circ \mathfrak{F}_0$ and a morphism of spans
\[
\tilde{\vartheta}_1: \vartheta_0 \circ \mu_{\mathfrak{Y}} \rightarrow \mu_{\mathfrak{Y}^{\prime}} \circ (\varphi_1 \times_{\varphi_0} \vartheta_0).
\]
Informally, $\vartheta$ defines a morphism of $\mathfrak{X}$-modules $\mathfrak{Y} \rightarrow \varphi^* \mathfrak{Y}^{\prime}$.  

Let $2 \mhyphen \mathsf{SegRel}\mathbb{S} \subset \mathbb{S}^{[1]}$ be the full subcategory of unital relative $2$-Segal simplicial sets.

\begin{Thm}
\label{thm:multiCatModule}
The equivalence $\mathsf{N}_{\bullet}: \mu \mathsf{Cat} \xrightarrow[]{\sim} 2\mhyphen \mathsf{Seg}\mathbb{S}$ lifts to an equivalence $\mathsf{N}^{\mathsf{rel}}_{\bullet}: \mu \mathsf{Cat} \mhyphen \mathsf{mod} \xrightarrow[]{\sim} 2\mhyphen \mathsf{SegRel}\mathbb{S}$.
\end{Thm}

\begin{proof}
We begin by modifying the proof of \cite[Theorem 3.3.6]{dyckerhoff2012b} so as to define a functor $2 \mhyphen \mathsf{SegRel}\mathbb{S} \rightarrow \mu\mathsf{Cat} \mhyphen \mathsf{mod}$. Let $F_{\bullet}: Y_{\bullet} \rightarrow X_{\bullet}$ be a unital relative $2$-Segal simplicial set. The multivalued category $\mathfrak{X}$ associated to $X_{\bullet}$ via Theorem \ref{thm:multiCat} has $\mathfrak{X}_0= X_0$ and $\mathfrak{X}_1 = X_1$ with the canonical maps $\partial_1, \partial_0: \mathfrak{X}_1 \rightarrow \mathfrak{X}_0$ and composition span
\[
m_{\mathfrak{X}}= \{ \mathfrak{X}_1 \times_{\mathfrak{X}_0} \mathfrak{X}_1 \xleftarrow[]{(\partial_2,\partial_0)} X_2 \xrightarrow[]{\partial_1} \mathfrak{X}_1 \}.
\]
Turning to the relative data, let $\mathfrak{Y}_0 = Y_0$ and $\mathfrak{F}_0 = F_0$. Define an action span by
\[
\mu_{\mathfrak{Y}} = \{ \mathfrak{X}_1 \times_{\mathfrak{X}_0} \mathfrak{Y}_0 \xleftarrow{(F_1, \partial_0)} Y_1 \xrightarrow[]{\partial_1} \mathfrak{Y}_0 \}.
\]
The morphisms of spans
\[
\mu_{\mathfrak{Y}} \circ (m_{\mathfrak{X}} \times \mathbf{1}_{\mathfrak{Y}_0}) \leftarrow \{ \mathfrak{X}_1 \times_{\mathfrak{X}_0} \mathfrak{X}_1 \times_{\mathfrak{X}_0} \mathfrak{Y}_0 \leftarrow Y_2 \rightarrow \mathfrak{Y}_0 \} \rightarrow \mu_{\mathfrak{Y}} \circ (\mathbf{1}_{\mathfrak{X}_1} \times \mu_{\mathfrak{Y}}),
\]
each of which is constructed as in the proof of Theorem \ref{thm:univHallMod} and is an isomorphism by the relative $2$-Segal conditions, combine to define the associator $\alpha_{\mathfrak{Y}}$. The unit isomorphism $\iota_{\mathfrak{Y}}$ is defined to be the inverse of the relative unit bijection $Y_{\{0\}} \rightarrow X_{\{0\}} \times_{X_{\{0,1\}} } Y_{\{0,1\}}$. Mac Lane coherence and unit compatibility are verified as in the proof of Theorem \ref{thm:univHallMod}. Hence $\mathfrak{Y}$ is a unital left $\mathfrak{X}$-module. At the level of morphisms the functor $2 \mhyphen \mathsf{SegRel}\mathbb{S} \rightarrow \mu\mathsf{Cat} \mhyphen \mathsf{mod}$ is defined in the obvious way.

To describe a quasi-inverse $\mathsf{N}^{\mathsf{rel}}_{\bullet}: \mu\mathsf{Cat} \mhyphen \mathsf{mod} \rightarrow 2 \mhyphen \mathsf{SegRel}\mathbb{S}$, let $\mathfrak{X} \in \mu \mathsf{Cat}$ with associated unital $2$-Segal simplicial set $X_{\bullet}=\mathsf{N}_{\bullet}(\mathfrak{X})$. Given a unital left $\mathfrak{X}$-module $\mathfrak{Y}$, let $Y_0= \mathfrak{Y}_0$ and let $Y_1$ be the middle set of the span $\mu_{\mathfrak{Y}}$. We have canonical maps $\partial_i: Y_1 \rightarrow Y_0$ and $F_i: Y_i \rightarrow X_i$, $i=0,1$, and, by inverting the map which defines $\iota_{\mathfrak{Y}}$, a map $s_0: Y_0 \rightarrow Y_1$. Moreover, these maps obey the $1$-truncated simplicial identities. For each $n \geq 2$ define
\[
\widetilde{Y}_n=
\underleftarrow{\lim}^{\mathsf{Set}}_{\sigma \in \mathcal{S}^{\mathsf{max}}_{P_{\mathbf{n}}}} F_{P_{\mathbf{n}}}(\sigma)
\]
where $\mathcal{S}^{\mathsf{max}}_{P_{\mathbf{n}}}$ is the category of symmetric embeddings of the form $\Delta^k \sqcup \Delta^k \hookrightarrow \Delta^{P_{\mathbf{n}}}$, $k \leq 2$, or $\Delta^{\mathbf{m}} \hookrightarrow \Delta^{P_{\mathbf{n}}}$, $m \leq 1$ (\textit{cf}. Section \ref{sec:symmSubdivisions}). Given $0 \leq i < j < k \leq n$ and an element $\tilde{y} \in \widetilde{Y}_n$, write $y_{ij}$ and $y_{ijk}$ for the corresponding elements of $Y_{\{i,j\}}$ and $X_{\{i,j,k\}}$, respectively. The associator $\alpha_{\mathfrak{Y}}$ defines a collection of bijections
\[
\alpha_{ijk} : X_{\{i,j,k\}} \times_{X_{\{i,k\}}} Y_{\{i,k\}} \rightarrow Y_{\{i,j\}} \times_{Y_{\{j\}}} Y_{\{j,k\}}
\]
which we use to define
\[
Y_n = \{ \tilde{y} \in \widetilde{Y}_n \mid \alpha_{ijk}(y_{ijk}, y_{ik}) =  (y_{ij}, y_{jk}) \mbox{ for all } 0 \leq i < j < k \leq n\}.
\]
Module theoretic Mac Lane coherence implies that the $Y_n$ assemble to a $1$-Segal simplicial set $Y_{\bullet}$ and that the canonical maps $F_n : Y_n \rightarrow X_n$ assemble to a simplicial morphism $F_{\bullet}$. To see that $F_{\bullet}$ satisfies the remaining relative $2$-Segal conditions, consider the commutative diagram
\[
\begin{tikzcd}
Y_n \arrow{r} \arrow{d} & X_n \times_{X_{\{0,n\}}} Y_{\{0,n\}} \arrow{d}  \\
Y_1 \times_{Y_0}  \dots \times_{Y_0} Y_1 \arrow{r} & X_{\{0,1,2\}} \times_{X_{\{0,2\}}} \cdots \times_{X_{\{0,n-1\}}} X_{\{0,n-1,n\}} \times_{X_{\{0,n\}}} Y_{\{0,n\}}.
\end{tikzcd}
\]
The vertical morphisms are bijections by the $1$- and $2$-Segal conditions on $Y_{\bullet}$ and $X_{\bullet}$, respectively. The bottom arrow is the iterated application of the module associator, relating the bracketings in which $n$ elements of $\mathfrak{X}$ act on $\mathfrak{Y}$ from right to left and from left to right, and is therefore a bijection. Hence $F_{\bullet}$ is relative $2$-Segal. The relative unit bijection $Y_{\{0\}} \rightarrow X_{\{0\}} \times_{X_{\{0,1\}} } Y_{\{0,1\}}$ is the inverse of $\iota_{\mathfrak{Y}}$. The compatibility of $\iota_{\mathfrak{Y}}$, $i^l_{\mathfrak{X}}$ and $\alpha_{\mathfrak{Y}}$ implies that the higher relative unit bijections hold. The functor $\mathsf{N}_{\bullet}^{\mathsf{rel}}$ assigns to a morphism $(\varphi, \vartheta)$ the pair $(\phi_{\bullet}, \theta_{\bullet})$, where $\phi_{\bullet}=\mathsf{N}_{\bullet}(\varphi)$, $\theta_0$ is equal to $\tilde{\vartheta}_0$ and $\theta_n$, $n \geq 1$, are the canonically induced maps, which are well-defined by the definition of $Y_n$ and $Y_n^{\prime}$, $n \geq 2$.
\end{proof}

\begin{Rem}
Theorem \ref{thm:multiCatModule} could also have been formulated in terms of right modules.
\end{Rem}

There is a semi-simplicial variant of Theorem \ref{thm:multiCatModule} where $\mu \mathsf{Cat}\mhyphen \mathsf{mod}$ is replaced with $\frac{1}{2}\mu \mathsf{Cat}\mhyphen \mathsf{mod}$, the category of left modules over multivalued semicategories, the prefix `semi' indicating that all data related to unit morphisms is omitted. As a simple special case, let $X_{\bullet}$ be a $2$-Segal semi-simplicial set with $X_0 = X_1 = \pt$. Then $X_{\bullet}$  defines a distributive monoidal endofunctor on $\mathsf{Set}$ by $F\otimes F^{\prime}= X_2 \times F \times F^{\prime}$. As explained in \cite[\S 3.7]{dyckerhoff2012b}, the associator reduces to a bijection $a: X_2 \times X_2 \rightarrow X_2 \times X_2$ which satisfies the pentagon equation
\[
a_{23} \circ a_{13} \circ a_{12} = a_{12} \circ a_{23}.
\]
Conversely, a bijective solution $a$ of the pentagon equation on a set $X_2$ extends to a $2$-Segal semi-simplicial set with $\mathsf{N}_0(X_2,a) = \mathsf{N}_1(X_2,a) =\pt$ and $\mathsf{N}_n(X_2,a)$, $n \geq 2$, the set of tuples $\{x_{ijk} \in X_2 \mid 0 \leq i < j < k \leq n \}$ which satisfy
\[
a(x_{ijk},x_{ikl}) =(x_{ijl},x_{jkl}), \qquad 0 \leq i < j < k < l \leq n.
\]
Modifying this construction, a relative $2$-Segal semi-simplicial set $Y_{\bullet} \rightarrow X_{\bullet}$ with $Y_0= \pt$ defines a monoidal endofunctor $F \boxtimes G = Y_1 \times F \times G$ which is a left $\otimes$-module. The module associator is a bijection $\alpha : X_2 \times Y_1 \rightarrow Y_1 \times Y_1$ which satisfies the $a$-pentagon equation
\[
\alpha_{23} \circ \alpha_{13} \circ a_{12} = \alpha_{12} \circ \alpha_{23}.
\]
Moreover, from a bijective solution to the $a$-pentagon equation we can construct a relative $2$-Segal semi-simplicial set with $\mathsf{N}_0(Y_1,\alpha)= \pt$, $\mathsf{N}_1(Y_1,\alpha)= Y_1$ and $\mathsf{N}_n(Y_1,\alpha)$, $n \geq 2$,  the subset of $\mathsf{N}_n(X_2,a) \times \{ y_{ij} \in Y_1 \mid 0 \leq i < j \leq n\}$ consisting of tuples which satisfy
\[
\alpha(x_{ijk},y_{ik}) =(y_{ij}, y_{jk}), \qquad 0 \leq i < j < k \leq n.
\]
The structure map $\mathsf{N}_{\bullet}(Y_1,\alpha) \rightarrow \mathsf{N}_{\bullet}(X_2, a)$ is the canonical projection.

\begin{Ex}
For any group $\mathsf{G}$ the map
\[
a: \mathsf{G} \times \mathsf{G} \rightarrow \mathsf{G} \times \mathsf{G}, \qquad  (x,y) \mapsto (xy,y)
\]
is a bijective solution to the pentagon equation and so defines a $2$-Segal semi-simplicial set $\mathsf{N}_{\bullet}(\mathsf{G})$; see \cite{kashaev1998}, \cite[Example 3.7.7]{dyckerhoff2012b}.

Let now $\rho: \mathsf{G} \rightarrow \Aut(M)$ be a left $\mathsf{G}$-action on a set $M$. Then
\[
\alpha: \mathsf{G} \times M \rightarrow M \times M, \qquad (x,m) \mapsto (\rho(x)m, m)
\]
solves the $a$-pentagon equation. Moreover, $\alpha$ is a bijection if and only if $\rho$ gives $M$ the structure of a $\mathsf{G}$-torsor. Hence, associated to each $\mathsf{G}$-torsor $M$ is a relative $2$-Segal simplicial set $\mathsf{N}_{\bullet}(\mathsf{G},M) \rightarrow \mathsf{N}_{\bullet}(\mathsf{G})$. When $M =\mathsf{G}$ with $\mathsf{G}$ acting by left multiplication the above construction recovers the left path $P^{\lhd} \mathsf{N}_{\bullet}(\mathsf{G}) \rightarrow \mathsf{N}_{\bullet}(\mathsf{G})$.
\end{Ex}

\subsection{Presheaves on Hall \texorpdfstring{$2$}{}-categories}

We begin by constructing a module over a Hall-type algebra from a  relative $2$-Segal semi-simplicial set. Fix a field $k$. Let $X_{\bullet}$ be a $2$-Segal simplicial set for which the map $(\partial_2, \partial_0)$ from the span \eqref{eq:multSpan} has finite fibres. The Hall category $H(X_{\bullet})$ \cite[\S 3.4]{dyckerhoff2012b} is the $k$-linear category with $\text{Ob}(H(X_{\bullet}))=X_0$ and $\Hom_{H(X_{\bullet})}(a,b) = \mathfrak{F}_0(X_{a \rightarrow b})$, the finitely supported $k$-valued functions on the set
\[
X_{a \rightarrow b} = \{a\} \times_{X_0} X_1 \times_{X_0} \{b\}.
\]
Composition of morphisms is defined by push-pull along the span
\begin{equation}
\label{eq:ssSpan}
X_{b \rightarrow c} \times X_{a \rightarrow b} \leftarrow \{p \in X_2 \mid \partial_{\{0\}} p= a, \; \partial_{\{1\}} p=b, \; \partial_{\{2\}} p= c \} \rightarrow X_{a \rightarrow c}.
\end{equation}
Writing $\mathbf{1}_x$ fo the characteristic function of $x \in X_{a \rightarrow b}$, composition in $H(X_{\bullet})$ becomes $\mathbf{1}_x \cdot \mathbf{1}_{x^{\prime}} = \sum_{x^{\prime \prime}} f_{x, x^{\prime}}^{x^{\prime \prime}} \mathbf{1}_{x^{\prime \prime}}$ where
\[
f_{x, x^{\prime}}^{x^{\prime \prime}}= \vert \{ p \in X_2 \mid \partial_2 p = x, \; \partial_1 p =x^{\prime}, \; \partial_0 p =x^{\prime \prime}\} \vert.
\] 
Similarly, a relative $2$-Segal simplicial set $F_{\bullet} : Y_{\bullet} \rightarrow X_{\bullet}$ for which the map $(F_1, \partial_0)$ from the span \eqref{eq:leftModuleSpan} has finite fibres defines a presheaf $\mathcal{F} :H(X_{\bullet})^{\mathsf{op}} \rightarrow \mathsf{Set}$ by setting $\mathcal{F}(a) = \mathfrak{F}_0(F_0^{-1}(a))$, $a \in X_0$, and using push-pull along the span
\begin{equation}
\label{eq:relssSpan}
X_{a \rightarrow b} \times F_0^{-1}(b) \leftarrow \{ q \in Y_1 \mid \partial_1 q \in F_0^{-1}(a), \; \partial_0 q \in F_0^{-1}(b) \} \rightarrow F_0^{-1}(a)
\end{equation}
to define the action map $\Hom_{H(X_{\bullet})}(a,b) \times \mathcal{F}(b) \rightarrow \mathcal{F}(a)$. Writing $\mathbf{1}_{\xi} \in \mathcal{F}(a)$ for the characteristic function of $\xi \in F_0^{-1}(a)$, we have $\mathbf{1}_x \star \mathbf{1}_{\xi^{\prime}} = \sum_{\xi^{\prime \prime}} g_{x, \xi^{\prime}}^{\xi^{\prime \prime}} \mathbf{1}_{\xi^{\prime \prime}}$ where
\[
g_{x, \xi^{\prime}}^{\xi^{\prime \prime}} = \vert
\{ q \in Y_1 \mid \partial_1q = \xi^{\prime \prime}, \; \partial_0 q= \xi^{\prime}, \; F_1(q) =x\} \vert.
\]

Before categorifying the above construction we recall some preliminary definitions from \cite[\S 3.5.B]{dyckerhoff2012b}. A category $\mathcal{C}$ is called $\sqcup$-semisimple if it is equivalent to an overcategory $\mathsf{Set}_{\slash B}$, in which case a simple object of $\mathcal{C}$ is an object which is isomorphic to a point $\{b\} \in \mathsf{Set}_{\slash B}$. Denote by $\| \mathcal{C}\|$ the set of isomorphism classes of simple objects of $\mathcal{C}$. A functor between $\sqcup$-semisimple categories is called additive (resp. simple) if it preserves coproducts (resp. simple objects). The $\sqcup$-semisimple categories, their additive functors and their natural transformations form a bicategory $\mathsf{Cat}^{\sqcup}$. Let $\mathsf{Cat}^{\sqcup !} \subset \mathsf{Cat}^{\sqcup}$ be the subbicategory whose morphisms are instead simple additive functors.

Define a $2$-functor by $A : \mathsf{Span}(\mathsf{Set}) \rightarrow \mathsf{Cat}^{\sqcup}$ by
\[
Z \mapsto \mathsf{Set}_{\slash Z}, \qquad (Z \xleftarrow[]{s} W \xrightarrow[]{p} Z^{\prime}) \mapsto (\mathsf{Set}_{\slash Z} \xrightarrow[]{p_* \circ s^*} \mathsf{Set}_{\slash Z^{\prime}})
\]
with the obvious action on $2$-morphisms. Define also $B : \mathsf{Cat}^{\sqcup} \rightarrow \mathsf{Span}(\mathsf{Set})$ by
\[
\mathcal{C} \mapsto \| \mathcal{C} \|, \qquad (\mathcal{C} \xrightarrow[]{\phi} \mathcal{C}^{\prime}) \mapsto (\| \mathcal{C} \| \xleftarrow[]{} \bigsqcup_{a \in \| \mathcal{C}\|} F^{\prime}(\phi(a)) \xrightarrow[]{} \| \mathcal{C}^{\prime} \| )
\]
where $F^{\prime} : \mathcal{C}^{\prime} \rightarrow \mathsf{Set}_{\slash Z^{\prime}}$ is some equivalence.

\begin{Prop}[{\cite[Proposition 3.5.4]{dyckerhoff2012b}}]
\label{prop:ssCatEquiv}
The $2$-functors $A: \mathsf{Span}(\mathsf{Set}) \leftrightarrow \mathsf{Cat}^{\sqcup}: B$ are mutually inverse $2$-equivalences which restrict to $2$-equivalences $\mathsf{Set} \leftrightarrow \mathsf{Cat}^{\sqcup !}$.
\end{Prop}

For later purposes, fix natural equivalences
\[
\kappa : \mathbf{1}_{\mathsf{Cat}^{\sqcup}} \Rightarrow A \circ B, \qquad \lambda: B \circ A \Rightarrow \mathbf{1}_{\mathsf{Span}(\mathsf{Set})}
\]
realizing the $2$-equivalences of Proposition \ref{prop:ssCatEquiv}.

A semibicategory is the data of a bicategory with all mention of unit $1$-morphisms omitted. A semibicategory is called $\sqcup$-semisimple if it is small, its morphism categories are $\sqcup$-semisimple and composition of morphisms is additive in each variable. A lax $2$-functor between $\sqcup$-semisimple semibicategories is called admissible if the associated functors on morphism categories are simple additive.

\begin{Def}
Two admissible lax $2$-functors $\phi, \phi^{\prime} : \mathbb{X} \rightarrow \mathbb{X}^{\prime}$ of $\sqcup$-semisimple semibicategories are called equivalent if they agree on objects and there exists a collection of natural isomorphisms
\[
\begin{tikzcd}[column sep=4em]
 \Hom_{\mathbb{X}} (a,b)
  \arrow[bend left=20]{r}[name=U,below]{}{\phi_{a,b}} 
  \arrow[bend right=20]{r}[name=D]{}[swap]{\phi^{\prime}_{a,b}}
& 
\Hom_{\mathbb{X}^{\prime}}(\phi(a),\phi(b)),
     \arrow[shorten <=3pt,shorten >=3pt,Rightarrow,to path={(U) -- node[label=left:\footnotesize \scriptsize $U_{a,b}$] {} (D)}]{}
\end{tikzcd}
\qquad a,b \in \text{\textnormal{Ob}}(\mathbb{X})
\]
which commute with the coherence maps for $\phi, \phi^{\prime}$ and the associativity isomorphisms.
\end{Def}

If $\mathbb{X}$, $\mathbb{X}^{\prime}$ are bicategories, then $\phi, \phi^{\prime}: \mathbb{X} \rightarrow \mathbb{X}^{\prime}$ are equivalent if there exists an invertible icon $U: \phi \Rightarrow \phi^{\prime}$. Slightly abusively, we will refer to $\{U_{a,b}\}_{a,b \in \text{Ob}(\mathbb{X})}$ as an icon even if $\mathbb{X}$, $\mathbb{X}^{\prime}$ are not bicategories. The $\sqcup$-semisimple semibicategories and their equivalence classes of admissible lax $2$-functors form a category $\frac{1}{2}\mathsf{biCat}^{\sqcup}$.

Let $X_{\bullet}$ be a $2$-Segal semi-simplicial set. Its Hall $2$-category $\mathbb{H}(X_{\bullet})$ is the $\sqcup$-semisimple semibicategory with
\[
\text{Ob}(\mathbb{H}(X_{\bullet})) = X_0, \qquad \Hom_{\mathbb{H}(X_{\bullet})}(a,b) = \mathsf{Set}_{\slash  X_{a \rightarrow b}}.
\]
Push-pull along the span \eqref{eq:ssSpan} defines composition of $1$-morphisms. Associated to a morphism $\phi_{\bullet}: X_{\bullet} \rightarrow X_{\bullet}^{\prime}$ is the admissible lax $2$-functor $\phi: \mathbb{H}(X_{\bullet}) \rightarrow \mathbb{H}(X_{\bullet}^{\prime})$ given by $\phi_0$ on objects and by pushforward along $\phi_1$ on morphisms.

\begin{Thm}[{\cite[Theorem 3.5.8]{dyckerhoff2012b}}]
\label{thm:biCat}
The assignment $X_{\bullet} \mapsto \mathbb{H}(X_{\bullet})$ extends to an equivalence between the category of $2$-Segal semi-simplicial sets and $\frac{1}{2}\mathsf{biCat}^{\sqcup}$.
\end{Thm}

\begin{proof}
A proof is outlined in \cite{dyckerhoff2012b}. We fill in some details which will be needed below. We first construct a functor $\Xi : \frac{1}{2} \mathsf{biCat}^{\sqcup} \rightarrow \frac{1}{2}\mu \mathsf{Cat}$. Define $\mathfrak{X}= \Xi(\mathbb{X}) \in \frac{1}{2}\mu \mathsf{Cat}$ by
\[
\mathfrak{X}_0 = \text{Ob}(\mathbb{X}), \qquad \mathfrak{X}_1 = \bigsqcup_{a,b \in \text{Ob}(\mathbb{X})} \| \Hom_{\mathbb{X}}(a,b) \|
\]
with span $m_{\mathfrak{X}}$ obtained by applying Proposition \ref{prop:ssCatEquiv} to the functors which define composition of $1$-morphisms in $\mathbb{X}$. Given $\phi: \mathbb{X} \rightarrow \mathbb{X}^{\prime}$, define $\varphi=\Xi(\phi)$ so that $\varphi_0$ is equal to $\phi$ on objects and $\varphi_1$ is given by the functors $\phi_{a,b}$. If $\phi, \phi^{\prime}: \mathbb{X} \rightarrow \mathbb{X}^{\prime}$ are equivalent, say via an icon $U: \phi \Rightarrow \phi^{\prime}$, then $\varphi_0 = \varphi^{\prime}_0$ and $\varphi_1 = \varphi^{\prime}_1$ due to the isomorphisms $
U_{a,b}(f) : \phi_{a,b}(f) \xrightarrow[]{\sim} \phi^{\prime}_{a,b}(f)$, $f \in \Hom_{\mathbb{X}}(a,b)$.

Define also a functor $\Omega: \frac{1}{2} \mu\mathsf{Cat} \rightarrow \frac{1}{2}\mathsf{biCat}^{\sqcup}$ by setting $\mathbb{X} = \Omega(\mathfrak{X})$, where $\text{Ob}(\mathbb{X}) = \mathfrak{X}_0$ and $\Hom_{\mathbb{X}}(a,b) = \mathsf{Set}_{\slash \mathfrak{X}_{a \rightarrow b}}$. Given $\varphi: \mathfrak{X} \rightarrow \mathfrak{X}^{\prime}$, let $\phi=\Omega(\varphi)$ be the morphism which is equal to $\varphi_0$ on objects and whose component functor $\phi_{a,b}$ is pushforward along $\varphi_1: \mathfrak{X}_{a \rightarrow b} \rightarrow \mathfrak{X}^{\prime}_{\varphi_0(a) \rightarrow \varphi_0(b)}$.

We claim that $\Xi$ and $\Omega$ are mutually inverse equivalences. Let $\epsilon : \Xi \circ \Omega \Rightarrow \mathbf{1}_{\frac{1}{2} \mu \mathsf{Cat}}$ be the natural isomorphism whose component $\epsilon_{\mathfrak{X}}: \Xi( \Omega(\mathfrak{X})) \rightarrow \mathfrak{X}$ has $(\epsilon_{\mathfrak{X}})_0$ equal to the identity and has $(\epsilon_{\mathfrak{X}})_1$ equal to the map
\[
\bigsqcup_{a,b \in \mathfrak{X}_0} \lambda_{\mathfrak{X}_{a \rightarrow b}} :
\bigsqcup_{a,b \in \mathfrak{X}_0}  \| \mathsf{Set}_{\slash \mathfrak{X}_{a \rightarrow b} } \|  \rightarrow \mathfrak{X}_1.
\]
Similarly, let $\eta : \mathbf{1}_{\frac{1}{2}\mathsf{biCat}^{\sqcup}} \Rightarrow \Omega \circ \Xi$ be the natural transformation whose component $\eta_{\mathbb{X}}: \mathbb{X} \rightarrow \Omega(\Xi(\mathbb{X}))$ is the identity on objects and is equal to the functor $\kappa_{\Hom_{\mathbb{X}}(a,b)}$ on morphism categories. Given $\phi: \mathbb{X} \rightarrow \mathbb{X}^{\prime}$, we need to check that the diagram 
\[
\begin{tikzcd}
\mathbb{X} \arrow{d}[left]{\phi} \arrow{r}{\eta_{\mathbb{X}}} & \Omega(\Xi(\mathbb{X})) \arrow{d}[right]{\Omega(\Xi(\mathbb{\phi}))} \\ \mathbb{X}^{\prime} \arrow{r}[below]{\eta_{\mathbb{X}^{\prime}}} & \Omega(\Xi(\mathbb{X}^{\prime}))
\end{tikzcd}
\]
commutes in $\frac{1}{2}\mathsf{biCat}^{\sqcup}$, which amounts to giving an invertible icon $U: \eta_{\mathbb{X}^{\prime}} \circ \phi \Rightarrow \Omega(\Xi(\phi)) \circ \eta_{\mathbb{X}}$. The required natural transformation $U_{a,b}$ is defined to be
\[
\begin{tikzcd}[column sep=7em]
\Hom_{\mathbb{X}}(a,b) \arrow{d}[left]{\phi_{a,b}} \arrow{r}{\kappa_{\Hom_{\mathbb{X}}(a,b)}} & \mathsf{Set}_{\slash \| \Hom_{\mathbb{X}}(a, b)\|} \arrow{d}[right]{\| \phi_{a,b} \|_*} \\ \Hom_{\mathbb{X}^{\prime}}(\phi(a),\phi(b)) \arrow{r}[below]{\kappa_{\Hom_{\mathbb{X}^{\prime}}(\phi(a), \phi(b))}} \arrow[Rightarrow]{ru}[above]{\kappa_{\phi_{a,b}}} & \mathsf{Set}_{\slash \| \Hom_{\mathbb{X}^{\prime}}(\phi(a), \phi(b)) \|}.
\end{tikzcd}
\]

This establishes an equivalence $\frac{1}{2} \mathsf{biCat}^{\sqcup} \simeq \frac{1}{2} \mu \mathsf{Cat}$. Using the semi-simplicial variant of Theorem \ref{thm:multiCat} we get an equivalence from $\frac{1}{2} \mathsf{biCat}^{\sqcup}$ to the category of $2$-Segal semi-simplicial sets which is a quasi-inverse of the functor $X_{\bullet} \mapsto \mathbb{H}(X_{\bullet})$.
\end{proof}

Motivated by Proposition \ref{prop:relNerve}, the goal of the remainder of this section is to interpret relative $2$-Segal semi-simplicial sets in terms of a certain class of presheaves. To this end, define a $\mathsf{Cat}^{\sqcup}$-valued presheaf on a $\sqcup$-semisimple semibicategory to be a $2$-functor $\mathbb{F}: \mathbb{X}^{\mathsf{op}} \rightarrow \mathsf{Cat}^{\sqcup}$ with $\mathbb{X} \in \frac{1}{2}\mathsf{biCat}^{\sqcup}$. Here $\mathbb{X}^{\mathsf{op}}$ is the semibicategory obtained from $\mathbb{X}$ by reversing only its $1$-cells. A morphism
\[
(\phi, \theta): (\mathbb{F}: \mathbb{X}^{\mathsf{op}} \rightarrow \mathsf{Cat}^{\sqcup}) \rightarrow (\mathbb{F}^{\prime}: \mathbb{X}^{\prime \mathsf{op}} \rightarrow \mathsf{Cat}^{\sqcup})
\]
consists of a lax $2$-functor $\phi: \mathbb{X} \rightarrow \mathbb{X}^{\prime}$ and an oplax natural transformation $\theta : \mathbb{F} \Rightarrow \mathbb{F}^{\prime} \circ \phi^{\mathsf{op}}$. The pair $(\phi, \theta)$ is called admissible if $\phi$ is admissible and the functors $\theta_a : \mathbb{F}(a) \rightarrow \mathbb{F}^{\prime}(\phi(a))$ are simple additive. The composition $\mathbb{F} \xrightarrow[]{(\phi, \theta)} \mathbb{F}^{\prime} \xrightarrow[]{(\xi, \zeta)} \mathbb{F}^{\prime \prime}$ is defined by the right whiskering of $\zeta$ and $\phi^{\mathsf{op}}$: $(\xi, \zeta) \circ (\phi, \theta) = (\xi \circ \phi, (\zeta \phi^{\mathsf{op}}) \circ \theta)$.

\begin{Def}
Two admissible morphisms
\[
(\phi, \theta), (\phi^{\prime}, \theta^{\prime}): (\mathbb{F}: \mathbb{X}^{\mathsf{op}} \rightarrow \mathsf{Cat}^{\sqcup}) \rightarrow (\mathbb{F}^{\prime}: \mathbb{X}^{\prime \mathsf{op}} \rightarrow \mathsf{Cat}^{\sqcup})
\]
are called equivalent if there exists an invertible icon $U : \phi \Rightarrow \phi^{\prime}$ and an invertible modification
\[
\begin{tikzcd}[column sep=4.5cm]
\mathbb{X}^{\mathsf{op}} 
  \arrow[bend left=25]{r}[name=U,below]{}{\mathbb{F}} 
  \arrow[bend right=25]{r}[name=D]{}[swap]{\mathbb{F}^{\prime} \circ \phi^{\prime \mathsf{op}}}
& 
\mathsf{Cat}^{\sqcup} .
  \arrow[Rightarrow,to path={(U) to[out=-150,in=150] node[left] {$\scriptstyle (\mathbb{F}^{\prime} U^{\mathsf{op}}) \circ \theta$} coordinate (M) (D)}]{}
  \arrow[Rightarrow,to path={(U) to[out=-30,in=30] node[right] {$\scriptstyle \theta^{\prime} $} coordinate (N)  (D)}]{}
  \arrow[to path={node[label={center:\scalebox{2}[0.75]{$\Rrightarrow$}},label={[yshift=-2pt]above:$\scriptstyle \Gamma$}] at ( $ (M)!0.55!(N) $ ) {}}]{}
\end{tikzcd}
\]
\end{Def}

Denote by $\mathsf{Psh}^{\sqcup}$ the category of $\mathsf{Cat}^{\sqcup}$-valued presheaves on $\sqcup$-semisimple semibicategories and their equivalence classes of admissible morphisms.

We now construct a functor from the category of relative $2$-Segal semi-simplicial sets to $\mathsf{Psh}^{\sqcup}$. Given $F_{\bullet}: Y_{\bullet} \rightarrow X_{\bullet}$, define $\mathbb{F}:\mathbb{H}(X_{\bullet})^{\mathsf{op}} \rightarrow \mathsf{Cat}^{\sqcup}$ as follows. Put $\mathbb{F}(a) = \mathsf{Set}_{\slash F_0^{-1}(a)}$, $a \in X_0$, and use the span \eqref{eq:relssSpan} to define a functor
\[
\mathbb{F}_{a,b}: \mathsf{Set}_{\slash  X_{a \rightarrow b}} \rightarrow [\mathsf{Set}_{\slash F_0^{-1}(b)}, \mathsf{Set}_{\slash F_0^{-1}(a)}],
\]
the image of which consists of additive functors by Proposition \ref{prop:ssCatEquiv}. Here and in what follows we use the equivalence $[\mathcal{C} \times \mathcal{D}, \mathcal{E}] \simeq [\mathcal{C}, [\mathcal{D}, \mathcal{E}]]$ for categories $\mathcal{C}, \mathcal{D}$ and $\mathcal{E}$, the first two of which are small. The relative $2$-Segal bijections induce the coherence isomorphisms for $\mathbb{F}$. Given a morphism
\[
\begin{tikzcd}
Y_{\bullet} \arrow{r}[above]{\theta_{\bullet}} \arrow{d}[left]{F_{\bullet}} & Y_{\bullet}^{\prime} \arrow{d}[right]{F_{\bullet}^{\prime}}  \\
X_{\bullet} \arrow{r}[below]{\phi_{\bullet}} & X_{\bullet}^{\prime}
\end{tikzcd}
\]
of relative $2$-Segal semi-simplicial sets, we need to define a morphism 
\[
(\phi, \theta): (\mathbb{F}: \mathbb{H}(X_{\bullet})^{\mathsf{op}} \rightarrow \mathsf{Cat}^{\sqcup}) \rightarrow (\mathbb{F}^{\prime}: \mathbb{H}(X_{\bullet}^{\prime})^{\mathsf{op}} \rightarrow \mathsf{Cat}^{\sqcup}).
\]
Let $\phi$ be the lax $2$-functor associated to $\phi_{\bullet}$. Let $\theta_a: \mathbb{F}(a) \rightarrow \mathbb{F}^{\prime}(\phi(a))$ be the pushforward along $\theta_0 \vert_{F_0^{-1}(a)} : F_0^{-1}(a) \rightarrow F_0^{\prime -1}(\phi(a))$. The natural transformation
\[
\begin{tikzpicture}[baseline= (a).]
\node[scale=1] (a) at (0,0){
\begin{tikzcd}[row sep=2em,column sep=4em]
\mathsf{Set}_{\slash X_{a \rightarrow b}} \arrow{d}[left]{\mathbb{F}_{a,b}} \arrow{r}{(\mathbb{F}^{\prime} \circ \phi^{\mathsf{op}})_{a,b}} & \left[ \mathsf{Set}_{\slash F_0^{\prime-1}(\phi(b))}, \mathsf{Set}_{\slash F_0^{\prime-1}(\phi(a))} \right] \arrow{d}[right]{- \circ \theta_b} \\
\left[ \mathsf{Set}_{\slash F_0^{-1}(b)}, \mathsf{Set}_{\slash F_0^{-1}(a)} \right] \arrow{r}[below]{\theta_a \circ -} \arrow[Rightarrow]{ru}[above]{\theta_{a,b}} & \left[ \mathsf{Set}_{\slash F_0^{-1}(b)}, \mathsf{Set}_{\slash F_0^{\prime-1}(\phi(a))} \right]
\end{tikzcd}
};
\end{tikzpicture}
\]
is induced by the morphism of spans determined by the restriction of the map $\theta_1: Y_1 \rightarrow Y_1^{\prime}$ to the middle set of the span \eqref{eq:relssSpan}.

\begin{Thm}
\label{thm:psh2Cat}
The assignment $F_{\bullet} \mapsto \mathbb{F}$ extends to an equivalence between the category of relative $2$-Segal semi-simplicial sets and $\mathsf{Psh}^{\sqcup}$.
\end{Thm}

\begin{proof}
Lift $\Xi: \frac{1}{2}\mathsf{biCat}^{\sqcup} \rightarrow \frac{1}{2}\mu\mathsf{Cat}$ to a functor $\Xi^{\mathsf{rel}}: \mathsf{Psh}^{\sqcup} \rightarrow \frac{1}{2}\mu\mathsf{Cat} \mhyphen \mathsf{mod}$ as follows. Define $\Xi^{\mathsf{rel}}(\mathbb{F}: \mathbb{X}^{\mathsf{op}} \rightarrow \mathsf{Cat}^{\sqcup}) = (\mathfrak{F}: \mathfrak{Y} \rightarrow \mathfrak{X})$ by $\mathfrak{X} = \Xi(\mathbb{X})$ and $\mathfrak{Y}_0 = \bigsqcup_{a \in \text{Ob}(\mathbb{X})} \| \mathbb{F}(a) \|$ with the canonical map $\mathfrak{F}_0: \mathfrak{Y}_0 \rightarrow \mathfrak{X}_0$ . The span $\mu_{\mathfrak{Y}}$ is obtained by applying Proposition \ref{prop:ssCatEquiv} to the functors $\mathbb{F}_{a,b}$ while the associator $\alpha_{\mathfrak{Y}}$ is defined using the coherence isomorphisms for $\mathbb{F}$. Given a morphism $(\phi, \theta): \mathbb{F} \rightarrow \mathbb{F}^{\prime}$, set $\varphi = \Xi(\phi)$, let $\vartheta_0 : \mathfrak{Y}_0 \rightarrow \mathfrak{Y}_0^{\prime}$ be the map determined by the functors $\theta_a$ and let $\tilde{\vartheta}_1: \vartheta_0 \circ \mu_{\mathfrak{Y}} \rightarrow \mu_{\mathfrak{Y}^{\prime}} \circ (\varphi_1 \times_{\varphi_0} \vartheta_0)$ be the morphism obtained by applying Proposition \ref{prop:ssCatEquiv} to the coherence natural transformations
\[
\begin{tikzpicture}[baseline= (a).]
\node[scale=1] (a) at (0,0){
\begin{tikzcd}[row sep=2.5em,column sep=5em]
\Hom_{\mathbb{X}}(a,b) \times \mathbb{F}(b) \arrow{d}[left]{\mathbb{F}_{a,b}} \arrow{r}{\phi_{a,b} \times \theta_b} & \Hom_{\mathbb{X}^{\prime}}(\phi(a),\phi(b)) \times \mathbb{F}^{\prime}(\phi(b)) \arrow{d}[right]{\mathbb{F}^{\prime}_{a,b}} \\
\mathbb{F}(a) \arrow{r}[below]{\theta_a} \arrow[Rightarrow]{ru}[above]{\theta_{a,b}} & \mathbb{F}^{\prime}(\phi(a)).
\end{tikzcd}
};
\end{tikzpicture}
\]
Suppose that $(\phi,\theta)$ and $(\phi^{\prime}, \theta^{\prime})$ are equivalent morphisms, say via an icon $U$ and a modification $\Gamma$. Then $\varphi = \varphi^{\prime}$. Since $U_a: \phi(a) \rightarrow \phi^{\prime}(a)$ is the identity we have $((\mathbb{F}^{\prime} U^{\mathsf{op}}) \circ \theta)_a = \theta_a$ and hence a natural isomorphism $\Gamma_a: \theta_a \Rightarrow \theta_a^{\prime}$, showing that $\vartheta_0 = \vartheta_0^{\prime}$. The equality $\tilde{\vartheta}_1 = \tilde{\vartheta}^{\prime}_1$ follows from the fact that the components $\{\Gamma_a\}_{a \in \text{Ob}(\mathbb{X})}$ commute with morphisms in $\mathbb{X}$.

Similarly, construct a lift $\Omega^{\mathsf{rel}}: \frac{1}{2}\mu\mathsf{Cat} \mhyphen \mathsf{mod} \rightarrow \mathsf{Psh}^{\sqcup}$ of $\Omega: \frac{1}{2}\mu\mathsf{Cat}\rightarrow \frac{1}{2}\mathsf{biCat}^{\sqcup}$. Define $\Omega^{\mathsf{rel}}(\mathfrak{F} : \mathfrak{Y} \rightarrow \mathfrak{X}) = (\mathbb{F}: \mathbb{X}^{\mathsf{op}} \rightarrow \mathsf{Cat}^{\sqcup})$ by $\mathbb{X} = \Omega(\mathfrak{X})$ with $\mathbb{F}(a) = \mathsf{Set}_{\slash \mathfrak{F}_0^{-1}(a)}$ and functors $\mathbb{F}_{a,b}$ determined by $\mu_{\mathfrak{Y}}$. Given a morphism $(\varphi, \vartheta) : \mathfrak{F} \rightarrow \mathfrak{F}^{\prime}$, put $\phi=\Omega(\varphi)$ and let $\theta$ be the oplax natural transformation whose component functor $\theta_a : \mathbb{F}(a) \rightarrow \mathbb{F}^{\prime}(\phi(a))$ is pushforward along the map $\vartheta_{0 \vert \mathfrak{F}_0^{-1}(a)} :
\mathfrak{F}_0^{-1}(a) \rightarrow \mathfrak{F}_0^{\prime -1}(\phi(a))$ and whose coherence natural transformations $\theta_{a,b}$ are induced by the morphism $\tilde{\vartheta}_1$.

To prove that $\Xi^{\mathsf{rel}}$ and $\Omega^{\mathsf{rel}}$ are inverse equivalences, we first lift $\epsilon : \Xi \circ \Omega \Rightarrow \mathbf{1}_{\frac{1}{2}\mu \mathsf{Cat}}$ to $\epsilon^{\mathsf{rel}} : \Xi^{\mathsf{rel}} \circ \Omega^{\mathsf{rel}} \Rightarrow \mathbf{1}_{\frac{1}{2} \mu \mathsf{Cat} \mhyphen \mathsf{mod}}$. On the total space $\epsilon^{\mathsf{rel}}_{\mathfrak{F}}: \Xi^{\mathsf{rel}}( \Omega^{\mathsf{rel}}(\mathfrak{F})) \rightarrow \mathfrak{F}$ is given by $\bigsqcup_{a \in \mathfrak{X}_0} \lambda_{\mathfrak{F}_0^{-1}(a)}$ on objects while the morphism $\vartheta_0 \circ \mu_{\Xi^{\mathsf{rel}}( \Omega^{\mathsf{rel}}(\mathfrak{Y}))} \rightarrow \mu_{\mathfrak{Y}} \circ (\varphi_1 \times_{\varphi_0} \vartheta_0)$ is $\lambda_{\mu_{\mathfrak{Y}}}$. Similarly, lift $\eta: \mathbf{1}_{\frac{1}{2} \mathsf{biCat}^{\sqcup}} \Rightarrow \Omega \circ \Xi$ to $\eta^{\mathsf{rel}} : \mathbf{1}_{\mathsf{Psh}^{\sqcup}} \Rightarrow \Omega^{\mathsf{rel}} \circ \Xi^{\mathsf{rel}}$ as follows. Writing $\eta^{\mathsf{rel}}_{\mathbb{F}}=(\phi_{\mathbb{F}}, \theta_{\mathbb{F}})$, put $\phi_{\mathbb{F}}=\eta_{\mathbb{X}}$. Since $\eta_{\mathbb{X}}$ is the identity on objects, we can define $\theta_{\mathbb{F}}: \mathbb{F} \rightarrow \Omega^{\mathsf{rel}} (\Xi^{\mathsf{rel}}(\mathbb{F})) \circ \eta_{\mathbb{X}}^{\mathsf{op}}$ so that its component functor $(\theta_{\mathbb{F}})_a$ is $\lambda_{\mathbb{F}(a)}$. Associated to the functor $\mathbb{F}_{a,b}$ is the natural transformation
\[
\begin{tikzpicture}[baseline= (a).]
\node[scale=1] (a) at (0,0){
\begin{tikzcd}[row sep=2.5em,column sep=7.5em]
\Hom_{\mathbb{X}}(a,b) \times \mathbb{F}(b) \arrow{d}[left]{\mathbb{F}_{a,b}} \arrow{r}{\lambda_{\Hom_{\mathbb{X}}(a,b) \times \mathbb{F}(b)}} & \mathsf{Set}_{\slash \|\Hom_{\mathbb{X}}(a,b) \times \mathbb{F}(b) \|} \arrow{d}[right]{\|\mathbb{F}_{a,b}\|_*} \\
\mathbb{F}(a) \arrow{r}[below]{\lambda_{\mathbb{F}(a)}} \arrow[Rightarrow]{ru}[above]{\lambda_{\mathbb{F}_{a,b}}} & \mathsf{Set}_{\slash{\| \mathbb{F}(a) \|}}
\end{tikzcd}
};
\end{tikzpicture}
\]
which we take to be the coherence data $(\theta_{\mathbb{F}})_{a,b}$. It remains to check that, for an admissible morphism $(\phi, \theta): \mathbb{F} \rightarrow \mathbb{F}^{\prime}$, the diagram 
\[
\begin{tikzcd}
\mathbb{F} \arrow{d}[left]{(\phi, \theta)} \arrow{r}{\eta^{\mathsf{rel}}_{\mathbb{F}}} & \Omega^{\mathsf{rel}}(\Xi^{\mathsf{rel}}(\mathbb{F})) \arrow{d}[right]{\Omega^{\mathsf{rel}}(\Xi^{\mathsf{rel}}(\phi, \theta))} \\ \mathbb{F}^{\prime} \arrow{r}[below]{\eta^{\mathsf{rel}}_{\mathbb{F}^{\prime}}} & \Omega^{\mathsf{rel}}(\Xi^{\mathsf{rel}}(\mathbb{F}^{\prime}))
\end{tikzcd}
\]
commutes in $\mathsf{Psh}^{\sqcup}$. That is, we need an invertible icon $U: \phi_{\mathbb{F}^{\prime}} \circ \phi \Rightarrow \eta_{\mathbb{X}}(\phi) \circ \phi_{\mathbb{F}}$, which exists by Theorem \ref{thm:biCat}, and an invertible modification
\[
\begin{tikzcd}[column sep=9cm]
\mathbb{X}^{\mathsf{op}} 
  \arrow[bend left=20]{r}[name=U,below]{}{\mathbb{F}} 
  \arrow[bend right=20]{r}[name=D]{}[swap]{\Omega^{\mathsf{rel}}(\Xi^{\mathsf{rel}}(\mathbb{F}^{\prime})) \circ (\eta_{\mathbb{X}^{\prime}} \circ \phi)^{\mathsf{op}}}
& 
\mathsf{Cat}^{\sqcup}.
  \arrow[Rightarrow,to path={(U) to[out=-150,in=150] node[left] {$\scriptstyle (\Omega^{\mathsf{rel}}(\Xi^{\mathsf{rel}}(\mathbb{F}^{\prime})) U^{\mathsf{op}})  \circ (\theta_{\mathbb{F}^{\prime}} \circ \theta)$} coordinate (M) (D)}]{}
  \arrow[Rightarrow,to path={(U) to[out=-30,in=30] node[right] {$\scriptstyle \Omega^{\mathsf{rel}} (\Xi^{\mathsf{rel}}(\theta)) \circ \theta_{\mathbb{F}}$} coordinate (N)  (D)}]{}
  \arrow[to path={node[label={center:\scalebox{2.75}[0.75]{$\Rrightarrow$}},label={[yshift=-2pt]above:$\scriptstyle\Gamma$}] at ( $ (M)!0.55!(N) $ ) {}}]{}
\end{tikzcd}
\]
The required natural transformations $\Gamma_a$ are defined to be
\[
\begin{tikzcd}[column sep=6em]
\mathbb{F}(a) \arrow{d}[left]{\theta_a} \arrow{r}{\lambda_{\mathbb{F}(a)}} & \mathsf{Set}_{\slash \| \mathbb{F}(a) \|} \arrow{d}[right]{\| \theta_a \|_*} \\ \mathbb{F}^{\prime}(\phi(a)) \arrow{r}[below]{\lambda_{\mathbb{F}^{\prime}(\phi(a))}} \arrow[Rightarrow]{ru}[above]{\lambda_{\theta_a}} & \mathsf{Set}_{\slash \| \mathbb{F}^{\prime}(\phi(a)) \|}.
\end{tikzcd}
\]
This establishes the equivalence $\mathsf{Psh}^{\sqcup} \simeq \frac{1}{2} \mu \mathsf{Cat}\mhyphen \mathsf{mod}$. Now apply Theorem \ref{thm:multiCatModule}.
\end{proof}

The simplicial variant of Theorem \ref{thm:biCat} is an equivalence between $2 \mhyphen \mathsf{Seg}\mathbb{S}$ and the category of $\sqcup$-semisimple bicategories with simple units. The modification of Theorem \ref{thm:psh2Cat} is a compatible equivalence between $2 \mhyphen \mathsf{SegRel}\mathbb{S}$ and the category of $\mathsf{Cat}^{\sqcup}$-valued presheaves over $\sqcup$-semisimple bicategories with simple units.

\begin{Ex}
In the above language, the semi-simplicial set $\widetilde{\mathbb{T}}_{\bullet}(M)$ of Section \ref{sec:pseudoPoly} defines a $\sqcup$-semisimple semibicategory $\mathbb{X}(M)$ with $\text{Ob}(\mathbb{X}(M)) = M$ and $\Hom_{\mathbb{X}(M)}(a,b) = \mathsf{Set}_{\slash M_{a \rightarrow b}}$, where $M_{a \rightarrow b} \subset C^0([-\infty, \infty],M)$ is the subset of paths from $a$ to $b$, with composition given by counting (suitably interpreted) pseudoholmorphic polygons in $M$ with prescribed boundary conditions. The relative $2$-Segal semi-simplicial morphism $\widetilde{\mathbb{T}}^{\tau}_{\bullet}(M) \rightarrow \widetilde{\mathbb{T}}_{\bullet}(M)$ of Theorem \ref{thm:realPseudoPoly} becomes the presheaf $\mathbb{F}^{\tau} : \mathbb{X}(M)^{\mathsf{op}} \rightarrow \mathsf{Cat}^{\sqcup}$ given by $\mathbb{F}^{\tau}(a) = \mathsf{Set}_{\slash a^{\tau} \sqcup M_{a \rightarrow \tau(a)}^{\tau}}$, where $a^{\tau} =a$ if $a \in M^{\tau}$ and $a^{\tau} = \varnothing$ otherwise and $M^{\tau}_{a \rightarrow \tau(a)} \subset M_{a \rightarrow \tau(a)}$ is the subset of real paths. Counts of real pseudoholomorphic $n$-gons, with $n \leq 4$, obeying one ordinary and two real boundary conditions determine the required functors $\mathbb{F}^{\tau}(b) \rightarrow \mathbb{F}^{\tau}(a)$.
\end{Ex}


\bibliographystyle{plain}
\bibliography{mybib}

\end{document}